\documentclass[12pt]{amsart}

\usepackage{amssymb,amscd,amsthm, verbatim,amsmath,color,fancyhdr, mathrsfs}
\usepackage{graphicx}
\usepackage{tikz}
\usepackage{tikz-cd}
\usetikzlibrary{cd}
\usetikzlibrary{shapes.geometric,positioning} 
\usepackage{appendix}
\usepackage[colorlinks=true,allcolors=black]{hyperref}
\makeatletter \def\l@subsection{\@tocline{2}{0pt}{1pc}{5pc}{}} \def\l@subsection{\@tocline{2}{0pt}{2pc}{6pc}{}} \makeatother

\DeclareRobustCommand{\SkipTocEntry}[5]{}
\usepackage[foot]{amsaddr}
\usepackage[letterpaper, left=2.5cm, right=2.5cm, top=2.5cm,
bottom=2.5cm,dvips]{geometry}

\newtheorem{thm}{Theorem}[section]
\newtheorem{prop}[thm]{Proposition}
\newtheorem{lem}[thm]{Lemma}

\newtheorem{cor}[thm]{Corollary}

\newtheorem{exmp}[thm]{Example}
\newtheorem{defn}[thm]{Definition}
\theoremstyle{remark}
\newtheorem*{remark}{Remark}

\newcommand\textcat[1]{\textnormal{\textbf{#1}}}
\newcommand\texthom[1]{\underline{\textnormal{#1}}}
\title{Koszul Monoids in Quasi-abelian Categories}

\author{Rhiannon Savage}
\address{Rhiannon Savage, The Mathematical Institute, The University of Oxford}
\email{rhiannon.savage@maths.ox.ac.uk}
\date{}

\begin{document}
\maketitle
\vskip.5in
\begin{abstract}Suppose that we have a bicomplete closed symmetric monoidal quasi-abelian category $\mathcal{E}$ with enough flat projectives, such as the category of complete bornological spaces $\textcat{CBorn}_k$ or the category of inductive limits of Banach spaces $\textcat{IndBan}_k$. Working with monoids in $\mathcal{E}$, we can generalise and extend the Koszul duality theory of Beilinson, Ginzburg, Soergel \cite{beilinsonginzburg96}. We use an element-free approach to define the notions of Koszul monoids, and quadratic monoids and their duals. Schneiders'  embedding \cite{schneiders99} of a quasi-abelian category into an abelian category, its left heart, allows us to prove an equivalence of certain subcategories of the derived categories of graded modules over Koszul monoids and their duals.
\end{abstract}
\tableofcontents

\section{Introduction}

\addtocontents{toc}{\SkipTocEntry}\subsection{Koszul Duality}

Suppose that we have a finite dimensional vector space $V$ over a field $k$. Then, there exists the following projective resolution of $k$
\begin{equation}\label{thiscomplexsymmetric}
    \dots\rightarrow{S(V)\otimes_k\bigwedge^2(V)}\rightarrow{S(V)\otimes_k V}\rightarrow{S(V)}\twoheadrightarrow{k}
\end{equation}In 1978, Bernstein, Gelfand, and Gelfand \cite{bernshteingelfand78} showed the following equivalence of bounded derived categories over the categories of graded $S(V)$ and $\bigwedge V^*$ modules.
\begin{equation*}
    \textcat{D}^b(\textcat{gr}S(V)\textcat{-Mod})\simeq{ \textcat{D}^b({\textcat{gr}\bigwedge V^*}\textcat{-Mod})  }
\end{equation*}The seminal paper of Beilinson, Ginzburg, and Soergel \cite{beilinsonginzburg96} extends these constructions to a more general collection of rings known as Koszul rings, of which $S(V)$ and $\bigwedge V$ are toy examples.

\begin{defn}\cite[Definition 1.1.2]{beilinsonginzburg96}\label{koszulringbgs}
A \textbf{Koszul ring} is a positively graded ring $A=\bigoplus_{j\geq 0}A_j$ such that $A_0$ is semisimple and $A_0$, considered as a graded left $A$-module, admits a graded projective resolution 
\begin{equation*}
\dots\rightarrow{P^2}\rightarrow{P^1}\rightarrow{P^0}\twoheadrightarrow{A_0}    
\end{equation*}such that each $P^i$ is generated by its degree $i$ component, i.e. $P^i=AP^i_i$.
\end{defn}These Koszul rings are examples of objects called \textit{quadratic rings} which, in \cite{beilinsonginzburg96}, are positively graded rings $A=\bigoplus_{j\geq 0} A_j$ such that $A_0$ is semisimple, and $A$ is generated over $A_0$ by $A_1$ with relations of degree two. Quadratic rings were originally referred to by Priddy in \cite{priddy70} as \textit{homogeneous pre-Koszul algebras}. Left finite quadratic rings, i.e. those rings $A$ such that each $A_i$ is finitely generated as a left $A_0$-module, admit dual objects $A^!$, called their \textit{quadratic duals}. These rings are dual in the sense that ${}^!(A^!)\simeq A\simeq ({}^!A)^!$.

Quadratic rings produce projective resolutions called \textit{Koszul complexes}. A quadratic ring over a semisimple ring $A_0=k$ is a Koszul ring if and only if its Koszul complex is a resolution of $k$. We note that the symmetric and exterior algebras are examples of Koszul rings. The quadratic dual of $S(V)$ is $\bigwedge V^*$ and the complex in Equation \ref{thiscomplexsymmetric} is the Koszul complex of $S(V)$. We have a slightly weaker equivalence of derived categories for Koszul rings than in the symmetric/exterior algebra case.

\begin{thm}
\cite[Theorem 2.12.1]{beilinsonginzburg96} Suppose that $A$ is a left finite Koszul ring. Then, there exists an equivalence of triangulated categories
\begin{equation*}
    \textcat{D}^\downarrow(\textcat{grA-Mod})\simeq{ \textcat{D}^\uparrow({\textcat{grA}}^!\textcat{-Mod})  }
\end{equation*}where $\textcat{D}^\downarrow(\textcat{grA-Mod})$ and ${ \textcat{D}^\uparrow({\textcat{grA}}^!\textcat{-Mod})  }$ denote certain subcategories of the derived categories of graded $A$ and $A^!$ modules. 
\end{thm}

{Several generalisations of this Koszul duality theory exist. For example, the various types of non-homogeneous, non-quadratic Koszul duality presented by Positselski in \cite{positselski11}.} Further, Koszul duality for wider classes of objects than graded algebras over semisimple rings have also been discussed. In general these use bar-cobar constructions to establish various kinds of dualities between objects such as algebraic operads \cite{ginzburgkapranov94}, D-modules \cite{kapranov91} or dg categories \cite{holsetinlazarev21}. The aim of this paper is to generalise the Koszul duality theory of Beilinson, Ginzburg, and Soergel to certain algebra-like objects in a type of category called a \textit{quasi-abelian category}.

\addtocontents{toc}{\SkipTocEntry}\subsection{Quasi-abelian Categories}

The theory of quasi-abelian categories was developed by Schneiders \cite{schneiders99} with the intention of developing a cohomological theory of sheaves with values in categories such as the category of filtered modules, or the category of locally convex topological vector spaces. Quasi-abelian categories possess enough structure to be able to do homological algebra in them, explained in Sections \ref{qacs} and \ref{haquasiabeliancats}. They are additive categories with kernels and cokernels, where we specify that only certain morphisms are strict. 

One of the most important properties of quasi-abelian categories is summarised in the following result. 

\begin{prop}\cite[c.f. Section 1.2]{schneiders99}
    There exists an abelian category $\mathcal{LH(E)}$, the `left heart' of $\mathcal{E}$, and a fully faithful functor 
    \begin{equation*}
        I:\mathcal{E}\rightarrow{\mathcal{LH(E)}}
    \end{equation*}inducing an equivalence of categories
    \begin{equation*}
        \textcat{D}(I):\textcat{D}(\mathcal{E})\rightarrow{\textcat{D}(\mathcal{LH(E)})}
    \end{equation*}
\end{prop}
This means that certain questions about derived quasi-abelian categories can be answered by considering the analogous theory in the abelian category. In particular it will allow us, in Section \ref{koszuldualityresult}, to discuss converging spectral sequences in the quasi-abelian setting.

\addtocontents{toc}{\SkipTocEntry}\subsection{Koszul Monoids}

For the purposes of this paper we will work within a bicomplete closed symmetric monoidal quasi-abelian category $\mathcal{E}$ with enough flat projectives. We recall the meaning of these words in Sections \ref{closedsymmetricmonoidal} and \ref{haquasiabeliancats}.

The key examples of such categories are the category of complete bornological spaces, $\textcat{CBorn}_k$, and the category of inductive limits of Banach spaces, $\textcat{IndBan}_k$. Recent research by Bambozzi, Ben-Bassat, Kelly, Kremnizer, and Mukherjee (see \cite{bambozzibenbassat15}, \cite{benbassatkremnitzer17}, to name a few) has considered a theory of derived analytic geometry with the main focus being on objects in
these categories. Developing the theory of homological algebra within quasi-abelian categories is extremely vital for this research. 

In the BGS setting \cite{beilinsonginzburg96}, Koszul rings are only defined when the $0$-graded part is semisimple. As well as generalising their theory to monoids in more general quasi-abelian categories, we also extend their theory within the category of rings by working with monoids which are instead \textit{pre-Koszul}, see Definition \ref{pre-koszul}. In Section \ref{koszulmonoids}, we define the Koszul monoid as follows. \begin{defn}
A \textbf{Koszul monoid} in $\mathcal{E}$ is a positively graded pre-Koszul monoid $A$ such that $A_0$, considered as a graded left $A$-module, admits a graded projective strict resolution of $A$-modules.
 \begin{equation*}
     \dots \rightarrow{P^2}\rightarrow{P^1}\rightarrow{P^0}\twoheadrightarrow{A_0}
 \end{equation*}with each $P^i$ generated by its degree $i$ component over $A_0$.
\end{defn}We also prove a sufficient condition for a pre-Koszul monoid $A$ to be Koszul. Namely, 
\begin{prop}
    $A$ is a Koszul monoid if and only if $\textnormal{Ext}^i_{\textcat{grA-Mod}}(A_0,A_0\langle n\rangle)=0$ unless $i=n$.
\end{prop}

{We give several examples of Koszul monoids. In particular, any Koszul ring in the sense of Definition \ref{koszulringbgs} is a Koszul monoid in the category of abelian groups. Moreover, we construct examples of tensor algebras, and symmetric and exterior examples which are Koszul within our quasi-abelian categories. }

This paper is the starting point of a more general exploration of Koszul duality in quasi-abelian categories. In particular, future research will develop Koszul duality for objects with analytic gradings, such as the analytic gradings given for $\textcat{IndBan}_k$ in \cite{kremnitzersmith17}.

\addtocontents{toc}{\SkipTocEntry}\subsection{Quadratic Monoids and their Duals}

One of the main important features of Koszul rings is that they have a dual ring. Rather than defining the Koszul dual of a Koszul monoid to be a certain Ext monoid, which in practice is often hard to write down, we approach the problem of defining a Koszul dual by first defining quadratic monoids, and then considering their quadratic duals. We note that any Koszul monoid is quadratic. In Section \ref{quadraticmonoids}, we make the following definition.
\begin{defn}
$A$ is a \textbf{quadratic monoid} with quadratic data $(A_1,R)$ if it is pre-Koszul, and there exists a strict graded epimorphism 
\begin{equation*}
    \pi:T_{A_0}(A_1)\twoheadrightarrow{A}
\end{equation*}such that there exists a strict epimorphism 
\begin{equation*}
    T_{A_0}(A_1)\otimes_{A_0}R\otimes_{A_0}T_{A_0}(A_1)\twoheadrightarrow{\textnormal{Ker}(\pi)}
\end{equation*}with $R=K_2:=\textnormal{Ker}(A_1\otimes_{A_0} A_1\twoheadrightarrow{A_2})$.
\end{defn}

One motivation for studying quadratic monoids is their potential applications in studying quantum groups. Indeed, Manin proposes in \cite{manin87} that a quadratic ring could be viewed as a `ring of functions on an imaginary space of \textit{noncommutative geometry}, or \textit{quantum space}'. In particular, quadratic monoids in categories such as $\textcat{IndBan}_k$ could provide us with a way of discussing certain types of analytic quantum groups.

There are many different ideas of what a dual object in an arbitrary monoidal category should be. For the purposes of this paper, we want to be able to define dual modules, and indeed we do this concretely in Section \ref{dualquadraticmonoids}. If $A$ is a positively graded monoid in $\mathcal{E}$ and $M$ is a left $A_0$-module, then we define the \textit{left dual $A_0$-module} $M^*$ to be $M^*:=\texthom{Hom}_{\textcat{A}_0\textcat{-Mod}}(M,A_0)$. Moreover, if we impose some extra conditions, making $M$ \textit{left dualisable}, such as reflexivity ${}^*(M^*)\simeq M$ and the existence of well-behaved evaluation and coevaluation morphisms, then our dual modules are easy to work with. In particular, when $A$ is pre-Koszul, we can make the following definition.
\begin{defn}
Suppose that $A$ is a left dualisable quadratic monoid $(A_1,R)$. We say that a positively graded pre-Koszul monoid $A^!$ is the \textbf{left dual quadratic monoid} of $A$ if it is the quadratic monoid with quadratic data $(A_1^*,R^\perp)$, where $R^\perp$ is the kernel of the map $A_1^*\otimes_{A_0} A_1^*\rightarrow{R^*}$.
\end{defn}

\addtocontents{toc}{\SkipTocEntry}\subsection{Koszul Duality}

Suppose that we have a quadratic monoid $A$ with dual $A^!$. In Section \ref{koszulcomplexsection}, we write down an explicit Koszul complex for $A$
\begin{equation*}
    \dots\rightarrow{A\otimes_{A_0} {}^*(A_2^!)}\rightarrow{A\otimes_{A_0} {}^*(A_1^!)}\rightarrow{A}\rightarrow{0}
\end{equation*}We have already noted that any Koszul monoid is quadratic. A weaker converse holds, any quadratic monoid is Koszul if the above complex is a resolution of $A_0$. 

We use the Koszul complex in Section \ref{koszuldualityresult} in our proof of the following theorem.
\begin{thm}
Suppose that $A$ is a left dualisable Koszul monoid with quadratic dual monoid $A^!$. Then, there is an equivalence of triangulated categories
\begin{equation*}
    \textcat{D}^\downarrow(\textcat{grA-Mod})\simeq{ \textcat{D}^\uparrow({\textcat{grA}}^!\textcat{-Mod})  }
\end{equation*}where $\textcat{D}^\downarrow(\textcat{grA-Mod})$ and ${ \textcat{D}^\uparrow({\textcat{grA}}^!\textcat{-Mod})  }$ denote certain subcategories of the derived categories of graded $A$ and $A^!$ modules. 
\end{thm}We have already developed much of the machinery needed to tackle this theorem, and the proof follows in much the same way as in \cite{beilinsonginzburg96}, with some subtle but important changes. In particular, we show that certain functors preserve quasi-isomorphisms by embedding the complexes in question in the left heart of our underlying quasi-abelian category and showing that certain spectral sequences converge there.

\addtocontents{toc}{\SkipTocEntry}\subsection*{Acknowledgements}

The catalyst for this research was my Masters dissertation, supervised by Jack Kelly. I am extremely grateful to him for his support and guidance throughout the entire process of writing this paper. I would also like to thank my PhD supervisor Kobi Kremnizer for his useful comments and his numerous ideas for potential applications of this research. Thanks also to Federico Bambozzi for his comments and corrections.

\addtocontents{toc}{\SkipTocEntry}\subsection*{Funding}

This research was partially conducted during the author's PhD, funded by an EPSRC studentship [EP/W523781/1 - no. 2580843].

\addtocontents{toc}{\SkipTocEntry}\subsection*{Publishing Statement}

This version of the article has been published in the Journal of Applied Categorical Structures, available here \url{https://doi.org/10.1007/s10485-023-09756-7} [10.1007/s10485-023-09756-7].

\section{Closed Symmetric Monoidal Categories}\label{closedsymmetricmonoidal}

We refer to \cite{maclane98} for the definition of a closed symmetric monoidal category. We merely note that a monoidal category consists of enough information to be able to define algebra objects or, as we will call them, monoids.
\subsection{Examples}

We state the following key examples of closed symmetric monoidal categories. Suppose that $k$ is a valued field. See Appendix \ref{banachspaces} for details on the definitions of the categories $\textcat{Ban}_k$,  $\textcat{Ban}_R^A$, and $\textcat{Ban}_R^{nA}$.

\begin{exmp}\label{banachexample}Suppose that $R$ is a Banach ring. The category of Archimedean Banach $R$-modules, $\textcat{Ban}_R^A$, is closed symmetric monoidal \cite[Proposition 3.17]{bambozzibenbassat15}. The monoidal product is the complete Archimedean projective tensor product, denoted by $M\hat{\otimes}_R N$. This is the completion of the $R$-module $M\otimes_R N$ with respect to the semi-norm given by 
\begin{equation*}
    |x|_{M\otimes_R N}=\inf\bigg\{\sum_{i\in I}|m_i|_M|n_i|_N \quad \bigl\lvert\quad x=\sum_{i\in I} m_i\otimes n_i, \, |I|<\infty\bigg\}
\end{equation*} The internal hom functor $\texthom{Hom}_{\textcat{Ban}_R^A}(M,N)$ for $M,N\in \textcat{Ban}_R^A$ is given by the $R$-module $\textnormal{Hom}_{\textcat{Ban}_R^A}(M,N)$ equipped with the semi-norm given by 
\begin{equation*}
    |f|_{\textnormal{sup}}=\sup_{m\in M\backslash\{0\}}\frac{|f(m)|_N}{|m|_M}
\end{equation*} If $R$ is a non-Archimedean Banach ring, we can consider the category $\textcat{Ban}_R^{nA}$ of non-Archimedean $R$-modules. This category is also a closed symmetric monoidal category. The monoidal structure is given by the complete non-Archimedean projective tensor product defined as the completion of $M\otimes_R N$ with respect to the norm given by 
\begin{equation*}
    |x|_{M\otimes_R N}=\inf\{\max_{i\in I}|m_i|_M|n_i|_N \quad \bigl\lvert\quad x=\sum_{i\in I} m_i\otimes n_i, \, |I|<\infty\}
\end{equation*}The internal hom is defined as in the Archimedean case.
\end{exmp}

\begin{prop}\cite[Proposition 3.12]{bambozzibenbassat15}
The category $\textcat{Ban}_R^A$ has all finite limits and colimits. 
\end{prop}

However, $\textcat{Ban}_R^A$ and $\textcat{Ban}_R^{nA}$ do not have infinite limits and colimits, see \cite[Lemma A.26]{benbassatkremnitzer17}, which will make it difficult to work with them in this paper. We will mainly work with categories of bornological and IndBanach spaces, see Appendices \ref{inductiveappendix} and \ref{bornologyappendix}.

\begin{exmp}
Consider the category $\textcat{Ban}_R^A$. The category of Ind-objects of Banach-modules $\textcat{IndBan}_R^A$ is closed symmetric monoidal. The monoidal product is given, for $M,N$ in $\textcat{IndBan}_R^A$, by 
\begin{align*}
    ``\varinjlim_{i\in I}"M_i\otimes ``\varinjlim_{j\in J}"N_j&=\underset{(i,j)\in I\times J}{``\varinjlim"}M_i\otimes N_j
\intertext{and the internal hom is given by} 
    \texthom{Hom}(``\varinjlim_{i\in I}"M_i,``\varinjlim_{j\in J}"N_j)&=\varprojlim_{i\in I}``\varinjlim_{j\in J}"\texthom{Hom}(M_i,N_j)
\end{align*}This definition also holds for $\textcat{IndBan}_R^{nA}$.
\end{exmp}
\begin{remark}
We note that the above construction generalises to Ind-objects in any small, closed, symmetric monoidal quasi-abelian (see Section \ref{qacs}) category by \cite[Proposition 2.1.19]{schneiders99}.
\end{remark}

\begin{defn}\cite[Definition 1.20]{meyer07}
Suppose that $V,W$ are bornological vector spaces over $k$. A set $L$ of linear maps $f:V\rightarrow{W}$ is \textbf{equibounded} if, for each $B$ bounded in $V$, the set
\begin{equation*}
  L(B)=\{f(x)\mid f\in L, x\in B\}
\end{equation*}is bounded in $W$.
\end{defn}
\begin{exmp}
Consider the category $\textcat{Born}_k$ of bornological $k$-vector spaces of convex type. This category is closed symmetric monoidal. Given $V,W\in\textcat{Born}_k$ we can endow $V\otimes_k W$ with the `projective tensor product bornology'. A basis is given by the absolutely convex hulls of all subsets of the form 
\begin{equation*}
    X\otimes Y=\{x\otimes_k y\mid x\in X, y\in Y\}
\end{equation*} with $X$ (resp. Y) an element of a basis of bounded absolutely convex subsets for the bornology on $V$ (resp. $W$). The internal hom is the vector space $\textnormal{Hom}_{\textcat{Born}_k}(V,W)$ equipped with the equiboundedness bornology. The closed symmetric monoidal structure on $\textcat{CBorn}_k$ is defined slightly differently. Given $V,W\in \textcat{CBorn}_k$ we define the monoidal product to be the completion (see \cite[Section 1.5.5]{meyer07}) of $V\otimes_kW$ with respect to the projective tensor product bornology. The internal hom is defined as in $\textcat{Born}_k$.
\end{exmp}

\begin{prop}\cite[Lemmas 3.29 and 3.53]{bambozzibenbassat15} The categories $\textcat{IndBan}^A_R$, $\textcat{IndBan}^{nA}_R$ and $\textcat{CBorn}_k$ have all limits and colimits. 
\end{prop}

\subsection{Monoids}\label{moduleappendix}

Suppose that $(\mathcal{C},\otimes,I)$ is a monoidal category. We can define a monoid, or algebra object, in $\mathcal{C}$. We note that a monoid in the monoidal category of abelian groups $(\textbf{Ab},\otimes_{\mathbb{Z}},\mathbb{Z})$ is a ring. 

\begin{defn}\cite[Section VII.3]{maclane98}\label{monoid}
A  \textbf{(unital associative) monoid} $(A,\mu,\eta)$ in $\mathcal{C}$ consists of:
\begin{itemize}
    \item An object $A$ in $\mathcal{C}$, 
    \item A multiplication morphism $\mu:A\otimes A\rightarrow{A}$, 
    \item A unit morphism $\eta:I\rightarrow{A}$.
\end{itemize} 
satisfying the natural associativity and unit conditions.
\end{defn}

\begin{exmp}
In the monoidal category $\textcat{Ban}_R^A$ (resp. $\textcat{Ban}_R^{nA}$) a monoid object will be called an Archimedean (resp. non-Archimedean) Banach $R$-algebra. It is an object $M\in \textcat{Ban}_R^A$ (resp. $\textcat{Ban}_R^{nA}$) along with a multiplication map $\mu:M\hat{\otimes}_R M\rightarrow{M}$, where the product is equipped with the appropriate norm, and a unit map $\eta:R\rightarrow{M}$. These maps are bounded in the sense of Definition \ref{boundedbanachmap} and satisfy the associativity and unit conditions stated above.
\end{exmp}

\begin{exmp}
A monoid in $\textcat{Born}_k$ will be called a bornological $k$-algebra. It is a $k$-algebra $A$ whose underlying vector space is of convex type, with a bornology such that multiplication $\mu:A\otimes_k A\rightarrow{A}$ is bounded in $V$. A monoid in $\textcat{CBorn}_k$ is a monoid in $\textcat{Born}_k$ whose underlying vector space is complete, and will be called a complete bornological $k$-algebra. Particular examples of bornological algebras come from Fr\'echet algebras and Banach algebras. 
\end{exmp}

\begin{exmp}
We can also define IndBanach algebras. However, an IndBanach algebra is not necessarily an inductive system of Banach algebras. Those that are of this form are called \textit{locally multiplicative}. 
\end{exmp}

\begin{defn}\cite[Section VII.3]{maclane98}
A \textbf{morphism} $f:(A,\mu,\eta)\rightarrow{(A',\mu',\eta')}$ of monoids is a morphism $f:A\rightarrow{A'}$ such that 
\begin{equation*}
 f\circ \mu=\mu'\circ (f\otimes f)\quad \text{and} \quad f\circ\eta=\eta'   
\end{equation*}
\end{defn}

The following definition uses the `left action of a monoid' from \cite[Section VII.4]{maclane98}.

\begin{defn}\label{modulemonoid}
A \textbf{left module} $M$ over a monoid $A$ in $\mathcal{C}$ consists of:
\begin{itemize}
    \item An object $M\in \mathcal{C}$, 
    \item A morphism $\lambda:A\otimes M \rightarrow{M}$
\end{itemize} 
satisfying the natural associativity and unit conditions.
\end{defn}

\begin{defn}
A \textbf{morphism of left $A$-modules} $g:(M,\lambda)\rightarrow{(M',\lambda')}$ is a morphism $g:M\rightarrow{M'}$ such that 
\begin{equation*}
\lambda'\circ(1_A\otimes g)=g\circ \lambda
\end{equation*} 
\end{defn}

\subsection{Graded Monoids}
Analogously to the definition of a graded ring, we can define graded monoids.
\begin{defn} A $(\mathbb{Z})$-\textbf{graded monoid} $A$ in $\mathcal{C}$ consists of:
\begin{itemize}
    \item A family of objects $\{A_i\}_{i\in\mathbb{Z}}$ in $\mathcal{C}$, 
    \item A unit morphism $\eta:I\rightarrow{A_0}$, 
    \item For all $i,j\in\mathbb{Z}$, a multiplication morphism $\mu_{i,j}:A_i\otimes A_j\rightarrow{A_{i+j}}$,
\end{itemize} satisfying the natural associativity and unit conditions. 
\end{defn}

\begin{remark}
If $\mathcal{C}$ has countable direct sums we can write a graded monoid $A$ as
\begin{equation*}
    A=\bigoplus_{i\in \mathbb{Z}}A_i
\end{equation*}
\end{remark}

\begin{defn}
A \textbf{graded morphism} $f:A\rightarrow{A'}$ of graded monoids $A,A'$ consists of a family of morphisms $\{f_i\}_{i\in\mathbb{Z}}$ in $\mathcal{C}$ with $f_i\in\textnormal{Hom}_\mathcal{C}(A_i,A_i')$ satisfying \begin{equation*}
 f_{i+j}\circ \mu_{i,j}=\mu'_{i,j}\circ (f_i\otimes f_j)
\end{equation*} for $i,j\in\mathbb{Z}.$
\end{defn}

\begin{defn}\label{gradedmodule}
A \textbf{left graded $A$-module} $M$ over a graded monoid $A$ in $\mathcal{C}$ consists of:
\begin{itemize}
    \item A family of objects $\{M_i\}_{i\in\mathbb{Z}}$ in $\mathcal{C}$,
    \item For all $i,j\in\mathbb{Z}$, morphisms $\lambda_{i,j}:A_i\otimes M_j\rightarrow{M_{i+j}}$,
\end{itemize}satisfying the natural associativity and unit conditions.
\end{defn}

\begin{exmp}
\begin{itemize}
    \item For a graded monoid $A$, we see that $A$ is a graded $(A_0, A_0)$-bimodule with module actions given by morphisms $\mu_{i,0}:A_i\otimes A_0\rightarrow{A_{i}}$ and $\mu_{0,i}:A_0\otimes A_i\rightarrow{A_i}$. 
    \item Suppose that $M$ is a graded module over a graded monoid $A$. Then, for each $i$, we can consider $M$ as a graded left $A_i$ module with graded module actions given by $\lambda_{i,j}:A_i\otimes M_j\rightarrow{M_{i+j}}$ for each $j$.
\end{itemize}

\end{exmp}

\begin{defn}
A \textbf{graded morphism} $g:M\rightarrow{N}$ of graded left $A$-modules $(M_i,\lambda_i)$ and $(M'_i,\lambda_i')$ consists of a family $\{g_i\}_{i\in\mathbb{Z}}$ of morphisms in $\mathcal{C}$ with $g_i\in \textnormal{Hom}_\mathcal{C}(M_i,M_i')$ satisfying 
\begin{equation*}\label{modulecomm}
    \lambda'_{i,j}\circ (1_{A_i}\otimes g_j)=g_{i+j}\circ \lambda_{i,j}
\end{equation*} for $i,j\in\mathbb{Z}$.
\end{defn}

\subsection{Categories of Modules}
Suppose we have a bicomplete closed symmetric monoidal category $(\mathcal{C},\otimes, I)$. Suppose that $A$ is a monoid in $\mathcal{C}$. We denote the category of left $A$-modules by $\textcat{A-Mod}$.

\begin{defn}
Let $M_1$ be a right $A$-module and $M_2$ be a left $A$-module. Their \textbf{tensor product over $A$}, $M_1\otimes_{A}M_2$, is defined to be the coequaliser of the two maps
\begin{equation*}
M_1\otimes A\otimes M_2\rightrightarrows{M_1\otimes M_2}
\end{equation*} given by the module actions of $A$ on $M_1$ and $M_2$.
\end{defn}

If $M_1$ is a $(B,A)$-bimodule and $M_2$ is an $(A,C)$-bimodule, then $M_1\otimes_{A} M_2$ is naturally a $(B,C)$-bimodule. The actions of $B$ on $M_1$ and $C$ on $M_2$ induce a morphism 
\begin{equation*}
    B\otimes M_1\otimes M_2\otimes C\rightarrow{M_1\otimes M_2}\rightarrow{M_1\otimes_{A}M_2} 
\end{equation*}which coequalizes the two morphisms from $B\otimes M_1\otimes A\otimes M_2\otimes C$ to $B\otimes M_1\otimes M_2\otimes C$, and hence there exists a morphism 
\begin{equation*}
    B\otimes (M_1\otimes_A M_2)\otimes C\rightarrow{M_1\otimes_A M_2}
\end{equation*}which is the bi-module action.
In particular, if $A$ is a commutative monoid and $M_1,M_2$ are left $A$-modules we see that $M_1\otimes_A M_2$ is a left $A$-module.\\

We can define a map $\texthom{Hom}(M_1,M_2)\rightarrow{\texthom{Hom}(A\otimes M_1,M_2)}$ by applying the internal hom functor $\texthom{Hom}(-,M_2)$ to the module action $A\otimes M_1\rightarrow{M_1}$. We can obtain another map $\texthom{Hom}(M_1,M_2)\rightarrow{\texthom{Hom}(A\otimes M_1,M_2)}$ in two stages. First, if we tensor the morphism $id\rightarrow{\texthom{Hom}(A,A)}$ with the identity of $\texthom{Hom}(M_1,M_2)$, we obtain a map
\begin{equation*}
    \texthom{Hom}(M_1,M_2)\rightarrow{\texthom{Hom}(A,A)\otimes\texthom{Hom}(M_1,M_2)}\rightarrow{\texthom{Hom}(A\otimes M_1,A\otimes M_2)}
\end{equation*}We can then use the module action $A\otimes M_2\rightarrow{M_2}$ to obtain a map \begin{equation*}
    \texthom{Hom}(A\otimes M_1,A\otimes M_2)\rightarrow{\texthom{Hom}(A\otimes M_1,M_2)}
\end{equation*}

\begin{defn}
For $M_1$ and $M_2$ left $A$-modules, the \textbf{internal hom} $\texthom{Hom}_{\textcat{A-Mod}}(M_1,M_2)$ is defined to be the equalizer of the two maps
\begin{equation*}
\begin{tikzcd}
    \texthom{Hom}(M_1,M_2) \arrow{dr}\arrow{rr} & & \texthom{{Hom}}( A\otimes M_1,M_2)\\
    & \texthom{Hom}(A\otimes M_1,A\otimes M_2) \arrow{ur}
\end{tikzcd}
\end{equation*}
\end{defn} We remark that if $A$ is a commutative monoid, then $\texthom{Hom}_{\textcat{A-Mod}}(M_1,M_2)$ is a left $A$-module. We collect together the following useful facts.
\begin{lem}\label{usefulfacts}\cite[Lemma 2.4]{benbassatkremnitzer17}
For any left $A$-module $M$ we have the following isomorphisms
\begin{equation*}
    \texthom{Hom}_\textcat{A-Mod}(A,M)\simeq M\quad A\otimes_A M\simeq M
\end{equation*}
\end{lem}
We have the following extensions of the tensor-hom adjunction. Suppose that $A$ and $B$ are two monoids in $\mathcal{C}$. Let $M_1$ be a $(B,A)$-bimodule, $M_2$ be a left $A$-module, and $M_3$ be a left $B$-module.

\begin{thm}\label{tensorhomadjunction}
The functor $M_1\otimes_A-:\textcat{A-Mod}\rightarrow{\textcat{B-Mod}}$ is left adjoint to the functor $\texthom{Hom}_{\textcat{B-Mod}}(M_1,-):\textcat{B-Mod}\rightarrow{\textcat{A-Mod}}$. Hence, there is a natural isomorphism 
\begin{equation*}
    \textnormal{Hom}_{\textcat{B-Mod}}(M_1\otimes_A M_2,M_3)\simeq \textnormal{Hom}_{\textcat{A-Mod}}(M_2,\texthom{Hom}_{\textcat{B-Mod}}(M_1,M_3))
\end{equation*}\end{thm}

\begin{proof}
Follows in a similar way to the proof for rings.
\end{proof}

\begin{cor}\label{tensorhomadjunctioncor}The above isomorphism lifts to an isomorphism 
\begin{equation*}
     \texthom{Hom}_{\textcat{B-Mod}}(M_1\otimes_A M_2,M_3)\simeq \texthom{Hom}_{\textcat{A-Mod}}(M_2,\texthom{Hom}_{\textcat{B-Mod}}(M_1,M_3))
\end{equation*}
 
\end{cor}

\begin{cor}\label{tensorhomcor2}
Suppose that $M$ and $N$ are both left $A,B$ modules. Also, suppose that $B$ is a right $A$-module. Then, there exist isomorphisms
\begin{equation*}
    \textnormal{Hom}_{\textcat{B-Mod}}(B\otimes_A M,N)\simeq \textnormal{Hom}_\textcat{A-Mod}(M,N)\simeq \textnormal{Hom}_{\textcat{B-Mod}}(M,\texthom{Hom}_{\textcat{A-Mod}}(B,N))
\end{equation*}natural in $M$ and $N$.
\end{cor}

\begin{proof}
We note that, by Theorem \ref{tensorhomadjunction}, we have an isomorphism 
\begin{equation*}
    \textnormal{Hom}_{\textcat{B-Mod}}(B\otimes_A M,N)\simeq \textnormal{Hom}_{\textcat{A-Mod}}(M,\texthom{Hom}_{\textcat{B-Mod}}(B,N))\simeq \textnormal{Hom}_\textcat{A-Mod}(M,N)
\end{equation*}Applying Theorem \ref{tensorhomadjunction} again, we obtain an isomorphism 
\begin{equation*}
   \textnormal{Hom}_{\textcat{B-Mod}}(M,\texthom{Hom}_{\textcat{A-Mod}}(B,N))\simeq \textnormal{Hom}_{\textcat{A-Mod}}(B\otimes_B M,N)\simeq \textnormal{Hom}_\textcat{A-Mod}(M,N)
\end{equation*}
\end{proof}

\begin{prop}\label{amodproperties}Suppose that $\mathcal{C}$ is in addition an additive category. If $A$ is a commutative monoid, then $\textcat{A-Mod}$ is a bicomplete closed symmetric monoidal additive category with monoidal product $\otimes_A$ and hom functor $\texthom{Hom}_{\textcat{A-Mod}}$. The unit object is $A$. 
\end{prop}
\begin{proof}
We can extend the proof of \cite[Lemma 2.3]{benbassatkremnitzer17}. If $\mathcal{C}$ has all limits and colimits then so does $\textcat{A-Mod}$ since products, equalisers, coproducts, and coequalizers can all be computed in $\mathcal{C}$.
\end{proof}



Suppose that $A$ is a positively graded monoid in $\mathcal{C}$. We denote by $\textcat{grA-Mod}$ the category of graded $A$-modules equipped with graded morphisms. For objects $M,N\in\textcat{grA-Mod}$, we can define a grading on $M\otimes_A N$ by 
\begin{align*}
    (M\otimes_A N)_n&=\bigoplus_{i+j=n}M_i\otimes_{A}N_j
\intertext{Similarly, we can define a grading on $\texthom{Hom}_{\textcat{A-Mod}}(M,N)$ by}
    \texthom{Hom}_{\textcat{A-Mod}}(M,N)_n&=\bigoplus_{i+j=n}\texthom{Hom}_{\textcat{A-Mod}}(M_i,N_{-j})
\end{align*}We will denote the internal hom of graded modules by $\texthom{Hom}_{\textcat{grA-Mod}}$.
\begin{cor}
With the conditions of Proposition \ref{amodproperties}, the category $\textcat{grA-Mod}$ is a bicomplete closed symmetric monoidal additive category with monoidal product $\otimes_A$ and hom functor $\texthom{Hom}_{\textcat{grA-Mod}}$. The unit object is $A$. 
\end{cor}
\begin{proof}
We note that the tensor product $\otimes_A$ and $\texthom{Hom}_{\textcat{grA-Mod}}$ give $\textcat{grA-Mod}$ the structure of a closed symmetric monoidal category. Since $\textcat{A-Mod}$ is bicomplete, so is $\textcat{grA-Mod}$ as products, equalisers, coproducts, and coequalizers of graded modules can all be given a graded structure. 
\end{proof}

\section{Quasi-abelian Categories}\label{qacs}

\subsection{Quasi-abelian Categories}
Suppose that $\mathcal{E}$ is an additive category with all kernels and cokernels. 
\begin{defn}
A morphism $f:X\rightarrow{Y}$ in $\mathcal{E}$ is \textbf{strict} if the induced morphism 
\begin{equation*}
    \overline{f}:\textnormal{Coim}(f)\rightarrow{\textnormal{Im}(f)}
\end{equation*} is an isomorphism. 
\end{defn}
\begin{defn}
$\mathcal{E}$ is \textbf{abelian} if every morphism is strict.
\end{defn}
\begin{remark}
We note that, for any morphism $f:X\rightarrow{Y}$ in $\mathcal{E}$, the canonical morphism
\begin{equation*}
    \textnormal{Ker}(f)\rightarrow{X}\quad \text{(resp. $Y\rightarrow{\textnormal{Coker}(f)}$)}
\end{equation*}is a strict monomorphism (resp. strict epimorphism). Moreover, a morphism $f$ is strict if and only if it can be decomposed as $f=m\circ e$ where $m$ is a strict monomorphism and $e$ is a strict epimorphism.
\end{remark}

\begin{prop}\label{isomorphismmonoepi}
If a strict map $f:X\rightarrow{Y}$ is both a monomorphism and an epimorphism, then it is an isomorphism.    
\end{prop}
\begin{proof}Indeed, since $f$ is a strict map, $\textnormal{Coim}(f)\simeq\textnormal{Im}(f)$. Moreover, as $f$ is a strict monomorphism, $X\simeq \textnormal{Ker}(Y\rightarrow{\textnormal{Coker}(f))}=\textnormal{Im}(f)$ and, as $f$ is a strict epimorphism, $Y\simeq\textnormal{Coker}(\textnormal{Ker}(f)\rightarrow{X})=\textnormal{Coim}(f)$. Therefore, we see that \begin{equation*}
    X\simeq \textnormal{Im}(f)\simeq \textnormal{Coim}(f)\simeq Y
\end{equation*} 

\end{proof}

Many of the categories we typically meet are abelian, such as the category of modules over a ring. However, we wish to be able to deal with certain categories, such as the categories discussed in Section \ref{closedsymmetricmonoidal}, which have a slightly weaker structure than abelian categories. We introduce the notion of a quasi-abelian category, and state a few results due to Schneiders \cite{schneiders99}.

\begin{defn} \label{quasiabelian} $\mathcal{E}$ is \textbf{quasi-abelian} if it satisfies the following two conditions
\begin{itemize}
    \item[$(\textbf{QA})$] In a pullback square
    \begin{equation*}
    \begin{tikzcd}
    X' \arrow{d}\arrow{r}{f'} & Y'\arrow{d}\\
    X  \arrow[r,"f"] & Y 
    \end{tikzcd}
\end{equation*}
    If $f$ is a strict epimorphism, then $f'$ is a strict epimorphism.
    \item[$(\textbf{QA}^*)$] In a pushout square
    \begin{equation*}
    \begin{tikzcd}
    X \arrow{d}\arrow{r}{f} & Y\arrow{d}\\
    X'  \arrow[r,"f'"'] & Y' 
    \end{tikzcd}
\end{equation*}
    If $f$ is a strict monomorphism, then $f'$ is a strict monomorphism.
\end{itemize}
\end{defn}

\begin{remark}Equivalently, the class of all kernel-cokernel pairs forms a Quillen exact structure on $\mathcal{E}$.
\end{remark}

Fix a quasi-abelian category $\mathcal{E}$. We state the following propositions about strict morphisms in quasi-abelian categories.
\begin{prop}\cite[Proposition 1.1.7]{schneiders99}\label{strictepisstable} The class of strict epimorphisms (resp. monomorphisms) of $\mathcal{E}$ is stable under composition. 
    
\end{prop}\begin{prop}\label{monoepimorphismtriangle}\cite[Proposition 1.1.8]{schneiders99} Let \begin{equation*}
    \begin{tikzcd}
    & Y \arrow{dr}{g}\\
    X \arrow{ur}{f} \arrow{rr}{h}  & & Z
    \end{tikzcd}
\end{equation*}be a commutative diagram in $\mathcal{E}$. If $h$ is a strict epimorphism, then $g$ is a strict epimorphism. Dually, if $h$ is a strict monomorphism, then $f$ is a strict monomorphism.

\end{prop}

\begin{cor}\label{rightinverseepimorphism}
    Any morphism in $\mathcal{E}$ with a right inverse is a strict epimorphism. Similarly, any morphism in $\mathcal{E}$ with a left inverse is a strict monomorphism.
\end{cor}

\begin{proof}
Indeed, suppose that $g:Y\rightarrow{X}$ is a morphism in $\mathcal{E}$ with a right inverse $f:X\rightarrow{Y}$. Then, these morphisms fit into the following commutative diagram. 
\begin{equation*}
     \begin{tikzcd}
    & Y \arrow{dr}{g}\\
    X \arrow{ur}{f} \arrow{rr}{id_X}  & & X
    \end{tikzcd}
\end{equation*}Since $id_X$ is a strict epimorphism, $g$ is a strict epimorphism by the previous proposition. 
\end{proof}

\begin{exmp}
Suppose that $k$ is a valued field. The category of Banach spaces $\textcat{Ban}_k$, is quasi-abelian \cite[Lemma A.30]{benbassatkremnitzer17}, but isn't abelian since there exist non-strict morphisms. Consider the Banach space $C[0,1]$ of continuous real-valued functions on $[0,1]$ equipped with the sup-norm $|f|_{C[0,1]}=sup_{x\in[0,1]}|f(x)|$. Also consider the Banach space $L^1[0,1]$ of Lebesgue classes of integrable real-valued functions on $[0,1]$ equipped with the norm $|f|_{L^1[0,1]}=\int_0^1 |f(x)|\,dx$. We can see that the image of the inclusion map 
\begin{equation*}
    \iota:C[0,1]\hookrightarrow{L^1[0,1]}
\end{equation*}is dense and non-closed, and hence the inclusion map is not strict \cite[Page. 83]{bourlesmarinescu11}.
\end{exmp}
\begin{exmp}
For $R$ an Archimedean (resp. non-Archimedean) Banach ring, the category of Banach modules over $R$, $\textcat{Ban}^A_R$ (resp. $\textcat{Ban}^{nA}_R$), is quasi-abelian \cite[Proposition 3.15, 3.18]{bambozzibenbassat15}. 
\end{exmp}

\begin{exmp}
If $\mathcal{C}$ is quasi-abelian, then so is $\textcat{Ind}\mathcal{C}$. This follows since, for any morphism $f:X\rightarrow{Y}$ in $\textcat{Ind}\mathcal{C}$, we have that $\textnormal{Ker}(f)_i=\textnormal{Ker}(f_i)$ and $\textnormal{Coker}(f)_i=\textnormal{Coker}(f_i)$. In particular, for $R$ a Banach ring, the categories $\textcat{IndBan}_R^A$ and $\textcat{IndBan}_R^{nA}$ are quasi-abelian.
\end{exmp}

\begin{exmp}The categories $\textcat{Born}_k$ and $\textcat{CBorn}_k$ of bornological and complete bornological spaces respectively are quasi-abelian \cite[Lemma 2.19]{bambozzibenbassatkremnizer18}. 
\end{exmp}
There is the concept of exactness in quasi-abelian categories. Namely, 

\begin{defn} A null sequence in $\mathcal{E}$
 \begin{equation*}
       X\xrightarrow{f}{Y}\xrightarrow{g}{Z}
    \end{equation*}
is \textbf{strictly exact} (resp.\textbf{ coexact}) if $f$ (resp. $g$) is strict and the canonical morphism 
\begin{equation*}
    \textnormal{Im}(f)\rightarrow{\textnormal{Ker}(g)}
\end{equation*}is an isomorphism.
\end{defn}

We can extend this definition to define strictly exact complexes of objects in $\mathcal{E}$. We will say that a chain complex $Y_\bullet$ is a \textbf{strict resolution} of $X$ if the complex $Y_\bullet\rightarrow{X}$ is strictly exact. We note that, in a quasi-abelian category, we don't have a well defined notion of how far away a complex is from being strictly exact, i.e. we have no notion of \textit{homology}. Indeed, for a null sequence
 \begin{equation*}
       X\xrightarrow{f}{Y}\xrightarrow{g}{Z}
    \end{equation*}we could perhaps define homology at $Y$ as $\textnormal{Coker}(\textnormal{Im}(f)\rightarrow{\textnormal{Ker}(g)})$ or $\textnormal{Im}(\textnormal{Ker}(g)\rightarrow{\textnormal{Coker}(f)})$. However, these candidates are not always isomorphic as not all maps are strict.

\begin{defn}
An additive functor $F:\mathcal{E}\rightarrow{\mathcal{F}}$ between quasi-abelian categories is \textbf{left} (resp. \textbf{right}) \textbf{exact} if it transforms any strictly exact (resp. coexact) sequence 
\begin{equation*}
        0\rightarrow{X}\rightarrow{Y}\rightarrow{Z}\rightarrow{0}
    \end{equation*}of $\mathcal{E}$ into a strictly exact (resp. coexact) sequence
    \begin{equation*}
    \begin{aligned}
        0&\rightarrow{ F(X)}\rightarrow{F(Y)}\rightarrow{F(Z)}\\
        (\text{resp. } & F(X)\rightarrow{F(Y)}\rightarrow{F(Z)}\rightarrow{0})
    \end{aligned}
    \end{equation*}
$F$ is \textbf{exact} if it is both left and right exact.
\end{defn}
\begin{defn}
An object $P$ in $\mathcal{E}$ is \textbf{projective} if \begin{equation*}
   \textnormal{Hom}_{\mathcal{E}}(P,-):\mathcal{E}\rightarrow{\textcat{Ab}}
\end{equation*} is exact. Equivalently, $P$ is projective if, for any strict epimorphism $f:X\rightarrow{Y}$, the associated map
\begin{equation*}
    \textnormal{Hom}_{\mathcal{E}}(P,X)\rightarrow{Hom}_{\mathcal{E}}(P,Y)
\end{equation*}is surjective. We say that $\mathcal{E}$ \textbf{has enough projectives} if, for any object $X$ of $\mathcal{E}$, there is a strict epimorphism
\begin{equation*}
    P\twoheadrightarrow{X}
\end{equation*} where $P$ is a projective object of $\mathcal{E}$.
\end{defn}

\begin{defn}
Suppose that $\mathcal{E}$ is closed symmetric monoidal with monoidal product $\otimes:\mathcal{E}\times\mathcal{E}\rightarrow{\mathcal{E}}$. An object $F\in\mathcal{E}$ is \textbf{flat} if the functor \begin{equation*}
    F\otimes -:\mathcal{E}\rightarrow{\mathcal{E}}
\end{equation*}is exact. We say that $\mathcal{E}$ \textbf{has enough flat projectives} if, for any object $X$ of $\mathcal{E}$, there is a strict epimorphism 
\begin{equation*}
    P\twoheadrightarrow{X}
\end{equation*}where $P$ is a flat and projective object of $\mathcal{E}$.
\end{defn}

\begin{lem}\cite[Lemma 2.21]{bambozzibenbassat15}
    Suppose that $\mathcal{E}$ is closed symmetric monoidal with enough flat projectives. Then, any projective object of $\mathcal{E}$ is flat. 
\end{lem}

\begin{exmp}Suppose that $k$ is a non-Archimedean valued field and let $V\in \textcat{Ban}_k$. Define the Banach space
\begin{equation*}
    c_0(V)=\{(c_v)_{v\in V-\{0\}}\mid c_v\in k, \lim_{v\in V-\{0\}}|c_vv|=0\}
\end{equation*}equipped with the norm
\begin{equation*}
    |(c_v)_{v\in V-\{0\}}|=\sup_{v\in V-\{0\}}|c_vv|
\end{equation*}{where $\lim_{v\in V-\{0\}}|c_vv|$ is defined as in \cite[Definition A.24]{benbassatkremnitzer17}}. By \cite[Lemma A.38]{benbassatkremnitzer17}, this Banach space is projective in $\textcat{Ban}_k$. Moreover, there is a strict epimorphism $c_0(V)\rightarrow{V}$ for every $V\in\textcat{Ban}_k$. Therefore, $\textcat{Ban}_k$ has enough projectives. When $k$ is Archimedean, $\textcat{Ban}_k$ also has enough projectives \cite[Lemma A.39]{benbassatkremnitzer17}. Moreover, in both cases all projectives are flat.
\end{exmp}

\begin{exmp}
More generally, we also note that the categories $\textcat{Ban}_R^{A}$ and $\textcat{Ban}_R^{nA}$ have enough projectives and that any projective is flat \cite[Lemma 3.42]{benbassatkremnitzer21}. Further, if $M,N$ are projective in $\textcat{Ban}_R^A$ (resp. $\textcat{Ban}_R^{nA}$), then $P\hat{\otimes}_RQ$ is also projective in $\textcat{Ban}_R^A$ (resp. $\textcat{Ban}_R^{nA}$) \cite[Lemma 3.35]{bambozzikremnitzer20}. 
\end{exmp}

\begin{exmp}
The categories $\textcat{IndBan}_R^A$, $\textcat{IndBan}_R^{nA}$ and $\textcat{CBorn}_k$ have enough flat projectives by \cite[Lemmas 3.29 and 3.53]{bambozzibenbassat15}.
\end{exmp}

\begin{lem}\label{projectivesplit}
$P$ is a projective object if and only if every strict epimorphism $f:X\rightarrow{P}$ in $\mathcal{E}$ splits, i.e. there exists $g:P\rightarrow{X}$ such that $f\circ g=id_P$.
\end{lem}
\begin{proof}
We note that, since $P$ is projective, a strict epimorphism $f:X\rightarrow{P}$ induces a surjection
\begin{equation*}
    \textnormal{Hom}_{\mathcal{E}}(P,X)\rightarrow{\textnormal{Hom}_{\mathcal{E}}(P,P)}
\end{equation*}The preimage of $id_P$ under this surjection is a map $g:P\rightarrow{X}$ such that $g\circ f=id_P$.

To prove the converse, suppose that we have a strict epimorphism $f:X\rightarrow{Y}$ and a morphism $h:P\rightarrow{Y}$. We examine the pullback
\begin{equation*}
    \begin{tikzcd}
    P\times_Y X\arrow{r}{f'}\arrow{d}{h'}& P \arrow[d,"h"]\\
    X \arrow[r,two heads,"f"] & Y
    \end{tikzcd}
\end{equation*}Since $\mathcal{E}$ is quasi-abelian, $f'$ is a strict epimorphism. Hence, by assumption, $f'$ splits. So, there exists a map $g:P\rightarrow{P\times_YX}$ such that $f'\circ g=id_P$. Consider the map $h'\circ g:P\rightarrow{X}$. We see that $f\circ(h'\circ g)=h\circ f'\circ g=h$. Therefore, $P$ is projective. 
\end{proof}
We will denote by $\mathcal{P}$ the collection of projective objects $P$ such that there is a strict epimorphism $P\twoheadrightarrow{X}$ for some object $X\in\mathcal{E}$.
\begin{prop}
    Suppose that $\mathcal{E}$ has enough projectives. A morphism $f:X\rightarrow{Y}$ in $\mathcal{E}$ is a strict epimorphism if and only if the associated morphism 
        \begin{equation*}
            f':\textnormal{Hom}_{\mathcal{E}}(P,X)\rightarrow{\textnormal{Hom}_{\mathcal{E}}(P,Y)}
        \end{equation*}is surjective for any $P\in\mathcal{P}$.
\end{prop}
\begin{proof}
The forward direction is clear by definition of a projective object. Conversely, since $\mathcal{E}$ has enough projectives, there is a projective object $P$ and a strict epimorphism $g:P\twoheadrightarrow{Y}$. If $f'$ is surjective, there exists a map $h:P\rightarrow{X}$ such that $f'(h)=f\circ h=g$. Then, by Proposition \ref{monoepimorphismtriangle}, $f$ must be a strict epimorphism.
\end{proof}The following result is obtained similarly using \cite[Lemma 2.27]{bambozzibenbassat15}.
\begin{prop}\label{projectivegenerator}
Suppose that $\mathcal{E}$ is closed symmetric monoidal with enough flat projectives. A morphism $f:X\rightarrow{Y}$ in $\mathcal{E}$ is a strict epimorphism if and only if the associated morphism 
        \begin{equation*}
            f':\textnormal{Hom}_{\mathcal{E}}(P,X)\rightarrow{\textnormal{Hom}_{\mathcal{E}}(P,Y)}
        \end{equation*}is surjective for any $P\in \mathcal{P}$.
\end{prop}

\begin{defn}
An object $I$ in $\mathcal{E}$ is \textbf{injective} if \begin{equation*}
   \textnormal{Hom}_{\mathcal{E}}(-,I):\mathcal{E}^{op}\rightarrow{\textcat{Ab}}
\end{equation*} is exact. Equivalently, $I$ is injective if, for any strict monomorphism $f:X\rightarrow{Y}$, the associated map
\begin{equation*}
    \textnormal{Hom}_{\mathcal{E}}(Y,I)\rightarrow{Hom}_{\mathcal{E}}(X,I)
\end{equation*}is surjective. We say that $\mathcal{E}$ \textbf{has enough injectives} if, for any object $X$ of $\mathcal{E}$, there is a strict monomorphism
\begin{equation*}
    X\rightarrow{I}
\end{equation*} where $I$ is an injective object of $\mathcal{E}$.
\end{defn}

\begin{exmp}
Suppose that $k$ is a non-Archimedean field with a non-trivial valuation. Suppose in addition that $k$ is spherically complete, i.e. the intersection of all disks in any chain is nonempty. Let $V\in\textcat{Ban}_k$ and define the Banach space
\begin{equation*}
    \ell_{\infty}(V)=\bigg\{(c_v)_{v\in V-\{0\}}\mid c_v\in k,\, \sup_{v\in V-\{0\}}\frac{|c_v|}{|v|}<\infty\bigg\}
\end{equation*}equipped with the norm 
\begin{equation*}
    |(c_v)_{v\in V-\{0\}}|=\sup_{v\in V-\{0\}}\frac{|c_v|}{|v|}
\end{equation*}By \cite[Lemma A.41]{benbassatkremnitzer17}, $\ell_\infty(V)$ is injective in $\textcat{Ban}_k$. Moreover, if $k$ is any valued field, $\textcat{Ban}_k$ has enough injectives \cite[Lemma A.42]{benbassatkremnitzer17}.
\end{exmp}

\begin{prop}\label{homstrictlyexactoriginal}
    Suppose that $\mathcal{E}$ is a closed symmetric monoidal quasi-abelian category with enough flat projectives. Then $\texthom{Hom}_{\mathcal{E}}(-,I)$ is exact if $I$ is injective.
\end{prop}

\begin{proof}
Suppose that we have a  short strictly exact sequence
\begin{equation*}
    0\rightarrow{X}\rightarrow{Y}\rightarrow{Z}\rightarrow{0}
\end{equation*}It is clear that $\texthom{Hom}_{\mathcal{E}}(-,I)=\texthom{Hom}_{\mathcal{E}^{op}}(I,-)$ is left exact since it has a left adjoint $-\otimes I$ in $\mathcal{E}^{op}$. We see that the sequence
\begin{equation*}
    0\rightarrow{\texthom{Hom}_\mathcal{E}(Z,I)}\rightarrow{\texthom{Hom}_\mathcal{E}(Y,I)}\rightarrow{\texthom{Hom}_\mathcal{E}(X,I)}
\end{equation*}is strictly exact. It remains to show that 
\begin{align*}
    \texthom{Hom}_\mathcal{E}(Y,I)&\rightarrow{\texthom{Hom}_\mathcal{E}(X,I)}
\intertext{is a strict epimorphism. By Proposition \ref{projectivegenerator} it suffices to show that for all flat projectives $P$, the following map is a surjection}
    \textnormal{Hom}_{\mathcal{E}}(P,\texthom{Hom}_\mathcal{E}(Y,I))&\rightarrow{\textnormal{Hom}_{\mathcal{E}}(P,\texthom{Hom}_\mathcal{E}(X,I))}
\intertext{
Equivalently, by the internal hom adjunction, we can just show that the following map is a surjection}
\textnormal{Hom}_{\mathcal{E}}(P\otimes Y,I)&\rightarrow{\textnormal{Hom}_{\mathcal{E}}(P\otimes X,I)}
\end{align*}Since $I$ is injective, this map is a surjection if $P\otimes X\rightarrow{P\otimes Y}$ is a strict monomorphism. The result then follows since $P$ is flat and the map $X\rightarrow{Y}$ is a strict monomorphism.
\end{proof}

\subsection{The Derived Category}

The derived category of a quasi-abelian category can be defined analogously to the abelian case. Suppose that we have a quasi-abelian category $\mathcal{E}$. We denote the category of cochain complexes in $\mathcal{E}$ by $\textcat{C}(\mathcal{E})$. We note that cohomology is not well defined in $\textcat{C}(\mathcal{E})$. However, since $\mathcal{E}$ is additive, cones of morphisms are well-defined and we can make the following definition.
\begin{defn}A cochain morphism $f^\bullet:C^\bullet\rightarrow{D^\bullet}$ in $\textcat{C}(\mathcal{E})$ is a \textbf{strict quasi-isomorphism} if $cone(f^\bullet)^\bullet$ is strictly exact. 
\end{defn} 

Since strict quasi-isomorphisms are stable under cochain homotopy equivalence, we can define the notion of a strict quasi-isomorphism in the homotopy category $\textcat{K}(\mathcal{E})$. Similarly to the abelian case, we make the following definition.
\begin{defn}The \textbf{derived category} $\textcat{D}(\mathcal{E})$ is the localisation of the homotopy category $\textcat{K}(\mathcal{E})$ at the class of all strict quasi-isomorphisms.
\end{defn}

\begin{remark}
We remark that, since homotopic maps become equal when strict quasi-isomorphisms are inverted, the category obtained by inverting strict quasi-isomorphisms in $\textcat{C}(\mathcal{E})$ is equivalent to the derived category $\textcat{D}(\mathcal{E})$. I will call both the \textit{derived category} $\textcat{D}(\mathcal{E})$.
\end{remark} 

\begin{exmp}
The categories $\textcat{IndBan}_k$ and $\textcat{CBorn}_k$ are tensor derived equivalent \cite[c.f. Proposition 3.16]{bambozzikremnitzer20} in the sense that there is an equivalence of derived categories
\begin{equation*}
    \textcat{D}(\textcat{CBorn}_k)\simeq \textcat{D}(\textcat{IndBan}_k)
\end{equation*}which preserves the monoidal structure.
\end{exmp}

In \cite[Section 1.2]{schneiders99}, Schneiders details an embedding of $\mathcal{E}$ into an abelian category, its \textit{abelianisation}, using t-structures. This embedding allows us to do homological algebra in quasi-abelian categories by working within abelian categories. We summarise his results in the following proposition. 
\begin{prop}\label{leftheartembedding}
    There exists an abelian category $\mathcal{LH(E)}$, the `left heart' of $\mathcal{E}$, and a fully faithful functor 
    \begin{equation*}
        I:\mathcal{E}\rightarrow{\mathcal{LH(E)}}
    \end{equation*}inducing an equivalence of categories
    \begin{equation*}
        \textcat{D}(I):\textcat{D}(\mathcal{E})\rightarrow{\textcat{D}(\mathcal{LH(E)})}
    \end{equation*}Moreover, $I$ has a left adjoint $C:\mathcal{LH(E)}\rightarrow{\mathcal{E}}$ with 
    \begin{equation*}
        C\circ I\simeq id_{\mathcal{E}}
    \end{equation*}
\end{prop}
We have the following corollary. 
\begin{cor}\label{strictexactembedding} \cite[Corollary 1.2.27]{schneiders99} A sequence 
\begin{equation*}
    X\rightarrow{Y}\rightarrow{Z}
\end{equation*}is strictly exact in $\mathcal{E}$ if and only if the sequence
\begin{equation*}
    I(X)\rightarrow{I(Y)}\rightarrow{I(Z)}
\end{equation*}is exact in $\mathcal{LH(E)}$.

\end{cor}

\begin{cor}\label{Iquasiiso}
A cochain morphism $f^\bullet$ is a strict quasi-isomorphism in $\mathcal{E}$ if and only if $I(f^\bullet)$ is a quasi-isomorphism in $\mathcal{LH(E)}$.
\end{cor}
\begin{proof}
Note that $f^\bullet$ is a strict quasi-isomorphism if and only if $cone(f^\bullet)^\bullet$ is a strictly exact complex. By the previous corollary, $cone(f^\bullet)^\bullet$ is strictly exact if and only if $I(cone(f^\bullet))^\bullet$ is exact in $\mathcal{LH(E)}$. Since $I$ preserves cones, this occurs if and only if $cone(I(f^\bullet))^\bullet$ is exact in $\mathcal{LH(E)}$, equivalently if and only if $I(f^\bullet)$ is a quasi-isomorphism in $\mathcal{LH(E)}$.
\end{proof}

\begin{prop}
If $f:X\rightarrow{Y}$ is a strict map in $\mathcal{E}$, then 
\begin{equation*}
     \textnormal{Im}(I(f))\simeq I(\textnormal{Im}(f))\quad\text{and}\quad\textnormal{Coim}(I(f))\simeq I(\textnormal{Coim}(f))
\end{equation*}It follows that, for a strict cochain complex $(C^\bullet,d^\bullet)$ in $\mathcal{E}$,
\begin{equation*}
    H^n(I(C^\bullet))\simeq\textnormal{Coker}(I(\textnormal{Im}(d^{n-1}))\rightarrow{I(\textnormal{Ker}(d^n))})
\end{equation*}
\end{prop}

\begin{proof}

Since $I$ is exact by Corollary \ref{strictexactembedding}, $I$ preserves kernels and cokernels. Therefore, 
\begin{equation*}\begin{aligned}
    \textnormal{Im}(I(f))&:=\textnormal{Ker}(I(Y)\rightarrow{\textnormal{Coker}(I(f))})\\
    &\simeq \textnormal{Ker}(I(Y)\rightarrow{I(\textnormal{Coker}(f)}))\\
    &\simeq I(\textnormal{Ker}(Y\rightarrow{\textnormal{Coker}(f)}))\\
    &\simeq I(\textnormal{Im}(f))
\end{aligned}\end{equation*}The result for the coimage follows similarly. Therefore, 
\begin{equation*}
    \begin{aligned}
         H^n(I(C^\bullet))&=\textnormal{Coker}(\textnormal{Im}(I(d^{n-1}))\rightarrow{\textnormal{Ker}(I(d^n))})\\
         &\simeq\textnormal{Coker}(I(\textnormal{Im}(d^{n-1}))\rightarrow{I(\textnormal{Ker}(d^n))})
    \end{aligned}
\end{equation*}
\end{proof}

\subsection{Derived Functors}

For any category $\textcat{Com}(\mathcal{E})$ consisting of chain complexes, the notation $\textcat{Com}^*(\mathcal{E})$ will be used to denote the subcategories consisting of bounded above $(-)$, bounded below $(+)$, or bounded $(b)$ chain complexes. 
\begin{prop}\label{inducedderived}
    If $F:\textcat{C}^*(\mathcal{E})\rightarrow{\textcat{C}^*(\mathcal{F})}$ maps strict quasi-isomorphisms to strict quasi-isomorphisms, then it induces a functor $\textcat{D}^*F:\textcat{D}^*(\mathcal{E})\rightarrow{\textcat{D}^*(\mathcal{F})}$.
\end{prop}

\begin{proof}
If $F$ maps strict quasi-isomorphisms to strict quasi-isomorphisms, then the induced functor $\textcat{K}F:\textcat{K}^*(\mathcal{E})\rightarrow{\textcat{K}^*(\mathcal{F})}$ does too. We denote our localisation functor by $Q_\mathcal{F}:\textcat{K}^*(\mathcal{F})\rightarrow{\textcat{D}^*(\mathcal{F})}$ and consider the composition $Q_\mathcal{F}\circ \textcat{K}F:\textcat{K}^*(\mathcal{E})\rightarrow{\textcat{D}^*(\mathcal{F})}$. This maps strict quasi-isomorphisms to isomorphisms, and hence, by definition of the localisation, this induces a functor $\textcat{D}F:\textcat{D}^*(\mathcal{E})\rightarrow{\textcat{D}^*(\mathcal{F})}$.
\end{proof}

\begin{prop}\label{strictquasiexact}
    A functor $F:\textcat{C}^*(\mathcal{E})\rightarrow{\textcat{C}^*(\mathcal{F})}$ maps strict quasi-isomorphisms to strict quasi-isomorphisms if it preserves strict exactness. 
\end{prop}

\begin{proof}
    Indeed, suppose we have a strict quasi-isomorphism $f^\bullet\in\textcat{C}^*(\mathcal{E}) $. Then, $cone(f^\bullet)^\bullet$ is strictly exact. If $F$ preserves strict exactness, then $F(cone(f^\bullet)^\bullet)\simeq cone(F(f^\bullet)^\bullet)$ is strictly exact, and hence $F(f^\bullet)$ is a strict quasi-isomorphism. 
\end{proof}

As in the abelian case, we can make the following definition. Let $\textcat{D}^+(\mathcal{E})$ be the derived category of the category of bounded below cochain complexes. Suppose that $F:\mathcal{E}\rightarrow{\mathcal{F}}$ is an additive functor, and denote the localisation functors by $Q_{\mathcal{E}}:\textcat{K}(\mathcal{E})\rightarrow{\textcat{D}(\mathcal{E})}$ and $Q_{\mathcal{F}}:\textcat{K}(\mathcal{F})\rightarrow{\textcat{D}(\mathcal{F})}$.
\begin{defn}
The \textbf{right derived functor} of $F$ is a pair $(\textnormal{\textbf{R}}F,\epsilon_F)$ consisting of a triangulated functor
\begin{equation*}
\textnormal{\textbf{R}}F:\textnormal{\textbf{D}}^+(\mathcal{E})\rightarrow{\textnormal{\textbf{D}}^+(\mathcal{F})}    
\end{equation*} and a natural transformation \begin{equation*}
\epsilon_F:Q_\mathcal{F}\circ\textcat{K}^+(F)\Rightarrow{\textnormal{\textbf{R}}F\circ Q_\mathcal{E}}    
\end{equation*} universal in the sense that, for any triangulated functor   $G:\textnormal{\textbf{D}}^+(\mathcal{E})\rightarrow{\textnormal{\textbf{D}}^+(\mathcal{F})}$,  and any natural transformation $\epsilon:Q_\mathcal{F}\circ\textcat{K}^+(F)\Rightarrow{G\circ Q_\mathcal{E}}$, there exists a unique natural transformation $\eta:\textnormal{\textbf{R}}F\Rightarrow{G}$ such that the diagram
\begin{equation*}
    \begin{tikzcd}
    & Q_\mathcal{F}\circ \textnormal{\textbf{K}}^+F \arrow[ld,"\epsilon_F"'] \arrow{rd}{\epsilon}\\
    \textnormal{\textbf{R}}F\circ Q_\mathcal{E} \arrow{rr}{\eta\circ id_{Q_\mathcal{E}}} & & G\circ Q_\mathcal{E}
    \end{tikzcd}
\end{equation*}
commutes. 
\end{defn}
\begin{remark}
We can define left derived functors similarly.
\end{remark}

In \cite[Section 1.3]{schneiders99}, Schneiders gives some conditions for the existence of left and right derived functors. In particular, if $\mathcal{E}$ has enough injective objects, then any additive functor $F:\mathcal{E}\rightarrow{\mathcal{F}}$ is right derivable. We remark that, since any projective object in $\mathcal{E}$ is injective in $\mathcal{E}^{op}$, if $\mathcal{E}$ has enough projectives then an additive functor $G:\mathcal{E}^{op}\rightarrow{\mathcal{F}}$ is right derivable. Hence, if $\mathcal{E}$ has enough projectives the functor $\textnormal{Hom}_{\mathcal{E}}(-,Y):\mathcal{E}^{op}:\rightarrow{\textcat{Ab}}$ is right derivable. As in the abelian case, we can make the following definition.

\begin{defn}
Suppose that $\mathcal{E}$ has enough projectives. Let $X$ be an object in $\mathcal{E}$ and $P^\bullet\rightarrow{X}$ be a strict projective resolution of $X$. Then, the \textbf{$i^{th}$ Ext object} of $X$ is 
\begin{equation*}
    \textnormal{Ext}_{\mathcal{E}}^i(X,Y):=\textcat{R}^i\textnormal{Hom}_{\mathcal{E}}(X,Y)=H^i(\textnormal{Hom}_\mathcal{E}(P^\bullet,Y))
\end{equation*}

\end{defn}

\begin{remark}
For an explanation as to why this is well defined, see \cite[Remark 12.11]{buhler09}.
\end{remark}

\begin{remark}
If $\mathcal{E}$ does not have enough projectives, then you can use the Yoneda construction \cite{yoneda60} to construct $\textnormal{Ext}$.
\end{remark}

\section{Modules in Quasi-abelian Categories}\label{haquasiabeliancats}
For the remainder of this paper we will fix a bicomplete closed symmetric monoidal quasi-abelian category $\mathcal{E}$ with enough flat projectives. We remark that the categories $\textcat{CBorn}_k$ and $\textcat{IndBan}_k$ satisfy these properties.
\subsection{Categories of Modules}
Recall that the category of left modules over a monoid $A$ is denoted by $\textcat{A-Mod}$.
\begin{prop}\cite[Proposition 1.5.1]{schneiders99}
Suppose that $A$ is a monoid in $\mathcal{E}$, then $\textcat{A-Mod}$ is quasi-abelian. Moreover, a morphism of $\textcat{A-Mod}$ is strict if and only if it is strict as a morphism of $\mathcal{E}$, since the forgetful functor $\textcat{A-Mod}\rightarrow{\mathcal{E}}$ preserves limits and colimits.
\end{prop}

\begin{prop}\cite[c.f. Proposition 1.5.2]{schneiders99}
The functor $A\otimes-:\mathcal{E}\rightarrow{\textcat{A-Mod}}$ is left adjoint to the forgetful functor $\textcat{A-Mod}\rightarrow{\mathcal{E}}$. For any object $X\in\mathcal{E}$, \begin{equation}\label{freeforgetfuladjunction}
    \textnormal{Hom}_{\textcat{A-Mod}}(A\otimes X,-)\simeq \textnormal{Hom}_{\mathcal{E}}(X,-)\end{equation}
\end{prop}

\begin{cor}\label{tensorprojective}
$P$ is a (flat) projective object in $\mathcal{E}$ if and only if $A\otimes P$ is a (flat) projective object in $\textcat{A-Mod}$. Moreover, if $\mathcal{E}$ has enough (flat) projectives, then so does $\textcat{A-Mod}$.
\end{cor}

\begin{proof}
By the previous proposition, we see that $\textnormal{Hom}_{\textcat{A-Mod}}(A\otimes P,-)$ is strictly exact if and only if $\textnormal{Hom}_{\mathcal{E}}(P,-)$ is. Hence, $P$ is projective in $\mathcal{E}$ if and only if $A\otimes P$ is. Now, suppose we have an $A$-module $M$. Since $\mathcal{E}$ has enough projectives, there exists a projective object $P$ in $\mathcal{E}$ and a strict epimorphism 
\begin{equation*}
    f:P\rightarrow{M}
\end{equation*} We note that $A\otimes P$ is a projective $A$-module and that, under the isomorphism in Equation \ref{freeforgetfuladjunction}, $f$ corresponds to a strict epimorphism $f':A\otimes P\rightarrow{M}$. Hence, $\textcat{A-Mod}$ has enough projectives. The flatness case follows from noting that $A$ is flat in $\textcat{A-Mod}$ and that the monoidal product of flat objects is flat.
\end{proof}

We state the following important lemma.

\begin{lem}\label{kernelprojective}
If $M$ and $N$ are projective $A$-modules and $f:M\rightarrow{N}$ is a strict epimorphism, then $\textnormal{Ker}(f)$ is a projective $A$-module.
\end{lem}

\begin{proof}
We consider the strictly exact sequence
\begin{equation*}
    0\rightarrow{\textnormal{Ker}(f)}\xrightarrow{\iota}{M}\xrightarrow{f}{N}\rightarrow{0}
\end{equation*}Since $N$ is a projective $A$-module, this sequence splits by Lemma \ref{projectivesplit} and hence 
\begin{equation*}
    M\simeq \textnormal{Ker}(f)\oplus N
\end{equation*}Therefore, since $\textnormal{Ker}(f)$ is a direct summand of a projective object and $N$ is projective, $\textnormal{Ker}(f)$ is projective.
\end{proof}

\subsection{Categories of Graded Modules}
Suppose now that $A$ is a graded monoid in $\mathcal{E}$. Since $\mathcal{E}$ has countable colimits we can denote $A$ as 
\begin{equation*}
    A=\bigoplus_{i\in\mathbb{Z}}A_i
\end{equation*}

\begin{prop}
The category of graded left $A$-modules, denoted $\textcat{grA-Mod}$, is quasi-abelian. 
\end{prop}
\begin{proof}
We note easily that $\textcat{grA-Mod}$ has all kernels and cokernels. Since $\textbf{A-Mod}$ is quasi-abelian, we see that, by working {in each degree}, the dual axioms in the definition of a quasi-abelian category are satisfied.
\end{proof}

\begin{lem}\label{gradedprojective}
If a graded left $A$-module $M$ is projective as an $A$-module then it is projective as a graded $A$-module. 
\end{lem}
\begin{proof}
It is easy to see that, for a graded morphism $N_1\rightarrow{N_2}$ of graded left $A$-modules, the induced surjection
\begin{equation*}
    \textnormal{Hom}_{\textbf{A-Mod}}(M,N_1)\rightarrow{\textnormal{Hom}_{\textbf{A-Mod}}(M,N_2)}
\end{equation*} must be a surjection in $\textcat{grA-Mod}$.
\end{proof}

\begin{cor}
If $M$ is a graded projective left $A$-module, then $A\otimes M$ is a graded projective left $A$-module.
\end{cor}

\begin{proof}
We see from Corollary \ref{tensorprojective} that $A\otimes M$ is a projective left $A$-module. The result then follows from Lemma \ref{gradedprojective}.
\end{proof}

Suppose now that $A$ is positively-graded. We note that, if $M$ is a left $A$-module, then $A\otimes_{A_0} M$ is a left $A$-module with obvious $A$-action.

\begin{prop}\label{strictmodulemap} Suppose that $M$ is a graded left $A$-module. There exists a map $A\otimes_{A_0} M\rightarrow{M}$ which is a {strict graded epimorphism}.
\end{prop}
\begin{proof}We note that the graded maps
\begin{equation*}
    A\otimes A_0\otimes M\rightrightarrows{A\otimes M}\rightarrow{M}
\end{equation*}given by the composition of the action of $A_0$ on $A$ and $M$ and the action of $A$ on $M$, are equal. Hence, by the universal property of the coequalizer, there exists a graded map $A\otimes_{A_0} M\rightarrow{M}$ such that the following diagram commutes. 

\begin{equation*}
    \begin{tikzcd}
    A\otimes A_0\otimes M \ar[r,shift left=.75ex]
  \ar[r,shift right=.75ex,swap] & A\otimes M \arrow{r} \arrow{dr} & {A\otimes_{A_0} M} \arrow{d}\\
  & & M
    \end{tikzcd}
\end{equation*}We note that the morphism $A\otimes M\rightarrow{M}$ is a strict epimorphism by Corollary \ref{rightinverseepimorphism} since it has a right inverse given by inclusion. Hence, by Proposition \ref{monoepimorphismtriangle} the map $A\otimes_{A_0} M\rightarrow{M}$ is a strict epimorphism.
\end{proof}
\begin{prop}\label{gradedtensor}
If $M$ is a projective left $A_0$-module, then $A\otimes_{A_0}M$ is a projective left $A$-module.
\end{prop}
\begin{proof}
We want to show that the functor 
\begin{equation*}
    \textnormal{Hom}_{\textcat{A-Mod}}(A\otimes_{A_0} M,-):\textbf{A-Mod}\rightarrow{\textcat{Ab}}
\end{equation*} maps strict epimorphisms to surjections. Indeed, by Theorem \ref{tensorhomadjunction}, 
\begin{equation*}
    \textnormal{Hom}_{\textcat{A-Mod}}(A\otimes_{A_0} M,-)\simeq \textnormal{Hom}_{\textcat{A}_0\textcat{-Mod}}(M,\texthom{Hom}_\textcat{A-Mod}(A,-))
\end{equation*}Now, by Lemma \ref{usefulfacts},  $\texthom{Hom}_{\textcat{A-Mod}}(A,-)$ defines an isomorphism. Since $M$ is a projective $A_0$-module, we see that the functor must send strict epimorphisms to surjections. 
\end{proof}

\begin{remark}
If $M$ is a graded projective left $A_0$-module, then it is easy to see that $A\otimes_{A_0} M$ is a graded projective left $A$-module with 
\begin{equation*}
    (A\otimes_{A_0} M)_n=\bigoplus_{i+j=n}A_i\otimes_{A_0} M_j=\bigoplus_{i+j=n}\textnormal{Coeq}(A_i\otimes A_0\otimes M_j\rightrightarrows{A_i\otimes M_j})
\end{equation*}
\end{remark}

\begin{lem}
If $\textcat{A}_0\textcat{-Mod}$ has enough projectives, then the category $\textcat{grA-Mod}$ has enough projectives.
\end{lem}
\begin{proof}
Suppose that $M$ is a graded left $A$-module, $M=\bigoplus_{i\in \mathbb{Z}}M_i$. Consider each component $M_i$ as a left $A_0$-module. Then, since $\textcat{A}_0\textcat{-Mod}$ has enough projectives, there exists a projective left $A_0$-module $P_i$ and an epimorphism $P_i\rightarrow{M_i}$. We note that $A\otimes_{A_0} P_i$ is a projective left $A$-module by Proposition \ref{gradedtensor}, and there exists a strict epimorphism $A\otimes_{A_0} P_i\rightarrow{M_i}$ given by the composition of the strict epimorphisms $A\otimes_{A_0} P_i\rightarrow{P_i}$ and $P_i\rightarrow{M_i}$. Taking the coproduct of all of these $A\otimes_{A_0}P_i$ in the category of graded $A$-modules gives us a graded projective $A$-module $\bigoplus_{i\in \mathbb{Z}}A\otimes_{A_0} P_i$ along with a strict graded epimorphism $\bigoplus_{i\in \mathbb{Z}}A\otimes_{A_0} P_i\rightarrow{M}$.

\end{proof}

\section{Koszul Monoids}\label{koszulmonoids}

For a positively graded monoid $A$, we let $A_{>0}=\bigoplus_{i> 0} A_i$. There is a short strictly exact sequence
\begin{equation*}
    0\rightarrow{A_{>0}}\xrightarrow{\iota}{A}\xrightarrow{\pi}{A_0}\rightarrow{0}
\end{equation*}We can consider $A_0$ as a graded left $A$-module with, for each $j$, the action $A_j\otimes A_0\rightarrow{A_0}$ defined to be the composition \begin{equation*}
    A_j\otimes A_0\xrightarrow{\mu_{j,0}}{A_j}\hookrightarrow{A}\twoheadrightarrow{ Coker(A_{>0}\xrightarrow{\iota}{A})}\simeq A_0
\end{equation*}

\begin{defn}
 We will say that a graded left $A$-module $M$ is \textbf{generated by its degree $i$ component over $A_0$} if the map 
 \begin{equation*}
     A\otimes_{A_0} M_i\twoheadrightarrow{M}
 \end{equation*}is a strict epimorphism.
\end{defn}
\begin{defn}
A graded left $A$-module is called \textbf{pure of weight n} if it is concentrated in degree $n$, i.e. $M=M_{n}$.
\end{defn}For a graded left $A$-module $M$ we define the grading shifts by 
\begin{equation*}
    (M\langle n\rangle)_i=M_{i-n}
\end{equation*} We can consider $A_0$ to be a pure graded left $A$-module concentrated in degree $0$ by defining the action of $A_i$ on $A_0$ to be $0$ unless $i=0$. We remark that for any $A$-modules $M$ and $N$ pure of weights $m,n$ respectively, we have that $\textnormal{Hom}_{\textcat{grA-Mod}}(M,N)=0$ unless $m=n$.

\begin{prop}\label{purezero}
If $M$ is an $A$-module generated by its degree $i$ component over $A_0$ and $N$ is pure of weight $n$, then 
\begin{equation*}
    \textnormal{Hom}_{\textcat{grA-Mod}}(M,N)=0
\end{equation*}unless $i=n$
\end{prop}

\begin{proof}
We note that, since there is a strict epimorphism $A\otimes_{A_0} M_i\twoheadrightarrow{M}$, there is an injection
\begin{equation*}
    {\textnormal{Hom}_{\textcat{grA-Mod}}(M,N)}\hookrightarrow\textnormal{Hom}_{\textcat{grA-Mod}}(A\otimes_{A_0}M_i,N)
\end{equation*}The result then follows since
\begin{equation*}
    \textnormal{Hom}_{\textcat{grA-Mod}}(A\otimes_{A_0}M_i,N)\simeq\textnormal{Hom}_{\textcat{grA}_0\textcat{-Mod}}(M_i,N_n)=0 
\end{equation*}if $i\neq n$. 
\end{proof}

In \cite{beilinsonginzburg96}, Koszul duality theory is developed for $A_0$ a semisimple ring. In our category $\mathcal{E}$, it is difficult to develop a satisfactory analogous definition of a semisimple object. Hence, we will develop our theory for $A_0$ satisfying the following more general conditions, and call $A$ a pre-Koszul monoid.

\begin{defn}\label{pre-koszul}We will say that a positively graded monoid $A$ is \textbf{pre-Koszul} if $A_0$ is injective as a module over itself, each $A_i$ is projective as an $A_0$-module, and, for any graded $A$-module $M$ living only in degrees $\geq i$, whenever \begin{equation*}
    \textnormal{Hom}_{\textcat{grA-Mod}}(M,A_0\langle n\rangle)=0
\end{equation*} unless $n=i$, then $M$ is generated by its $i^{th}$ component over $A_0$.
\end{defn}
\begin{remark}
We note that our definition of being pre-Koszul differs from that given by Priddy \cite[Section 2]{priddy70}.
\end{remark}\begin{prop}\label{semisimpleprekoszul}
In the category of rings, if $A_0$ is a semisimple ring then $A$ is pre-Koszul.
\end{prop}

\begin{proof}
Indeed, we note that all $A_0$-modules are projective and that $A_0$ is injective over itself. It therefore suffices to prove the second condition. Suppose that we have a graded $A$-module $M$ living only in degrees $\geq i$. We will show that if \begin{equation*}
    \textnormal{Hom}_{\textcat{grA-Mod}}(M,A_0\langle n\rangle)=0
    \end{equation*}unless $n=i$, then there exists a surjection $A\otimes_{A_0} M_i\twoheadrightarrow{M}$. Choose the least $j$ such that there exists $m_j\in M_j$ not in the image of the map $A_{j-i}\otimes_{A_0} M_i\rightarrow{M_j}$. We note that $j>i$. We consider $M_j$ as a pure $A_0$-module of weight $j$ and consider the simple $A_0$-submodule $A_0 m_j$ of $M_j$. Since pure modules over $A_0$ are semisimple,  $M_j$ is semisimple, and hence completely reducible. Therefore, there exists an $A_0$-submodule $N$ of $M_j$ such that 
\begin{equation*}
    M_j=A_0m_j\oplus N
\end{equation*}We define an $A$-module morphism $f:M\rightarrow{A_0\langle j\rangle}$ as follows 
\begin{equation*}
    f(m)=\begin{cases}a &\text{if }m=am_j\in A_0m_j\text{ for some }a\in A_0\\
    0 &\text{otherwise}\end{cases}
\end{equation*}and extend linearly. Then, we see that this map is graded since $f(M_k)\subseteq A_0\langle{j}\rangle_k$. This map is non-zero which contradicts that 
\begin{equation*}
    \textnormal{Hom}_{\textcat{grA-Mod}}(M,A_0\langle j\rangle)=0
\end{equation*}Hence, there exists a surjection $A\otimes_{A_0} M_i\twoheadrightarrow{M}$.
\end{proof}

\begin{remark}
We note that if $A$ is pre-Koszul in the category of rings, then $A_0$ is not necessarily semisimple. Indeed, not every module over $A_0$ is guaranteed to be projective. See \cite[Example 4.4]{greenreitensolberg} for an example of a ring $A$ such that $A_0$ is not semisimple but $A$ is pre-Koszul.
\end{remark}

\begin{defn}
A \textbf{Koszul monoid} in $\mathcal{E}$ is a positively graded pre-Koszul monoid $A$ such that $A_0$, considered as a graded left $A$-module, admits a graded projective strict resolution of $A$-modules.
 \begin{equation*}
     \dots \rightarrow{P^2}\rightarrow{P^1}\rightarrow{P^0}\twoheadrightarrow{A_0}
 \end{equation*}with each $P^i$ generated by its degree $i$ component over $A_0$.
\end{defn}

\begin{remark}
The condition on $P^i$ says that the diagonal part of the resolution generates the rest.
\end{remark}
{\begin{exmp}
    We note that in the category of abelian groups, any Koszul ring in the sense of \cite[Definition 1.1.2]{beilinsonginzburg96} is a Koszul monoid. 
\end{exmp}}

{\begin{exmp}
    Suppose that we have a Koszul algebra $R$ over a field $k$ in the sense of \cite[Definition 1.1.2]{beilinsonginzburg96}. We consider the fine bornology \cite[Definition 1.11]{meyer07} on its underlying vector space such that $R$ becomes a complete bornological ring \cite[Example 1.27]{meyer07}. Suppose further that $k$ is self-injective in $\textcat{CBorn}_k$, e.g. $k=\mathbb{C}, \mathbb{R}$. Then, $R$ is a Koszul monoid in $\textcat{CBorn}_k$ since the functor $\textcat{Vect}_k\rightarrow{\textcat{CBorn}_k}$ is fully faithful and exact \cite[Example 1.77]{meyer07}.
\end{exmp}}

We now discuss in what sense a Koszul monoid is `as close to being semisimple' as a graded monoid can be. For the rest of this section, we fix a positively graded monoid $A$ in $\mathcal{E}$.

\begin{prop}\label{propprojresolutionforM}
Let $M\in\textcat{grA-Mod}$ be a projective left $A_0$-module living only in degrees $\geq n$. Then, $M$ admits a graded projective strict resolution of $A$-modules
\begin{equation*}
    \dots\rightarrow{P^2}\rightarrow{P^1}\rightarrow{P^0}\twoheadrightarrow{M}
\end{equation*}such that $P^i$ lives only in degrees $\geq n+i$. So $P^i=\bigoplus_{j\geq n+i}P^i_j$.
\end{prop}

\begin{proof}We may assume, without any loss of generality, that $n=0$. We consider the module $P^0=A\otimes_{A_0} M$. This is a graded projective $A$-module by Proposition \ref{gradedtensor} and, moreover, it lives only in positive degree. By Proposition \ref{strictmodulemap}, there exists a strict graded epimorphism $d^0:P^0\rightarrow{M}$. This fits into a short strictly exact sequence
\begin{equation*}
    0\rightarrow{K^0}\xrightarrow{\iota^0}{P^0}\xrightarrow{d^0}{M}\rightarrow{0}
\end{equation*} with $K^0=\textnormal{Ker}(d^0)$, a graded $A_0$-module, and $\iota^0:K^0\rightarrow{P^0}$ the inclusion map, a strict monomorphism. Since $P^0_0=A_0\otimes_{A_0} M_0\simeq M_0$ we see that in degree $0$, $d^0$ is a monomorphism. Hence, $K^0$ lives only in degree $\geq 1$. 

We now let $P^1=A\otimes_{A_0}K^0$. Since $K^0$ is the kernel of a strict epimorphism of projective $A_0$-modules, it is a projective $A_0$-module by Lemma \ref{kernelprojective}. Hence, $P^1$ is a projective $A$-module by Proposition \ref{gradedtensor}. We note that $P^1$ is graded and lives only in degrees $\geq 1$. We construct a map $d^1:P^1\rightarrow{P^0}$ as the composition of the map $P^1\rightarrow{K^0}$ and the inclusion map $\iota^0:K^0\rightarrow{P^0}$. This map $d^1$ is strict since it a composition of a strict monomorphism and a strict epimorphism. The sequence
\begin{equation*}
    P^1\xrightarrow{d^1}{P^0}\xrightarrow{d^0}{M}\rightarrow{0}
\end{equation*} is strictly exact.

We can continue in this way to construct our desired projective resolution. Indeed, suppose we have a strictly exact sequence of graded projective $A$-modules defined up to degree $i$
\begin{equation*}
    P^i\xrightarrow{d^i}{\dots}\rightarrow{P^1}\xrightarrow{d^1}{P^0}\xrightarrow{d^0}{M}\rightarrow{0}
\end{equation*} with each $P^j$ living only in degree $\geq j$. Suppose also that the sequence is constructed such that $P^j=A\otimes_{A_0} K^{j-1}$ with $K^{j-1}$ the graded $A_0$-module $\textnormal{Ker}(d^{j-1}:P^{j-1}\twoheadrightarrow{\textnormal{Im}(d^{j-1})}\simeq K^{j-2})$. We let $K^{i+1}=\textnormal{Ker}(d^{i+1}:P^{i}\twoheadrightarrow{\textnormal{Im}(d^i)}\simeq K^{i-1})$, which is a projective $A_0$-module, and define $P^{i+1}=A\otimes_{A_0} K^i$. This is a projective graded $A$-module living only in degrees $\geq i+1$. The differential $d^{i+1}:P^{i+1}\rightarrow{P^i}$ is the strict map defined as the composition of the map $P^{i+1}=A\otimes_{A_0}K^i\rightarrow{K^i}$ and the inclusion $K^i\rightarrow{P^i}$. The extended sequence
\begin{equation*}
    P^{i+1}\xrightarrow{d^{i+1}}{P^i}\xrightarrow{d^i}{\dots}\rightarrow{P^1}\xrightarrow{d^1}{P^0}\xrightarrow{d^0}{M}\rightarrow{0}
\end{equation*}is strictly exact. 

\end{proof}

\begin{cor}
\label{projresolutionforM}
If $M\in\textcat{grA-Mod}$ is a pure projective left $A_0$-module of weight $n$, then $M$ admits a graded projective strict resolution of $A$-modules
\begin{equation*}
    \dots\rightarrow{ P^2}\rightarrow{P^1}\rightarrow{P^0}\twoheadrightarrow{M}
\end{equation*} such that $P^i$ lives only in degrees $\geq n+i$.
\end{cor}

\begin{cor}Let $M,N\in \textcat{grA-Mod}$ be pure of weights $m,n$, with $M$ projective as a left module over $A_0$. Then, 
\begin{equation*}
    \textnormal{Ext}_{\textcat{grA-Mod}}^i(M,N)=0
\end{equation*}for $i>n-m$.
\end{cor}

\begin{proof}
Without loss of generality we may assume $m=0$. By Corollary \ref{projresolutionforM}, $M$ admits a graded projective resolution of $A$-modules
\begin{equation*}
    \dots\rightarrow{P^2}\rightarrow{P^1}\rightarrow{P^0}\twoheadrightarrow{M}
\end{equation*} such that each $P^i$ lives only in degrees $\geq i$, i.e. $P^i=\bigoplus_{j\geq i}P^i_j$. We examine the complex $\textnormal{Hom}_{\textcat{grA-Mod}}(P^\bullet,N)$ which has objects
\begin{equation*}
    \textnormal{Hom}_{\textcat{grA-Mod}}(P^i,N)=\textnormal{Hom}_{\textcat{grA-Mod}}(\bigoplus_{j\geq i}P^i_j, N)
\end{equation*}Since $N$ is pure of weight $n$ we see that $\textnormal{Hom}_{\textcat{grA-Mod}}(\bigoplus_{j\geq i}P^i_j, N_{n})=0$ for $i>n$. Hence, 
\begin{equation*}
    \textnormal{Ext}_{\textcat{grA-Mod}}^i(M,N)=0
\end{equation*}for $i>n$.
\end{proof}

\begin{prop}
If $A$ is a Koszul monoid then, for any pure left $A$-module $M$ of weight $n$, we have $\textnormal{Ext}_{\textcat{grA-Mod}}^i(A_0,M)=0$ unless $i=n$.
\end{prop}

\begin{proof}
Suppose that $A$ is a Koszul monoid. Then, by definition, $A_0$ admits a graded projective strict resolution of $A$-modules.
 \begin{equation*}
     \dots \rightarrow{P^2}\rightarrow{P^1}\rightarrow{P^0}\twoheadrightarrow{A_0}
 \end{equation*}with each $P^i$ generated by its degree $i$ component over $A_0$. We see that 
 $\textnormal{Ext}_{\textcat{grA-Mod}}^i(A_0,M)$ is the $i^{th}$ cohomology object of the complex $\textnormal{Hom}_{\textcat{grA-Mod}}(P^\bullet,M)$. But, since $P^i$ is generated by its degree $i$ component and $M$ is pure of weight $n$, so $M=M_n$, we see that all the terms in the complex $\textnormal{Hom}_{\textcat{grA-Mod}}(P^\bullet,M)$ are zero other than $\textnormal{Hom}_{\textcat{grA-Mod}}(P^n,M)$. Thus, $\textnormal{Ext}_{\textcat{grA-Mod}}^i(A_0,M)=0$ unless $i=n$. 
\end{proof}

\begin{lem}\label{exthomkernel}
Suppose that there exists a graded projective strict exact sequence of $A$-modules 
\begin{equation*}
   P^i\xrightarrow{d^i}{\dots}\rightarrow{P^1}\xrightarrow{d^1}{P^0}\xrightarrow{d^0}{M}\rightarrow{0}
\end{equation*}such that $P^i$ is generated by its degree $i$ component over $A_0$. Then, if we let \begin{equation*}
K^i:=\textnormal{Ker}(d^i:P^i\twoheadrightarrow{\textnormal{Im}(d^i)})\end{equation*} we have that 
\begin{equation*}
    \textnormal{Ext}^{i+1}_{\textcat{grA-Mod}}(M,N)=\textnormal{Hom}_{\textcat{grA-Mod}}(K^i,N)
\end{equation*}for any pure $N\in\textcat{grA-Mod}$.
\end{lem}

\begin{proof}
By dimension shifting, we see that there is an exact sequence
\begin{equation*}
    0\rightarrow\textnormal{Hom}(\textnormal{Im}(d^i),N)\rightarrow{\textnormal{Hom}(P^i,N)}\rightarrow{\textnormal{Hom}(K^i,N)}\rightarrow{\textnormal{Ext}^{i+1}(M,N)}\rightarrow{0}
\end{equation*}Therefore, 
\begin{equation*}
    \textnormal{Ext}^{i+1}_{\textcat{grA-Mod}}(M,N)=\textnormal{Coker}(\textnormal{Hom}_{\textcat{grA-Mod}}(P^i,N)\rightarrow{\textnormal{Hom}_{\textcat{grA-Mod}}(K^i,N)})
\end{equation*}Now, since $P^i$ is generated by its degree $i$ component over $A_0$ and $K^i$ lives only in degrees $\geq i+1$, we see that $\textnormal{Hom}_{\textcat{grA-Mod}}(P^i,N)\rightarrow{\textnormal{Hom}_{\textcat{grA-Mod}}(K^i,N)}$ is the zero map, and hence
\begin{equation*}
    \textnormal{Ext}^{i+1}_{\textcat{grA-Mod}}(M,N)\simeq \textnormal{Hom}_{\textcat{grA-Mod}}(K^i,N)
\end{equation*}
\end{proof}

\begin{prop}\label{A0shiftkoszul}
{If $A$ is a pre-Koszul monoid} and $\textnormal{Ext}_{\textcat{grA-Mod}}^i(A_0,A_0\langle n\rangle)=0$ unless $i=n$, then $A$ is Koszul.
\end{prop}

\begin{proof}

We will construct a projective resolution $P^\bullet$ for $A_0$ using a similar method to Proposition \ref{propprojresolutionforM}. For $A$ to be Koszul, we want each of our modules $P^i$ to be generated in degree $i$ over $A_0$. As before, we can take $P^0=A$. Suppose that the strict resolution in question is constructed up to degree $i$
\begin{equation*}
    P^i\xrightarrow{d^i}{\dots}\rightarrow{P^1}\xrightarrow{d^1}{P^0}\xrightarrow{d^0}{A_0}\rightarrow{0}
\end{equation*}with each $P^j$ a graded projective $A$-module generated by its degree $j$ component over $A_0$ and with $d^j$ a strict monomorphism in degree $j$. We consider $K^{i}:=\textnormal{Ker}(d^{i}:P^{i}\twoheadrightarrow{\textnormal{Im}(d^{i})})$. This is a projective left $A_0$-module living only in degrees $\geq i+1$.  Since $P^i$ is generated in degree $i$ and $A_0\langle n\rangle$ is pure of weight $n$, we see that by Lemma \ref{exthomkernel},

\begin{equation*}
    \textnormal{Ext}^{i+1}_{\textcat{grA-Mod}}(A_0,A_0\langle n\rangle)=\textnormal{Hom}_{\textcat{grA-Mod}}(K^i,A_0\langle n\rangle)
\end{equation*}which, by assumption, is zero for $i+1\neq n$. Hence, since $K^{i}$ lives only in degree $\geq i+1$, $K^{i}$ is generated by its $({i+1})^{th}$ component because $A$ is pre-Koszul. We take $P^{i+1}$ to be the graded projective $A$-module $A\otimes_{A_0} K^{i}_{i+1}$. There is clearly a strict epimorphism 
\begin{equation*}
    A\otimes_{A_0} P^{i+1}_{i+1}=A\otimes_{A_0}A_0\otimes_{A_0} K^i_{i+1}\simeq A\otimes_{A_0} K^i_{i+1}\twoheadrightarrow{P^{i+1}}
\end{equation*}and therefore $P^{i+1}$ is generated by its $(i+1)^{th}$ component. It is simple to check that we have strict exactness in degree $i+1$.

\end{proof}

\begin{prop}\label{koszulcondition}If $A$ is a pre-Koszul monoid, then it is Koszul if and only if \begin{equation*}\textnormal{Ext}_{\textcat{grA-Mod}}^i(A_0,A_0\langle n\rangle)=0
\end{equation*}unless $i=n$.
\end{prop}
\begin{proof}
This is clear from the two propositions above. 
\end{proof}
\begin{defn}
If $A$ is a Koszul monoid, we define the \textbf{Koszul complex} to be the sequence 
\begin{equation*}
    \dots\rightarrow{P^3}\rightarrow{P^2}\rightarrow{P^1}
\end{equation*} constructed in Proposition \ref{A0shiftkoszul}. We note that this complex gives a resolution of $A_0$.
\end{defn}

\section{Quadratic Monoids}\label{quadraticmonoids}

Analogously to quadratic rings, we would like a concept of a quadratic monoid in our category $\mathcal{E}$.

\subsection{The Tensor Monoid}

We assert the existence of the tensor monoid using the following theorem.
\begin{thm}\cite[Theorem VII.3.2]{maclane98}The forgetful functor $\textcat{Mon}(\mathcal{E})\rightarrow{\mathcal{E}}$ has a left adjoint.
\end{thm}

\begin{defn}The \textbf{tensor monoid} $T(A)$ of a monoid $A$ in $\mathcal{E}$ is the graded monoid
\begin{equation*}
    T(A)=\bigoplus_{n=0}^\infty T_n(A)=\bigoplus_{n=0}^\infty A^{\otimes n}
\end{equation*}with multiplication $\mu:T(A)\otimes T(A)\rightarrow{T(A)}$ determined by the canonical isomorphism 
\begin{equation*}
    \mu_{i,j}:T_i(A)\otimes T_j(A)\rightarrow{T_{i+j}(A)}
\end{equation*}The unit is the inclusion 
\begin{equation*}
    \eta:T_0(V)\rightarrow{T(V)}
\end{equation*}
\end{defn}

\begin{remark}
If $A$ is a monoid and $M$ is an $(A,A)$-bimodule, we can define the tensor module $T_A(M)=A\oplus M\oplus(M\otimes_A M)\oplus\dots$.
\end{remark}

\begin{lem}
    If $T_{A_0}(A_1)$ is a pre-Koszul monoid, then it is Koszul. 
\end{lem}

\begin{proof}If $T_{A_0}(A_1)$ is pre-Koszul, then $A_1$ is projective over $A_0$ and hence $A_0$ admits a {projective resolution $P^\bullet$ as a $T_{A_0}(A_1)$-module}
\begin{equation*}
    0\rightarrow{T_{A_0}(A_1)\otimes_{A_0} A_1}\rightarrow{T_{A_0}(A_1)}\rightarrow{A_0}\rightarrow{0}
\end{equation*}with each $T_{A_0}(A_1)$-module $P^i$ projective and generated in degree $i$.
    
\end{proof}
{\begin{exmp}
Suppose that $R_0$ is a semisimple Banach algebra over a non-trivially valued field $k$ and suppose that $R_0$ is injective over itself as an element of $\textcat{CBorn}_k$, for example $R_0=\mathbb{C}$. Suppose further that $R_1$ is a projective $R_0$-module, e.g. $R_1=\ell^1(\mathbb{C}):=\{(c_i)_{i\in \mathbb{C}- \{0\}}\mid c_i\in\mathbb{C},\sum_{i\in \mathbb{C}- \{0\}}||c_i||<\infty\}$. Then, we can consider the tensor algebra $T_{R_0}(R_1)$ in $\textcat{CBorn}_k$. We note that $R_1$ is also projective as a module in $\textcat{CBorn}_k$. To show that $T_{R_0}(R_1)$ is pre-Koszul it suffices to prove the hom condition. Suppose that $M=\bigoplus_{k\geq i}``\varinjlim_{j\in J}"(M_j)_k$ is a graded $T_{R_0}(R_1)$-module in $\textcat{CBorn}_k$ living only in degrees $\geq i$ and that
\begin{equation*}
    \textnormal{Hom}_{\textcat{gr}T_{R_0}(R_1)\textcat{-Mod}(\textcat{CBorn}_k)}(M,R_0\langle n\rangle)=0
\end{equation*}unless $n=i$. Then, we see that, identifying $\textcat{CBorn}_k$ with the full subcategory of essentially monomorphic objects in $\textcat{IndBan}_k$ 
\begin{equation*}
    \bigoplus_{k\geq i}\varprojlim_{j\in J}\textnormal{Hom}_{T_{R_0}(R_1)\textcat{-Mod}(\textcat{Ban}_k)}((M_j)_k,R_0\langle n\rangle)=0
\end{equation*}unless $n=i$. Hence, we see that for each $k\geq i$, 
\begin{equation*}
    \varprojlim_{j\in J}\textnormal{Hom}_{T_{R_0}(R_1)\textcat{-Mod}(\textcat{Ban}_k)}((M_j)_k,R_0\langle n\rangle)=0
\end{equation*}unless $n=i$. Since $M\in\textcat{CBorn}_k$, each of the system morphisms in $M$ are monomorphisms and, hence, each of the morphisms in the system $(\textnormal{Hom}_{T_{R_0}(R_1)\textcat{-Mod}(\textcat{Ban}_k)}((M_j)_k,R_0\langle n\rangle))_{j\in J}$ are epimorphisms. Hence, if the inverse limit of these abelian groups is zero, we see that each $\textnormal{Hom}_{T_{R_0}(R_1)\textcat{-Mod}(\textcat{Ban}_k)}((M_j)_k,R_0\langle n\rangle)=0$ unless $i=n$. Since $R_0$ is semisimple, we therefore can use a similar reasoning to Proposition \ref{semisimpleprekoszul}, noting that by \cite[Lemma A.29]{benbassatkremnitzer17} a strict epimorphism of Banach spaces is equivalently a surjection, to show that for all $j\in J$ there is a strict epimorphism 
\begin{equation*}
    T_{R_0}(R_1)\otimes_{R_0}(M_j)_i\rightarrow{M_j}
\end{equation*}in $T_{R_0}(R_1)\textcat{-Mod}(\textcat{Ban}_k)$. Hence, we see that, by \cite[Proposition 2.10]{bambozzibenbassat15} there is a strict epimorphism 
\begin{equation*}
    T_{R_0}(R_1)\otimes_{R_0}``\varinjlim_{j\in J}" (M_j)_i\simeq ``\varinjlim_{j\in J}"T_{R_0}(R_1)\otimes_{R_0}(M_j)_i\rightarrow{``\varinjlim_{j'\in J}"M_{j'}}
\end{equation*}in $\textcat{gr}T_{R_0}(R_1)\textcat{-Mod}(\textcat{CBorn}_k)$.  
\end{exmp}}

If $A$ is positively graded, we can consider $A_1$ as an $(A_0,A_0)$-bimodule. There is a canonical morphism 
\begin{equation*}
    \pi:T_{A_0}(A_1)\rightarrow{A}
\end{equation*}formed by `linearly' extending the multiplication 
\begin{equation*}
    \mu_i:A_1^{\otimes_{A_0}i}\rightarrow{A}
\end{equation*}for all $i\in\mathbb{N}$. We make the following definition.
\begin{defn}\label{quadraticmonoid}
We say that $A$ is a \textbf{quadratic monoid} with quadratic data $(A_1,R)$ if $A$ is pre-Koszul, and there exists a strict graded epimorphism 
\begin{equation*}
    \pi:T_{A_0}(A_1)\twoheadrightarrow{A}
\end{equation*}such that there exists a strict epimorphism 
\begin{equation*}
    T_{A_0}(A_1)\otimes_{A_0}R\otimes_{A_0}T_{A_0}(A_1)\twoheadrightarrow{\textnormal{Ker}(\pi)}
\end{equation*}with $R=K_2:=\textnormal{Ker}(A_1\otimes_{A_0} A_1\twoheadrightarrow{A_2})$.
\end{defn}

\begin{remark}
We see that $A\simeq \textnormal{Coker}(\textnormal{Ker}(\pi)\rightarrow{T_{A_0}(A_1))}$ with $\textnormal{Ker}(\pi)$ generated by $R$. By some abuse of notation, we denote this quadratic monoid by 
\begin{equation*}
    A=T_{A_0}(A_1)/(R)
\end{equation*}We note that $A$ is in some sense generated by $A_1$ over $A_0$ with relations of degree two.
\end{remark}
\begin{exmp}
If the tensor monoid $T_{A_0}(A_1)$ is pre-Koszul, then it is quadratic. We may take $R=0$ and see that there clearly exists a strict epimorphism $\pi:T_{A_0}(A_1)\rightarrow{T_{A_0}(A_1)}$ with zero kernel. 
\end{exmp}

\begin{lem}Let $K=\textnormal{Ker}(\pi)$. Then, $K$ is graded as a left $A_0$-module by 
\begin{equation*}
    K_i=\textnormal{Ker}(\mu_i:A_1^{\otimes_{A_0} i}\rightarrow{A_i})
\end{equation*}and there exists a strict epimorphism 
\begin{equation*}
    \bigoplus_{j=0}^{i-2}A_1^{\otimes j}\otimes_{A_0} R\otimes_{A_0} A_1^{\otimes i-j-2}\twoheadrightarrow{K_i}
\end{equation*}Moreover, for each $i\geq 0$, 
\begin{equation*}
    A_i\simeq \textnormal{Coker}(K_i\rightarrow{A_1^{\otimes i}})
\end{equation*} 
\end{lem}

We now fix a positively graded pre-Koszul monoid $A$. We recall that there is a short strictly exact sequence
\begin{equation*}
    0\rightarrow{A_{>0}}\xrightarrow{\iota}{A}\xrightarrow{\pi}{A_0}\rightarrow{0}
\end{equation*}
\begin{defn}
$A$ is a \textbf{quotient of $T_{A_0}(A_1)$} if there exists a strict epimorphism 
\begin{equation*}
    T_{A_0}(A_1)\twoheadrightarrow{A}
\end{equation*}
\end{defn}

\begin{lem}\label{quotientlemma}
If $A$ is a quotient of $T_{A_0}(A_1)$, then there exists a strict epimorphism 
\begin{equation*}
    A\otimes_{A_0}A_1\twoheadrightarrow{A_{>0}}
\end{equation*}given, in each degree $i>0$, by the $(A_0,A_0)$-bimodule map
\begin{equation*}
    \lambda_{i-1,1}:A_{i-1}\otimes_{A_0}A_1\rightarrow{A_i}
\end{equation*}
\end{lem}

\begin{proof}
It suffices to show that $\lambda_{i-1,1}$ is a strict epimorphism for each $i$. We see that the following diagram commutes
\begin{equation*}\begin{tikzcd}
   & A_{i-1}\otimes_{A_0}A_1 \arrow{dr}{\lambda_{i-1,1}}\\
    A_1^{\otimes_{A_0}i} \arrow{rr}{\mu_i} \arrow{ur} & & A_i
\end{tikzcd}\end{equation*} and hence, since $\mu_i$ is a strict epimorphism, $\lambda_{i-1,1}$ is also a strict epimorphism by Proposition \ref{monoepimorphismtriangle}.
\end{proof}

\begin{remark}
We see that $A_{>0}$ is generated by $A_1$ over $A_0$.
\end{remark}

\begin{prop}For any pure $M\in\textcat{grA-Mod}$, $\textnormal{Ext}^1_{\textcat{grA-Mod}}(A_0,M)=\textnormal{Hom}_{\textcat{grA-Mod}}(A_{>0},M)$
\end{prop}

\begin{proof}
We consider the graded projective strict exact sequence of $A$-modules $A\rightarrow{A_0}\rightarrow{0}$. We note that $A$ is generated by its degree $0$ component over $A_0$ and $A_{>0}=\textnormal{Ker}(A\rightarrow{A_0})$. Therefore, by Lemma \ref{exthomkernel}, we see that 
\begin{equation*}
    \textnormal{Ext}^1_{\textcat{grA-Mod}}(A_0,M)=\textnormal{Hom}_{\textcat{grA-mod}}(A_{>0},N)
\end{equation*}
\end{proof}

\begin{cor}\label{quotientoftensor}
If $\textnormal{Ext}^1_{\textcat{grA-Mod}}(A_0,A_0\langle n\rangle)=0$ unless $n=1$, then  there exists a strict epimorphism $A\otimes_{A_0}A_1\rightarrow{A_{>0}}$.
Moreover, $A$ is a quotient of $T_{A_0}(A_1)$.
\end{cor}

\begin{proof}
Consider $A_0\langle n\rangle$ as a pure graded $A$-module of weight $n$. Using the previous proposition, we note that
\begin{equation*}
    \textnormal{Ext}^1_{\textcat{grA-Mod}}(A_0,A_0\langle n\rangle)=\textnormal{Hom}_{\textcat{grA-Mod}}(A_{>0},A_0\langle n\rangle)
\end{equation*} and hence
\begin{equation*}
    \textnormal{Hom}_{\textcat{grA-Mod}}(A_{>0},A_0\langle n\rangle)=0
\end{equation*}unless $n=1$. Then, since $A$ is pre-Koszul, $A_{>0}$ is generated by its component in degree $1$, so there exists a strict epimorphism \begin{equation*}
    A\otimes_{A_0}A_1\twoheadrightarrow{A_{>0}}
\end{equation*}Hence, for each $i$, there exists a strict epimorphism 
\begin{equation*}
    A_{i-1}\otimes_{A_0}A_1\twoheadrightarrow{A_i}
\end{equation*}For each $i>0$, we can construct a chain of strict epimorphisms
\begin{equation*}
    A_1^{\otimes i}=(A_1\otimes_{A_0}A_1)\otimes_{A_0}A_1^{\otimes i-2}\twoheadrightarrow{A_2\otimes A_1^{\otimes i-2}}\twoheadrightarrow{\dots}\twoheadrightarrow{A_i}
\end{equation*}and hence there exists a strict epimorphism 
\begin{equation*}
    T_{A_0}(A_1)\twoheadrightarrow{A}
\end{equation*}
\end{proof}

\begin{prop}\label{ext2quadratic}
Suppose that $A$ is a quotient of $T_{A_0}(A_1)$. If $\textnormal{Ext}_{\textcat{grA-Mod}}^2(A_0,A_0\langle n\rangle)=0$ unless $n=2$, then $A$ is a quadratic monoid. 
\end{prop}

\begin{proof}
It suffices to show that there is a strict epimorphism
\begin{equation*}
    T_{A_0}(A_1)\otimes_{A_0}R\otimes_{A_0}T_{A_0}(A_1)\twoheadrightarrow{\textnormal{Ker}(\pi)}
\end{equation*}where $R=K_2$. Let $K=\textnormal{Ker}(\pi)$. We note that, since the $0^{th}$ and $1^{st}$ components of $T_{A_0}(A_1)$ are $A_0$ and $A_1$ respectively, then $K$ only exists in degree $\geq 2$. We have a strictly exact sequence
\begin{equation*}
    0\rightarrow{K}\xrightarrow{\iota}{T_{A_0}(A_1)}_{>0}\xrightarrow{\pi}{A}\rightarrow{A_0}\rightarrow{0}
\end{equation*}where $T_{A_0}(A_1)$ is a projective $A$-module generated by its degree $1$ component over $A_0$. By Lemma \ref{exthomkernel}, 
\begin{equation*}
    \textnormal{Ext}^2_{\textcat{grA-Mod}}(A_0, A_0\langle n\rangle)=\textnormal{Hom}_{\textcat{grA-Mod}}(K,A_0\langle n\rangle)
\end{equation*}
Hence, $\textnormal{Hom}_{\textcat{grA-Mod}}(K,A_0\langle n\rangle)=0$ unless $n=2$. Since $A$ is pre-Koszul, there exists a strict epimorphism 
\begin{equation*}
    A\otimes_{A_0}K_2\twoheadrightarrow{K}
\end{equation*}Since there exists a strict epimorphism $T_{A_0}(A_1)\twoheadrightarrow{A}$ and strict epimorphisms are stable under composition by Proposition \ref{strictepisstable}, there is a strict epimorphism 
\begin{equation*}\begin{aligned}
T_{A_0}(A_1)\otimes_{A_0} K_2\otimes_{A_0} T_{A_0}(A_1)\twoheadrightarrow{A\otimes_{A_0} K_2\otimes_{A_0}A}\twoheadrightarrow{A\otimes_{A_0}K_2}\twoheadrightarrow{K}
\end{aligned}
\end{equation*}where the second step follows by Proposition \ref{strictmodulemap}. 
\end{proof}
The key result of this section is the following.
\begin{cor}
Any Koszul monoid is quadratic.
\end{cor}

\begin{proof}
Suppose $A$ is a Koszul monoid. Then, by Proposition \ref{koszulcondition}, $\textnormal{Ext}^i_{\textcat{grA-Mod}}(A_0,A_0\langle n\rangle)=0$ unless $n=i$. Hence, $\textnormal{Ext}^1_{\textcat{grA-Mod}}(A_0,A_0\langle n\rangle)=0$ unless $n=1$. Therefore, by Corollary \ref{quotientoftensor}, $A$ is a quotient of $T_{A_0}(A_1)$. Moreover, since $\textnormal{Ext}^2_{\textcat{grA-Mod}}(A_0,A_0\langle n\rangle)=0$ unless $n=2$, then by Proposition \ref{ext2quadratic}, $A$ is a quadratic monoid.
\end{proof}

\section{Dual Quadratic Monoids}\label{dualquadraticmonoids}

\subsection{Dual Objects}

Suppose that $A$ is a positively graded monoid in $\mathcal{E}$ and that $M$ is a left  $A_0$-module. We want to define the notion of a dual $A_0$-module, $M^*$. Indeed there are various notions of what a dual object should be in an arbitrary monoidal category. We make the following definition which explicitly constructs $M^*$ as an $A_0$-module. Our theory is closely related to the definition of a dual object from \cite[Section 2.10]{etingof15}. Indeed our definition of the dual of a dualisable $A_0$-module corresponds exactly to their dual object in the monoidal category $\textcat{A}_0\textcat{-Mod}$.

\begin{defn}The \textbf{left dual $A_0$-module $M^*$} is defined to be 
\begin{equation*}
    M^*:=\texthom{Hom}_{\textcat{A}_0\textcat{-Mod}}(M,A_0)
\end{equation*}If $M$ is instead a right $A_0$-module, the \textbf{right dual $A_0$-module} ${}^*M$ is defined to be
\begin{equation*}
    {}^*M:=\texthom{Hom}_{\textcat{Mod-A}_0}(M,A_0)
\end{equation*}
\end{defn}
\begin{remark}
We note that $A_0^*=\texthom{Hom}_{\textcat{A}_0\textcat{-Mod}}(A_0,A_0)\simeq A_0$ by Lemma \ref{usefulfacts}.
\end{remark}
If $M$ is graded, we may define a grading on $M^*$ by
\begin{equation*}
    (M^*)_i=\texthom{Hom}_{\textcat{A}_0\textcat{-Mod}}(M_{-i},A_0)=(M_{-i})^*
\end{equation*}By the internal hom adjunction, for any left $A_0$-module $M$ there is an isomorphism 
\begin{equation*}
    \begin{aligned}\textnormal{Hom}_{\textcat{A}_0\textcat{-Mod}}(M^*\otimes_{A_0}M,A_0)&\simeq \textnormal{Hom}_{\textcat{A}_0\textcat{-Mod}}(M^*,\texthom{Hom}_{\textcat{A}_0\textcat{-Mod}}(M,A_0))\\
    &=\textnormal{Hom}_{\textcat{A}_0\textcat{-Mod}}(M^*,M^*)
\end{aligned}\end{equation*}

\begin{defn}
We define the \textbf{evaluation morphism} $ev_M:M^*\otimes_{A_0}M\rightarrow{A_0}$ to be the image of $id_{M^*}$ under the isomorphism 
\begin{equation*}
    \textnormal{Hom}_{\textcat{A}_0\textcat{-Mod}}(M^*,M^*)\simeq \textnormal{Hom}_{\textcat{A}_0\textcat{-Mod}}(M^*\otimes_{A
    _0}M,A_0)
\end{equation*}
\end{defn}

\begin{defn}
We say that an $(A_0,A_0)$-bimodule $M$ is \textbf{left dualisable} if ${}^*(M^*)\simeq M$ and there exists a coevalution morphism $coev_M:A_0\rightarrow{M\otimes_{A_0}M^*}$ such that the compositions
\begin{equation*}\label{equation1}
    M\xrightarrow{coev_M\otimes_{A_0} id_M}{(M\otimes_{A_0} M^*)\otimes_{A_0} M}\xrightarrow{}{M\otimes_{A_0}(M^*\otimes_{A_0} M)}\xrightarrow{id_M\otimes_{A_0} ev_M}{M}
\end{equation*}and
\begin{equation*}\label{equation2}
    M^*\xrightarrow{id_{M^*}\otimes_{A_0} coev_M}{M^*\otimes_{A_0} (M\otimes_{A_0} M^*)}\xrightarrow{}{(M^*\otimes_{A_0} M)\otimes_{A_0} M^*}\xrightarrow{ev_M\otimes_{A_0} id_{M^*}}{M^*}
\end{equation*}are the identity morphisms. 
\end{defn}

\begin{remark}
We remark that if $M$ is an $(A_0,A_0)$-bimodule, then $M^*$ and ${}^*M$ are $(A_0,A_0)$-bimodules.
\end{remark}

\begin{remark}
We say that an $(A_0,A_0)$-bimodule $M$ is \textbf{right dualisable} if $({}^*M)^*\simeq M$ and there exists a coevaluation morphism $coev_M:A_0\rightarrow{{}^*M\otimes_{A_0}M }$ satisfying similar conditions to above, with the evaluation morphism defined to be $ev_M:M\otimes_{A_0} {}^*M\rightarrow{A_0}$.
\end{remark}

\begin{exmp}
We note that when $A_0=k$ is a field, the dualisable modules are precisely the finite dimensional vector spaces. When $A_0$ is a semisimple ring, the dualisable modules are precisely the finitely generated ones.
\end{exmp}

We now prove a few important propositions.

\begin{prop}\label{aboveproposition}
Suppose that $M$ is a left dualisable $(A_0,A_0)$-bimodule, $N_1$ is an $(A_0,A_0)$-bimodule and $N_2$ is a left $A_0$-module, then we have that 
\begin{equation*}
    \texthom{Hom}_{\textcat{A}_0\textcat{-Mod}}(N_1\otimes_{A_0} M,N_2)\simeq \texthom{Hom}_{\textcat{A}_0\textcat{-Mod}}(N_1,M^*\otimes_{A_0}N_2)
\end{equation*}
\end{prop}
\begin{proof}
This isomorphism is induced by the isomorphism 
\begin{equation*}
    N_1\simeq N_1\otimes_{A_0} A_0\xrightarrow{id_{N_1}\otimes_{A_0} coev_{M}} N_1\otimes_{A_0}M\otimes_{A_0}M^*\rightarrow{N_2\otimes_{A_0} M^*}\xrightarrow{s_{N_2,M^*}}{M^*\otimes_{A_0} N_2}
\end{equation*}

\end{proof}

\begin{prop}\label{dualtensorhomprop}
For a left dualisable $(A_0,A_0)$-bimodule $M$ and a left $A_0$-module $N$, we have 
\begin{equation*}
    M^*\otimes_{A_0}N\simeq\texthom{Hom}_{\textcat{A}_0\textcat{-Mod}}(M,N)
\end{equation*}
\end{prop}

\begin{proof}Indeed, by Propositions \ref{aboveproposition} and \ref{usefulfacts}, and the internal hom adjunction,
\begin{equation*}
    \begin{aligned}
    \texthom{Hom}_{\textcat{A}_0\textcat{-Mod}}(M,N)
    &\simeq \textnormal{Hom}_{\textcat{A}_0\textcat{-Mod}}(A_0,\texthom{Hom}_{\textcat{A}_0\textcat{-Mod}}(M,N))\\
    &\simeq \textnormal{Hom}_{\textcat{A}_0\textcat{-Mod}}(A_0\otimes_{A_0}M,N)\\
    &\simeq \textnormal{Hom}_{\textcat{A}_0\textcat{-Mod}}(A_0,M^*\otimes_{A_0} N)\\
    &\simeq M^*\otimes_{A_0} N
    \end{aligned}
\end{equation*}
\end{proof}
\begin{cor}\label{rightdualisableiso}
For a right dualisable $(A_0,A_0)$-bimodule $M$ and a left $A_0$-module $N$ we have 
\begin{equation*}
    M\otimes_{A_0} N\simeq\texthom{Hom}_{\textcat{A}_0\textcat{-Mod}}({}^*M,N)
\end{equation*}
\end{cor}
\begin{proof}
Indeed, since ${}^*M$ is a left dualisable $(A_0,A_0)$-bimodule, applying the previous proposition we obtain
\begin{equation*}
    M\otimes_{A_0} N\simeq ({}^*M)^*\otimes_{A_0} N\simeq\texthom{Hom}_{\textcat{A}_0\textcat{-Mod}}({}^*M,N)
\end{equation*}
\end{proof}

\begin{cor}\label{tensorduals}
Suppose that $M$ is a left $(A_0,A_0)$-bimodule and $N$ is a left dualisable $(A_0,A_0)$-bimodule. Then 
\begin{equation*}
    (M\otimes_{A_0} N)^*\simeq N^*\otimes_{A_0} M^*
\end{equation*}Further, if $M$ is left dualisable, then $M\otimes_{A_0} N$ is left dualisable.
\end{cor}
\begin{proof}
We have that, by Proposition \ref{tensorhomadjunctioncor}.
\begin{align*}
    (M\otimes_{A_0}N)^*:=\texthom{Hom}_{\textcat{A}_0\textcat{-Mod}}(M\otimes_{A_0}N,A_0)&\simeq\texthom{Hom}_{\textcat{A}_0\textcat{-Mod}}(N,\texthom{Hom}_{\textcat{A}_0\textcat{-Mod}}(M,A_0))\\
    &\simeq \texthom{Hom}_{\textcat{A}_0\textcat{-Mod}}(N,M^*)\\
    \intertext{and since $N\simeq {}^*(N^*)$ then, by Proposition \ref {aboveproposition},}
    &\simeq N^*\otimes_{A_0}M^*
\end{align*}Now suppose that $M$ is dualisable. We define the coevaluation \begin{equation*}
    coev_{M\otimes_{A_0}N}:A_0\rightarrow{(M\otimes_{A_0}N)\otimes_{A_0} (M\otimes_{A_0} N)^*}
\end{equation*} as the composition
\begin{equation*}\begin{aligned}
    A_0\xrightarrow{coev_M}{M\otimes_{A_0} M^*}\simeq & M\otimes_{A_0} A_0\otimes_{A_0} M^*
    \xrightarrow{id_M\otimes_{A_0}coev_N\otimes_{A_0}id_{M^*}}{}\\
    &\rightarrow{M\otimes_{A_0} N\otimes_{A_0} N^*\otimes_{A_0}M^*}\simeq (M\otimes_{A_0} N)\otimes_{A_0} (M\otimes_{A_0} N)^*
\end{aligned}\end{equation*}We can easily check that this is compatible with the evaluation morphism
\begin{equation*}
    ev_{M\otimes_{A_0} N}:(M\otimes_{A_0} N)^*\otimes_{A_0} (M\otimes_{A_0} N)\rightarrow{A_0}
\end{equation*}which can be written as the composition
\begin{equation*}\begin{aligned}
    (M\otimes_{A_0} N)^*\otimes_{A_0} (M\otimes_{A_0} N)\simeq &N^*\otimes_{A_0} (M^*\otimes_{A_0}M)\otimes_{A_0} N\xrightarrow{id_{N^*}\otimes ev_M\otimes_{A_0} id_N}{}\\
    &\rightarrow{N^*\otimes_{A_0} N}\xrightarrow{ev_N}A_0
\end{aligned}\end{equation*}
\end{proof}

\begin{prop}
If $M=\bigoplus_{i\in I}M_i$ is a finite direct sum of left dualisable $(A_0,A_0)$-bimodules then 
\begin{equation*}
    M^*\simeq\bigoplus_{i\in I}M^*_i
\end{equation*}
\end{prop}

\begin{proof}We note that, since the internal hom functor preserves finite coproducts in $\textcat{A}_0\textcat{-Mod}$,
\begin{equation*}\begin{aligned}
    M^*:=\texthom{Hom}_{\textcat{A}_0\textcat{-Mod}}(M,A_0)&=\texthom{Hom}_{\textcat{A}_0\textcat{-Mod}}(\bigoplus_{i\in I}M_i,A_0)\\
    &\simeq\bigoplus_{i\in I}\texthom{Hom}_{\textcat{A}_0\textcat{-Mod}}(M_i,A_0)\\
    &=\bigoplus_{i\in I}M_i^*
\end{aligned}\end{equation*}
\end{proof}

\begin{defn}
Suppose that $M$ is a left dualisable $(A_0,A_0)$-bimodule and $N$ is a left $A_0$-module. If $f:M\rightarrow{N}$ is a left $A_0$-module map we define the \textbf{left dual map} $f^*:N^*\rightarrow{M^*}$ to be the map 
\begin{equation*}\begin{aligned}
    N^*\xrightarrow{id_{N^*}\otimes_{A_0}coev_M}{N^*\otimes (M\otimes_{A_0} M^*)}\xrightarrow{a_{N^*,M,M^*}}{(N^*\otimes_{A_0} M)\otimes_{A_0}M^*}\\
    \xrightarrow{(id_{N^*}\otimes_{A_0} f)\otimes_{A_0} id_{M^*}}{(N^*\otimes_{A_0} N)\otimes_{A_0} M^*}
    \xrightarrow{ev_N\otimes_{A_0}id_{M^*}}{M^*}
\end{aligned}\end{equation*}
\end{defn}

We note that this map makes the following diagrams commute
\begin{equation*} 
    \begin{tikzcd}
    N^*\otimes_{A_0} M \arrow{d}{id_{N^*}\otimes_{A_0} f} \arrow{rr}{f^*\otimes_{A_0}id_M} && M^*\otimes_{A_0} M\arrow{d}{ev_M}\\
    N^*\otimes_{A_0}N \arrow{rr}{ev_N} & & A_0
    \end{tikzcd}\quad \begin{tikzcd}
    A_0 \arrow{d}{coev_N} \arrow{rr}{coev_M} && M\otimes_{A_0} M^*\arrow{d}{f\otimes_{A_0} id_{M^*}}\\
    N\otimes_{A_0}N^* \arrow{rr}{id_N\otimes_{A_0} f^*} & & N\otimes_{A_0} M^*
    \end{tikzcd}
\end{equation*}

\begin{remark}
We remark that if $f$ is strict, then so is $f^*$, being a composition of strict maps. 
\end{remark}
We state without proof the following easy propositions. 
\begin{prop}
Suppose that $N$ and $M$ are left dualisable. If $f:M\rightarrow{N}$ is a left $A_0$-module map, then ${}^*(f^*)\simeq f$.
\end{prop}

We will call any $f$ satisfying this condition \textit{left dualisable}. Similarly, any $f$ such that $({}^*f)^*\simeq f$ is \textit{right dualisable}. 
\begin{prop}
Suppose that $f:M\rightarrow{N}$ and $g:L\rightarrow{M}$ are left dualisable $A_0$-module maps between $(A_0,A_0)$-bimodules, then $(f\circ g)^*=g^*\circ f^*$.
\end{prop}

\begin{prop}\label{monodualepi}
Suppose that $f:M\rightarrow{N}$ is a left dualisable epimorphism of $(A_0,A_0)$-bimodules, then $f^*:N^*\rightarrow{M^*}$ is a monomorphism on right dualisable maps. Similarly, if $f$ is a monomorphism, then $f^*$ is an epimorphism on right dualisable maps.
\end{prop}
\begin{proof}
Suppose that $f$ is an epimorphism, and that we have two right dualisable maps $g_1,g_2:L\rightarrow{N^*}$ of $(A_0,A_0)$-bimodules such that $f^*\circ g_1=f^*\circ g_2$. Then, ${}^*(f^*\circ g_1)={}^*(f^*\circ g_2)$, and hence ${}^*g_1\circ f={}^*g_2\circ f$. Since $f$ is an epimorphism, ${}^*g_1={}^*g_2$, and hence $g_1\simeq({}^*g_1)^*=({}^*g_2)^*\simeq g_2$. Therefore, $f^*$ is a monomorphism. The other claim follows similarly.
\end{proof}

\begin{prop}\label{dualkercoker2}Suppose that $A_0$ is injective as a module over itself. Suppose that we have a map $f:M\rightarrow{N}$ between a left dualisable $(A_0,A_0)$-bimodule $M$ and a left $A_0$-module $N$. Then, \begin{equation*}
    \textnormal{Ker}(f^*)\simeq \textnormal{Coker}(f)^*\quad\text{and}\quad \textnormal{Ker}(f)^*\simeq \textnormal{Coker}(f^*)
\end{equation*}
\end{prop}
\begin{proof}
Consider the strictly exact sequence 
\begin{equation*}
    0\rightarrow{\textnormal{Ker}(f)}\rightarrow{M}\xrightarrow{f}{N}\rightarrow{\textnormal{Coker}(f)}\rightarrow{0}
\end{equation*}If we apply the functor $\texthom{Hom}_{\textcat{A}_0\textcat{-Mod}}(-,A_0)$, which is strictly exact by Proposition \ref{homstrictlyexactoriginal}, we obtain the strictly exact sequence
\begin{equation*}
    0\rightarrow{\textnormal{Coker}(f)^*}\rightarrow{N^*}\xrightarrow{f^*}{M^*}\rightarrow{\textnormal{Ker}(f)^*}\rightarrow{0}
\end{equation*}and so our result follows.
\end{proof}
For the rest of this subsection we will assume that $A_0$ is injective as a module over itself.  
\begin{prop}\label{dualkercoker}
If $N$ is instead a right $A_0$-module and $M$ is right dualisable, then $\textnormal{Ker}({}^*f)\simeq{{}^*\textnormal{Coker}(f)}$ and ${}^*\textnormal{Ker}(f)\simeq{\textnormal{Coker}({}^*f)}$.
\end{prop}

\begin{cor}\label{kerneldualisable}
If $f:M\rightarrow{N}$ is a map between left dualisable $(A_0,A_0)$-bimodules $M$ and $N$, then $\textnormal{Ker}(f)$ is a left dualisable $(A_0,A_0)$-bimodule.
\end{cor}

\begin{proof}
 We first note that ${}^*(\textnormal{Ker}(f)^*)\simeq {}^*\textnormal{Coker}(f^*)\simeq \textnormal{Ker}({}^*(f^*))\simeq \textnormal{Ker}(f)$ using Propositions \ref{dualkercoker2} and \ref{dualkercoker}. Let $K=\textnormal{Ker}(f)$. We define a coevaluation map $coev_K:A_0\rightarrow{K\otimes_{A_0}K^*}$ as follows. We note that $K^*\simeq \textnormal{Coker}(f^*)$ by Proposition \ref{dualkercoker}. Denote by $\iota$ the inclusion map $K\hookrightarrow{M}$ and by $\pi$ the cokernel map $M^*\rightarrow{\textnormal{Coker}(f^*)}\simeq K^*$. Consider the composition 
\begin{equation*}
    A_0\xrightarrow{coev_{M}}{M\otimes_{A_0} M^*}\xrightarrow{id_M\otimes_{A_0} \pi}{M\otimes_{A_0} \textnormal{Coker}(f^*)}
\end{equation*} We note that the following diagram commutes
\begin{equation*}\begin{tikzcd}
    A_0\arrow{d}{coev_N}\arrow{r}{coev_{M}} & {M\otimes_{A_0} M^*} \arrow{d}{f\otimes_{A_0} id_{M^*}} \arrow{r}{id_M\otimes_{A_0} \pi} & {M\otimes_{A_0} \textnormal{Coker}(f^*)} \arrow{d}{f\otimes_{A_0}id_{K^*}}\\
    N\otimes_{A_0} N^* \arrow{r}{id_N\otimes_{A_0} f^*}& N\otimes_{A_0} M^* \arrow{r}{id_N\otimes_{A_0}\pi} & N\otimes_{A_0} \textnormal{Coker}(f^*)
\end{tikzcd}\end{equation*}
and hence we see that
\begin{equation*}
    (f\otimes_{A_0} id_{K^*})\circ (id_M\otimes_{A_0} \pi)\circ coev_M=(id_N\otimes_{A_0} \pi)\circ (id_N\otimes_{A_0} f^*)\circ coev_N=0
\end{equation*}since $\pi\circ f^*=0$. Therefore, by the universal property of the kernel, there exists a map \begin{equation*}coev_K:A_0\rightarrow{\textnormal{Ker}(f)\otimes_{A_0}}\textnormal{Coker}(f^*)\simeq K\otimes_{A_0} K^*\end{equation*} such that $(\iota\otimes_{A_0}id_{K^*})\circ coev_K=(id_{M}\otimes_{A_0}\pi)\circ coev_M $. We note that the evaluation map $ev_K:K^*\otimes_{A_0} K\rightarrow{A_0}$ exists and satisfies $ev_K\circ (\pi\otimes_{A_0}id_K)=ev_M\circ (id_{M^*}\otimes_{A_0}\iota)$. We check that this coevaluation is compatible with the evaluation. Indeed, 

\begin{align*}
    &\iota\circ (id_K\otimes_{A_0} ev_K)\circ a_{K,K^*,K}\circ (coev_K\otimes_{A_0} id_K)\\
    &=(\iota\otimes_{A_0} ev_K)\circ a_{K,K^*,K}\circ (coev_K\otimes_{A_0} id_K)\\
    &=(id_M\otimes_{A_0}ev_K)\circ a_{M,K^*,K}\circ((\iota\otimes_{A_0} id_{K^*})\otimes_{A_0}id_K)\circ  (coev_K\otimes_{A_0} id_K)\\
    \intertext{and since $(\iota\otimes_{A_0}id_{K^*})\circ coev_K=(id_{M}\otimes_{A_0}\pi)\circ coev_M$, we have}
    &=(id_M\otimes_{A_0}ev_K)\circ a_{M,K^*,K}\circ ((id_M\otimes_{A_0} \pi)\otimes_{A_0} id_K)\circ (coev_M\otimes_{A_0} id_K)\\
    \intertext{and since $ev_K\circ (\pi\otimes_{A_0} id_K)=ev_M\circ (id_{M^*}\otimes_{A_0} \iota)$,}
    &=(id_M\otimes_{A_0} ev_M)\circ a_{M,M^*,M}\circ(id_M\otimes_{A_0}(id_{M^*}\otimes_{A_0}\iota))\circ(coev_M\otimes_{A_0} id_K)\\
    &=(id_M\otimes_{A_0} ev_M)\circ a_{M,M^*,M}\circ(coev_M\otimes_{A_0} id_M)\circ \iota\\
    \intertext{using the compatibility of $ev_M$ and $coev_M$,}
    &=id_M\circ \iota\\
    &=\iota\circ id_K
\end{align*}Therefore, as $\iota$ is a monomorphism, we get that
\begin{equation*}
 (id_K\otimes_{A_0} ev_K)\circ a_{K,K^*,K}\circ (coev_K\otimes_{A_0} id_K)=id_K    
\end{equation*}
The other compatibility condition follows similarly. Hence, $\textnormal{Ker}(f)$ is left dualisable.
\end{proof}

\begin{cor}\label{cokerdualisable}
If $f:M\rightarrow{N}$ is a map between left dualisable $(A_0,A_0)$-bimodules $M$ and $N$, then $\textnormal{Coker}(f)$ is a left dualisable $(A_0,A_0)$-bimodule.
\end{cor}

\begin{proof}
This follows using a similar proof to the previous proposition.
\end{proof}

\subsection{Dual Quadratic Monoids}
\begin{defn}
We will say that a positively graded monoid $A$ is \textbf{left dualisable} if each $A_i$ is a left dualisable $(A_0,A_0)$-bimodule. 
\end{defn}

Suppose that $A$ is a left dualisable quadratic monoid with quadratic data $(A_1,R)$. We note that $R=K_2$ is a left dualisable $(A_0,A_0)$-bimodule by Corollary \ref{kerneldualisable}. If we dualise the strict monomorphism 
\begin{equation*}
    \iota:R\hookrightarrow{A_1\otimes_{A_0}A_1}
\end{equation*}we obtain a map 
\begin{equation*}
    \iota^*:A_1^*\otimes_{A_0}A_1^*\simeq (A_1\otimes_{A_0}A_1)^*\rightarrow{R^*}
\end{equation*}
\begin{defn}
The \textbf{left orthogonal $A_0$-submodule $R^\perp$} is the kernel of the map \begin{equation*}
    \iota^*:A_1^*\otimes_{A_0}A_1^*\rightarrow{R^*}
\end{equation*}
\end{defn}

\begin{defn}
Suppose that $A$ is a left dualisable quadratic monoid $(A_1,R)$. We say that a positively graded pre-Koszul monoid $A^!$ is the \textbf{left dual quadratic monoid} of $A$ if 
\begin{equation*}
    A^!\simeq T_{A_0}(A_1^*)/(R^\perp)
\end{equation*}where we use the notation of Definition \ref{quadraticmonoid}. 
\end{defn}

\begin{remark}
In this paper, we will not explore under which conditions $A^!$ exists and is pre-Koszul. It may be more tempting to define, for a Koszul monoid $A$, the dual Koszul monoid to be $A^!=\texthom{Ext}_{\textcat{A}_0\textcat{-Mod}}(A_0,A_0)$ but we will not explore this approach in this paper.
\end{remark}

Spelling this out, there exists a strict epimorphism $\pi^!:T_{A_0}(A_1^*)\twoheadrightarrow{A^!}$ with a strict epimorphism 
\begin{equation*}
    T_{A_0}(A_1^*)\otimes_{A_0}R^\perp\otimes_{A_0} T_{A_0}(A_1^*)\twoheadrightarrow{\textnormal{Ker}(\pi^!)}
\end{equation*}
\begin{exmp}
If we once again consider the quadratic tensor monoid $T_{A_0}(A_1)$, we see that $R^\perp=\textnormal{Ker}(A_1^*\otimes_{A_0}A_1^*\rightarrow{0})\simeq A_1^*\otimes_{A_0} A_1^*$. Therefore, \begin{equation*}
    A^!=\textnormal{Coker}(T_{A_0}(A_1^*)\otimes_{A_0} R^\perp\otimes_{A_0}T_{A_0}(A_1^*)\hookrightarrow{T_{A_0}(A_1^*)})\simeq A_0\oplus A_1^*
\end{equation*}
\end{exmp}

If $A$ is instead a right dualisable quadratic monoid, we can define the \textbf{right orthogonal $A_0$-submodule} ${}^\perp R$ to be the kernel of the map 
\begin{equation*}
    {}^*\iota:{}^*A_1\otimes_{A_0}{}^*A_1\rightarrow{{}^*R}
\end{equation*}The \textbf{right dual quadratic monoid} is then a positively graded pre-Koszul monoid with presentation 
\begin{equation*}
    {}^*A\simeq T_{A_0}({}^*A_1)/({}^\perp R)
\end{equation*}

Now, suppose that $A$ is a left dualisable quadratic monoid with left dual quadratic monoid $A^!$. 
\begin{lem}\label{Risomorphism}We have ${}^\perp (R^\perp)\simeq R$.
\end{lem}\begin{proof}
Indeed, 
\begin{align*}
    {}^\perp(R^\perp)&=\textnormal{Ker}({}^*(A_1^*)\otimes_{A_0}{}^*(A_1^*)\rightarrow{{}^*(R^\perp)})\\
    &\simeq \textnormal{Ker}(A_1\otimes_{A_0}A_1\rightarrow{{}^*(R^\perp)})\\
    &= \textnormal{Ker}(A_1\otimes_{A_0}A_1\rightarrow{{}^*\textnormal{Ker}(A_1^*\otimes_{A_0}A_1^*\rightarrow{R^*}}))\\
    \intertext{Now, by Proposition \ref{dualkercoker},}
    &\simeq \textnormal{Ker}(A_1\otimes_{A_0}A_1\rightarrow{\textnormal{Coker}(R\hookrightarrow{A_1\otimes_{A_0}A_1)}})\\
    &\simeq \textnormal{Im}(R\hookrightarrow{A_1\otimes_{A_0}A_1})\\
    &\simeq R
\end{align*}
\end{proof}

\begin{prop}\label{A!!}
${}^!(A^!)\simeq A$.
\end{prop}

\begin{proof}
This follows immediately using that $A_1$ is dualisable and Lemma \ref{Risomorphism}.

\end{proof}

\begin{lem}
$A_1^!\simeq A_1^*$.
\end{lem}

\begin{proof}
We note that $K_1^!=0$, and hence the strict epimorphism $A_1^*\rightarrow{A_1^!}$ is an isomorphism. 
\end{proof}
    We have the following results, similar to the ones stated in the previous section.
\begin{lem}
We let $K^!=\textnormal{Ker}(\pi^!)$. Then, $K^!$ is graded as a left $A_0$-module by 
\begin{equation*}
    K_i^!=\textnormal{Ker}((A_1^*)^{\otimes i}\rightarrow{A_i^!})
\end{equation*}and there exists a strict epimorphism 
\begin{equation*}
    \bigoplus_{j=0}^{i-2}(A_1^*)^{\otimes j}\otimes_{A_0} R^\perp\otimes_{A_0}(A_1^*)^{\otimes (i-j-2)}\twoheadrightarrow{K_i^!}
\end{equation*}Moreover, for each $i\geq 0$, 
\begin{equation*}
    A_i^!\simeq\textnormal{Coker}(K_i^!\rightarrow{(A_1^*)^{\otimes i}})
\end{equation*}
\end{lem}
\begin{prop}
$A^!$ is right dualisable.
\end{prop}

\begin{proof}We will show that each $A_i^!$ is a right dualisable $(A_0,A_0)$-bimodule. Indeed, we note that $(A_1^*)^{\otimes i}$ is right dualisable using Proposition \ref{tensorduals} and the fact that $A_1$ is left dualisable. We can apply Corollary \ref{kerneldualisable} to show that $R^\perp$, being the kernel of right dualisable objects, is right dualisable. Hence, since $A^!_i$ is isomorphic to a cokernel of right dualisable objects, it is right dualisable by Corollary \ref{cokerdualisable}.
\end{proof}
\begin{prop}\label{kernelperpA!}
We have 
\begin{equation*}
    {}^*(A_i^!)\simeq {}^\perp(K_i^!)\end{equation*}
\end{prop}

\begin{proof}
Indeed, using Proposition \ref{dualkercoker},
\begin{align*}
    {}^*(A_i^!)&\simeq{}^*\textnormal{Coker}(K_i^!\rightarrow{(A_1^*)^{\otimes i}})\\
    &\simeq \textnormal{Ker}(A_1^{\otimes i}\rightarrow{{}^*(K_i^!)})\\
    &\simeq {}^\perp(K_i^!)
\end{align*}

\end{proof}

\begin{prop}\label{pullbackiso}
Define $P_i$ to be the pullback of all the monomorphisms 
\begin{equation*}
    A_1^{\otimes j}\otimes_{A_0} R\otimes_{A_0} A_1^{\otimes (i-j-2)}\hookrightarrow{A_1^{\otimes i}}
\end{equation*} where $j$ ranges from $0$ to $i-2$. Then, there exists an isomorphism 
\begin{equation*}
    P_i\simeq {}^\perp(K_i^!)\simeq {}^*(A_i^!)
\end{equation*}for each $i$.
\end{prop}
\begin{proof}
There exists a strict epimorphism of right dualisable $(A_0,A_0)$-bimodules
\begin{equation*}
    \bigoplus_{j=0}^{i-2} (A_1^*)^{\otimes j}\otimes_{A_0} R^\perp\otimes_{A_0} (A_1^*)^{\otimes (i-j-2)}\twoheadrightarrow{K_i^!}
\end{equation*}Dualising this map we obtain, by Proposition \ref{monodualepi}, a strict monomorphism
\begin{equation*}
    {}^*(K_i^!)\hookrightarrow{\bigoplus_{j=0}^{i-2}A_1^{\otimes j}\otimes_{A_0} {}^*(R^\perp)\otimes_{A_0} A_1^{\otimes (i-j-2)}}
\end{equation*} We know that ${}^\perp(K_i^!)=\textnormal{Ker}(A_1^{\otimes i}\rightarrow{{}^*(K_i^!))}$. Hence, 
\begin{equation*}
    {}^\perp(K_i^!)=\textnormal{Ker}\bigg({A_1^{\otimes i}\rightarrow{\bigoplus_{j=0}^{i-2}A_1^{\otimes j}\otimes_{A_0} {}^*(R^\perp)\otimes_{A_0} A_1^{\otimes (i-j-2)}}}\bigg)
\end{equation*}To show that $P_i\simeq{}^\perp(K_i^!)$, it suffices to show that $P_i$ is also the kernel of this map. Now, since ${}^*(R^\perp)\simeq \textnormal{Coker}(\iota:R\hookrightarrow{A_1\otimes_{A_0}A_1})$ we see that the map 
\begin{equation*}
    P_i\rightarrow{A_1^{\otimes i}}\rightarrow{\bigoplus_{j=0}^{i-2} A_1^{\otimes j}\otimes_{A_0}{}^*(R^\perp)\otimes_{A_0}A_1^{\otimes (i-j-2)}}
\end{equation*}is zero since, by definition of the pullback, it factors through $A_1^{\otimes j}\otimes_{A_0} R\otimes_{A_0} A_1^{\otimes (i-j-2)}$ for all $0\leq j\leq i-2$. 

Now, suppose that we have an object $W$ such that the map 
\begin{equation*}
    W\rightarrow{A_1^{\otimes i}}\rightarrow{\bigoplus_{j=0}^{i-2} A_1^{\otimes j}\otimes_{A_0} {}^*(R^\perp)\otimes_{A_0} A_1^{\otimes (i-j-2)}}
\end{equation*}is zero. We see that, since the composite $R\hookrightarrow{A_1\otimes_{A_0}A_1\rightarrow{A_2}}$ is zero, there exists a map ${}^*(R^\perp)\rightarrow{A_2}$ such that the following diagram commutes
\begin{equation*}
      \begin{tikzcd}
       A_1\otimes_{A_0}A_1 \arrow{r} \arrow{dr} &{}^*(R^\perp) \arrow{d}\\
        &A_2
    \end{tikzcd}
\end{equation*}and hence the following diagram commutes
\begin{equation*}
    \begin{tikzcd}
     A_1^{\otimes i} \arrow{r} \arrow{dr} & A_1^{\otimes j}\otimes_{A_0}{}^*(R^\perp)\otimes_{A_0} A_1^{\otimes (i-j-2)} \arrow{d}\\
        & A_1^{\otimes j}\otimes_{A_0}A_2\otimes_{A_0} A_1^{\otimes (i-j-2)}
\end{tikzcd}
\end{equation*}The map 
\begin{align*}
    W&\rightarrow{A_1^{\otimes i}}\rightarrow{ A_1^{\otimes j}\otimes_{A_0}A_2\otimes_{A_0}A_1^{\otimes (i-j-2)}}\\
\intertext{is therefore zero since it factors through the zero map}
    W&\rightarrow{A_1^{\otimes j}\otimes_{A_0} {}^*(R^\perp)\otimes_{A_0} A_1^{\otimes (i-j-2)}}
\intertext{Since $R= K_2=\textnormal{Ker}(A_1\otimes_{A_0}A_1\twoheadrightarrow{A_2})$, there exists a map}
    W&\rightarrow{A_1^{\otimes j}\otimes_{A_0} R\otimes_{A_0} A_1^{\otimes (i-j-2)}}
\end{align*}such that the following diagram commutes
\begin{equation*}
    \begin{tikzcd}
     A_1^{\otimes j}\otimes_{A_0} R\otimes_{A_0} A_1^{\otimes (i-j-2)} \arrow{r} & A_1^{\otimes i}\\
     W\arrow{u}\arrow{ur}
    \end{tikzcd}
\end{equation*}Hence, since, for each $0\leq j\leq i-2$, all the maps
\begin{equation*}
    W\rightarrow{A_1^{\otimes j}\otimes_{A_0}R\otimes_{A_0} A_1^{\otimes (i-j-2)}}\rightarrow{A_1^{\otimes i}}
\end{equation*}are equal, then by definition of the pullback there exists a map $W\rightarrow{P_i}$ such that the following diagram commutes for all $0\leq j\leq i-2$,
\begin{equation*}
    \begin{tikzcd}[column sep = 12]
     A_1^{\otimes j}\otimes_{A_0} R\otimes_{A_0} A_1^{\otimes (i-j-2)} \arrow{r} & A_1^{\otimes i} \arrow{r} & \bigoplus_{j=0}^{i-2} A_1^{\otimes j}\otimes_{A_0} {}^*(R^\perp)\otimes_{A_0} A_1^{\otimes (i-j-2)}\\
    P_i\arrow{u} & W\arrow{l}\arrow{u}
    \end{tikzcd}
\end{equation*}Hence, $P_i$ is a kernel of the map $A_1^{\otimes i}\rightarrow{\bigoplus_{j=0}^{i-2}A_1^{\otimes j}\otimes_{A_0} {}^*(R^\perp)\otimes_{A_0}A_1^{\otimes (i-j-2)}}$ so, by the uniqueness of the kernel
\begin{equation*}
    P_i\simeq {}^\perp(K_i^!)
\end{equation*}
\end{proof}

\section{The Koszul Complex}\label{koszulcomplexsection}

Suppose that $A$ is a left dualisable quadratic monoid $(A_1,R)$. Suppose further that the dual quadratic monoid exists, and denote it by $A^!$. For each $i\geq 0$, we let $\mathcal{K}^i$ be the $A$-module
\begin{equation*}
    \mathcal{K}^i=A\otimes_{A_0} {}^*(A_i^!)
\end{equation*}We note that this module lives only in degree $\geq i$. 

\begin{prop}$\mathcal{K}^i$ is a projective $A$-module.
\end{prop}
\begin{proof}We want to show that the functor
\begin{equation*}
    \textnormal{Hom}_{\textcat{A-Mod}}(A\otimes_{A_0} {}^*(A_i^!),-):\textcat{A-Mod}\rightarrow{\textcat{Ab}}
\end{equation*}maps strict epimorphisms to surjections. Indeed, we note that $A_i^!$ is right dualisable, and hence $({}^*(A_i^!))^*\simeq A_i^!$. Therefore, by Theorem \ref{tensorhomadjunction} and Proposition \ref{dualtensorhomprop}, 
\begin{equation*}\begin{aligned}
    \textnormal{Hom}_{\textcat{A-Mod}}(A\otimes_{A_0} {}^*(A_i^!),-)&\simeq \textnormal{Hom}_{\textcat{A-Mod}}(A, \texthom{Hom}_{\textcat{A}_0\textcat{-Mod}}({}^*(A_i^!)\otimes_{A_0},-))\\
    &\simeq\textnormal{Hom}_{\textcat{A-Mod}}(A, A_i^!\otimes_{A_0}-)
\end{aligned}\end{equation*}We note that $A_i^!\otimes_{A_0}-$ maps strict epimorphisms to strict epimorphisms since it is right exact. The claim follows since $A$ is a projective $A$-module, and therefore $\textnormal{Hom}_{\textcat{A-Mod}}(A,-)$ maps strict epimorphisms to surjections.

\end{proof}

We consider the map 
\begin{equation*}
    \mu_{i,1}^!:A_i^!\otimes_{A_0} A_1^!\rightarrow{A_{i+1}^!}
\end{equation*}which is the multiplication map on $A^!$. Dualising, we obtain a map 
\begin{equation*}
    {}^*\mu_{i,1}^!:{}^*(A_{i+1}^!)\rightarrow{{}^*(A_1^!)\otimes_{A_0} {}^*(A_i^!)}\simeq A_1\otimes_{A_0} {}^*(A_i^!)
\end{equation*}When tensored with $A$, this gives a map 
\begin{equation*}\begin{aligned}
    \mathcal{K}^{i+1}=A\otimes_{A_0} {}^*(A_{i+1}^!)&\xrightarrow{id_A\otimes_{A_0}{}^*\mu_{i,1}^!}{A\otimes_{A_0}A_1\otimes_{A_0} {}^*(A_i^!)}\\
    &\xrightarrow{\mu\otimes_{A_0} id_{{}^*(A_i^!)}}{{A\otimes_{A_0} {}^*(A_i^!)}}=\mathcal{K}^i
\end{aligned}\end{equation*}which we will denote by $d^{i+1}:\mathcal{K}^{i+1}\rightarrow{\mathcal{K}^i}$. Being a composition of strict maps, it is strict. We let
\begin{equation*}
    \begin{aligned}
    Z^i&=\textnormal{Ker}(d^i:\mathcal{K}^i\rightarrow{\mathcal{K}^{i-1}})\\
    B^i&=\textnormal{Im}(d^{i+1}:\mathcal{K}^{i+1}\rightarrow{\mathcal{K}^i})\end{aligned}
\end{equation*}

\begin{lem}\label{differentialszero}We have $d^i\circ d^{i+1}=0$.
\end{lem}
\begin{proof}
{By Proposition \ref{pullbackiso}, we have that ${}^*(A_i^!)$ is isomorphic to the pullback of all the monomorphisms 
\begin{equation*}
    A_1^{\otimes j}\otimes_{A_0} R\otimes_{A_0} A_1^{\otimes (i-j-2)}\hookrightarrow{A_1^{\otimes i}}
\end{equation*}where $j$ ranges from $0$ to $i-2$. In particular, we can identify the morphism ${}^*\mu_{i,1}^!:{}^*(A_{i+1}^!)\rightarrow{A_1\otimes_{A_0}{}^*(A_i^!)}$ with the morphism $P_{i+1}\rightarrow{A_1\otimes_{A_0} P_i}$ induced by the pullback. Similarly, the morphism $id_{A_1}\otimes_{A_0}{}^*\mu_{i-1,1}^!$
can be identified with the morphism $A_1\otimes_{A_0} P_i\rightarrow{A_1\otimes_{A_0}A_1\otimes_{A_0} P_{i-1}}$. In particular, using the properties of the pullback, we note that we have the following commutative diagram.}
\begin{equation*}
\begin{tikzcd}
        A\otimes_{A_0}P_{i+1} \arrow{r} \arrow{d}{id_A\otimes_{A_0}{}^*\mu_{i,1}^!} & A\otimes_{A_0} R\otimes_{A_0} A_1^{\otimes i-1} \arrow{dd}\\
        A\otimes_{A_0} A_1\otimes_{A_0} P_i \arrow{d}{id_A\otimes_{A_0} id_{A_1}\otimes_{A_0}{}^*\mu_{i-1,1}^!}\\
        A\otimes_{A_0} A_1\otimes_{A_0}A_1\otimes_{A_0} P_{i-1} \arrow{r} \arrow{d}{id_A\otimes_{A_0}\mu_2\otimes_{A_0} id_{{}^*(A_{i-1}^!)}} & A\otimes_{A_0} A_1^{\otimes i+1} \arrow{d}{id_A\otimes_{A_0} \mu_2\otimes_{A_0} id_{A_1}^{\otimes i-1}}\\
        A\otimes_{A_0} P_{i-1}\arrow{r} & A\otimes_{A_0} A_1^{\otimes i-1}
    \end{tikzcd}
\end{equation*}{We see that the left vertical morphism is precisely the composition $d^i\circ d^{i+1}$. Since $R=\textnormal{Ker}(\mu_2)$, then the right vertical morphism is zero. Hence, since the bottom horizontal morphism is a strict monomorphism, being a pullback of strict monomorphisms, then we see that $d^i\circ d^{i+1}=0$.}
\end{proof}

\begin{defn}The \textbf{Koszul complex} $\mathcal{K}^\bullet$ of $A$ is the complex
\begin{equation*}
    \mathcal{K}^\bullet=\dots A\otimes_{A_0} {}^*(A_2^!)\rightarrow{A\otimes_{A_0} {}^*(A_1^!)}\rightarrow{A}
\end{equation*}with differentials $d^{i+1}:\mathcal{K}^{i+1}\rightarrow{\mathcal{K}^i}$ defined as above.
\end{defn}

We note that the objects of the Koszul complex are graded with components 
\begin{equation*}
    \mathcal{K}^i_j=A_{j-i}\otimes_{A_0} {}^*(A_i^!)
\end{equation*}Each $\mathcal{K}^i$ is projective as an $A$-module and is generated by its degree $i$ component over $A_0$ since $\mathcal{K}^i_i=A_0\otimes_{A_0} {}^*(A_i^!)\simeq {}^*(A_i^!)$. We will show in this section that $A$ is a Koszul monoid if and only if this complex provides a resolution of $A_0$.

\begin{prop}
For all $i\geq 0$, the multiplication map $\mu_{i,1}^!:A_i^!\otimes_{A_0} A_1^!\rightarrow{A_{i+1}^!}$ is a strict epimorphism.
\end{prop}
\begin{proof}This follows directly from Lemma \ref{quotientlemma} since $A^!$ is a quotient of $T_{A_0}(A_1^*)$.
\end{proof}

\begin{cor}\label{dualmono}The right dual map ${}^*\mu_{i,1}^!:{}^*(A_{i+1}^!)\rightarrow{{}^*(A_i^!\otimes_{A_0}A_1^!)}\simeq A_1\otimes_{A_0} {}^*(A_i^!)$ is a strict monomorphism.
\end{cor}

\begin{proof}
This follows from Proposition \ref{monodualepi}.\end{proof}

\begin{prop}\label{strictmonodifferential}
The $(i+1)$-th component of the graded map $d^{i+1}:\mathcal{K}^{i+1}\rightarrow{\mathcal{K}^i}$ is a strict monomorphism.
\end{prop}

\begin{proof}
We note that 
\begin{equation*}
    \mathcal{K}^{i+1}_{i+1}\simeq A_0\otimes_{A_0} {}^*(A_{i+1}^!)\simeq {}^*(A_{i+1}^!)
\end{equation*}and 
\begin{equation*}
    \mathcal{K}_{i+1}^i\simeq A_1\otimes_{A_0}{}^*(A_i^!)
\end{equation*}Hence, the map $d_{i+1}^{i+1}$ is equivalent to the map
\begin{equation*}
    {}^*\mu_{i,1}^!:{}^*(A_{i+1}^!)\rightarrow{A_1\otimes_{A_0} {}^*(A_i^!)}
\end{equation*}which is a strict monomorphism by the previous corollary. 
\end{proof}
\begin{cor}\label{imageisopushout}
$B_{i+1}^i\simeq {}^*(A_{i+1}^!)$.
\end{cor}

\begin{cor}\label{zipositivedegree}
For each $i\geq 0$, $Z^i$ lives only in degrees $\geq i+1$.
\end{cor}
\begin{proof}
We note that $\mathcal{K}^i_j=0$ for $j<i$, and hence $Z_j^i=0$ for $j<i$. By Proposition \ref{strictmonodifferential}, $Z_i^i=0$, and our result follows.
\end{proof}

\begin{lem}\label{monoisotriangle}
If there exist monomorphisms $\iota:X\rightarrow{Y}$ and $\iota':Z\rightarrow{Y}$, along with maps $f:Z\rightarrow{X}$ and $g:X\rightarrow{Z}$ such that the following diagram\begin{equation*}
    \begin{tikzcd}
     X\arrow[r,hook,"\iota"] \arrow[d,shift left, "g"]& Y\\
     Z\arrow[u,shift left,"f"] \arrow[ur,hook,"\iota'"]
    \end{tikzcd}
    \end{equation*}commutes, then $f$ is an isomorphism. 

\end{lem}
\begin{proof}
Indeed we have that $\iota\circ f=\iota'$ and $\iota'\circ g=\iota$. Hence, $\iota\circ (f\circ g)=\iota'\circ g=\iota$ and, since $\iota$ is a monomorphism, we have that $f\circ g=id_X$. Similarly $\iota'\circ (g\circ f)=\iota\circ f=\iota'$ and, since $\iota'$ is a monomorphism, we have that $g\circ f=id_Z$.
\end{proof}
We note that 
\begin{equation*}
    Z_{i+1}^i=\textnormal{Ker}(d_{i+1}^i:A_1\otimes_{A_0} {}^*(A_i^!)\rightarrow{A_2\otimes_{A_0}{}^*(A_{i-1}^!))}
\end{equation*}\begin{prop}\label{kernelisopushout}
$Z_{i+1}^i\simeq {}^*(A_{i+1}^!)$.
\end{prop}

\begin{proof}We first note that, since $d^i_{i+1}\circ d^{i+1}_{i+1}=0$ by Proposition \ref{differentialszero}, there exists a map ${}^*(A_{i+1}^!)\rightarrow{Z_{i+1}^i}$ such that the following diagram commutes
\begin{equation*}
    \begin{tikzcd}
    Z^i_{i+1}\arrow[hook]{r} & A_1\otimes_{A_0} {}^*(A_i^!)\\
    {}^*(A_{i+1}^!)\arrow[ur,hook, "d^{i+1}_{i+1}"]\arrow{u}
    \end{tikzcd}
\end{equation*}

By Proposition \ref{pullbackiso}, ${}^*(A_{i+1}^!)\simeq P_{i+1}$. We will show that there exists a map $Z^i_{i+1}\rightarrow{P_{i+1}}$. We note that, by Corollary \ref{dualmono}, there exists a strict monomorphism ${}^*(A_i^!)\rightarrow{A_1\otimes_{A_0} {}^*(A_{i-1}^!)}$. The kernel of the map $A_1\otimes_{A_0} A_1\otimes_{A_0} {}^*(A_{i-1}^!)\rightarrow{A_2\otimes_{A_0} {}^*(A_{i-1}^!)}$ is exactly $R\otimes_{A_0} {}^*(A_{i-1}^!)$. Since the map 
\begin{equation*}
    Z_{i+1}^i\rightarrow{A_1\otimes_{A_0}A_1\otimes_{A_0} {}^*(A_{i-1}^!)}\rightarrow{A_2\otimes_{A_0} {}^*(A_{i-1}^!)}
\end{equation*} is exactly the zero map 
\begin{equation*}
    Z_{i+1}^i\rightarrow{A_1\otimes_{A_0} {}^*(A_i^!)}\xrightarrow{d^i_{i+1}}{A_2\otimes_{A_0} {}^*(A_{i-1}^!)}
\end{equation*}then, by the universal property of the kernel, there exists a map $Z_{i+1}^i\rightarrow{R\otimes_{A_0}{}^*(A_{i-1}^!)}$ such that the following diagram commutes
\begin{equation*}
    \begin{tikzcd}
    Z_{i+1}^i\arrow{r}\arrow{drr} & A_1\otimes_{A_0} {}^*(A_i^!) \arrow{r} \arrow[rr,bend left=20, "d^i_{i+1}"] & A_1\otimes_{A_0}A_1\otimes_{A_0}{}^*(A_{i-1}^!)\arrow{r} & A_2\otimes_{A_0} {}^*(A_{i-1}^!)\\
    & & R\otimes_{A_0} {}^*(A_{i-1}^!)\arrow{u}
    \end{tikzcd}
\end{equation*}

Using the definition of $P_i\simeq{}^*(A_i^!)$ from Proposition \ref{pullbackiso}, we see that, for all $0\leq j\leq i-2$, the maps
\begin{equation*}
    Z_{i+1}^i\rightarrow{A_1\otimes_{A_0} {}^*(A_i^!)}\rightarrow{A_1\otimes_{A_0}(A_1^{\otimes j}\otimes_{A_0} R\otimes_{A_0} A_1^{\otimes i-j-2}})\rightarrow{A_1^{\otimes (i+1)}}
\end{equation*}are all equal. Further, using the definition of ${}^*(A_{i-1}^!)\simeq P_{i-1}$ and the commutative diagram, for each $0\leq j\leq i-3$, the maps 
\begin{equation*}
    \begin{aligned}Z_{i+1}^i&\rightarrow{A_1\otimes_{A_0}A_1\otimes_{A_0} {}^*(A_{i-1}^!)}\\&\rightarrow{A_1\otimes_{A_0}A_1\otimes_{A_0}(A_1^{\otimes j}\otimes_{A_0} R\otimes_{A_0} A_1^{\otimes i-j-3}})\\&\rightarrow{A_1^{\otimes (i+1)}}
\end{aligned}\end{equation*}are equal to the map 
\begin{equation*}
    Z_{i+1}^i\rightarrow{R\otimes_{A_0} A_1^{\otimes (i-1)}}\rightarrow{A_1^{\otimes (i+1)}}
\end{equation*}Hence, for all $0\leq j\leq i+1$, the maps
\begin{equation*}
    Z_{i+1}^i\rightarrow{A_1^{\otimes j}\otimes_{A_0}R\otimes_{A_0} A_1^{\otimes (i+j-1)}}\rightarrow{A_1^{\otimes (i+1)}}
\end{equation*} are equal, and thus there exists a map $Z^i_{i+1}\rightarrow{P_{i+1}}\simeq {}^*(A_{i+1}^!)$. This clearly makes the following diagram commute
\begin{equation*}
   \begin{tikzcd}
    Z^i_{i+1}\arrow[hook]{r}\arrow{d} & A_1\otimes_{A_0} {}^*(A_i^!)\\
    {}^*(A_{i+1}^!)\arrow[ur,hook, "d^{i+1}_{i+1}"]
    \end{tikzcd}
\end{equation*}Therefore, by Lemma \ref{monoisotriangle}, we have an isomorphism $Z_{i+1}^i\simeq {}^*(A_{i+1}^!)$.

\end{proof}

\begin{thm} \label{koszulresolutiontheorem}Suppose that $A$ is a left dualisable quadratic monoid $(A_1,R)$ with dual quadratic monoid $A^!$. Then, $A$ is a Koszul monoid if and only if its Koszul complex is a resolution of $A_0$.
\end{thm}

\begin{proof}
If the Koszul complex of $A$ is a resolution of $A_0$ then clearly $A$ is a Koszul monoid since its Koszul complex consists of projective $A$-modules $\mathcal{K}^i$ generated by their $i^{th}$ components. 

On the other hand, suppose that $A$ is a Koszul monoid. We need to show that the complex $\mathcal{K}^\bullet\rightarrow{A_0}$ is strictly exact. We prove this inductively. Indeed, the map $\mathcal{K}^0=A\rightarrow{A_0}$ is the natural projection map which is a strict epimorphism. Its kernel is isomorphic to $A_{>0}$ which is isomorphic to the image of the map 
\begin{equation*}
    \mathcal{K}^1=A\otimes_{A_0} {}^*(A_1^!)\simeq A\otimes_{A_0} A_1\rightarrow{A=\mathcal{K}^0}
\end{equation*}by Lemma \ref{quotientlemma}. Now, suppose that the Koszul complex is strictly exact up to degree $i$. Then, by Lemma \ref{exthomkernel}
\begin{equation*}
    \textnormal{Ext}^{i+1}_{\textcat{grA-Mod}}(A_0,A_0\langle n\rangle)=\textnormal{Hom}_{\textcat{grA-Mod}}(Z^{i},A_0\langle n\rangle)
\end{equation*}By Proposition \ref{koszulcondition}, if $A$ is Koszul then $\textnormal{Ext}^{i+1}_{\textcat{grA-mod}}(A_0,A_0\langle n\rangle)=0$ unless $i+1=n$. Therefore, $\textnormal{Hom}_{\textcat{grA-Mod}}(Z^{i},A_0\langle n\rangle)=0$
unless $i+1=n$. Since $A$ is pre-Koszul, $Z^{i}$ is generated by its $(i+1)^{th}$ component. Therefore, there exists a strict epimorphism 
\begin{equation*}
    A\otimes_{A_0} Z^{i}_{i+1}\twoheadrightarrow{Z^{i}}
\end{equation*}Hence, by Propositions \ref{imageisopushout} and \ref{kernelisopushout}, there is a strict epimorphism \begin{equation*}
    \mathcal{K}^{i+1}=A\otimes_{A_0} {}^*(A_{i+1}^!)\simeq A\otimes_{A_0} Z^i_{i+1}\simeq A\otimes_{A_0} B_{i+1}^i\twoheadrightarrow{Z^i}
\end{equation*} and the following diagram commutes 
\begin{equation*}
    \begin{tikzcd}
 \mathcal{K}^{i+1} \arrow{rr}{d^{i+1}} \arrow[two heads]{dr} & &  {\mathcal{K}^i}\\
 & Z^i \arrow[hook]{ur}
    \end{tikzcd}
\end{equation*}By the universal property of the image $B^i$, there exists a unique map $B^i\rightarrow{Z^i}$ such that the following diagram commutes
\begin{equation*}
     \begin{tikzcd}
     \mathcal{K}^{i+1} \arrow{rr} \arrow[two heads]{dr} \arrow[two heads]{ddr} & &  {\mathcal{K}^i}\\
 & B^i  \arrow{d} \arrow[hook]{ur} \\
 & Z^i \arrow[hook]{uur}
\end{tikzcd}
\end{equation*}By Proposition \ref{monoepimorphismtriangle}, the map $B^i\rightarrow{Z^i}$ must be a strict epimorphism. Since $d^i\circ d^{i+1}=0$, we also know that there is a unique strict monomorphism $B^i\rightarrow{Z^i}$ such that the following diagram commutes
\begin{equation*}
    \begin{tikzcd}
    Z^i \arrow[hook]{r} & \mathcal{K}^i\\
    B^i \arrow[hook]{ur} \arrow[hook]{u}
    \end{tikzcd}
\end{equation*}By uniqueness, the two maps $B^i\rightarrow{Z^i}$ must be the same. Since this map is both a strict monomorphism and a strict epimorphism, it must be a strict isomorphism by Proposition \ref{isomorphismmonoepi}. Therefore $B^i\simeq Z^i$ and we are done.
\end{proof}

\begin{prop}\label{A!iskoszul}
If $A$ is a Koszul monoid, then $A^!$ is too. 
\end{prop}

\begin{proof}
If $A$ is Koszul, then the complex
\begin{equation*}
    \dots\rightarrow{A\otimes_{A_0}{}^*(A_2^!)}\rightarrow{A\otimes_{A_0}{}^*(A_1^!)}\rightarrow{A}\rightarrow{A_0}\rightarrow{0}
\end{equation*}is strictly exact by Theorem \ref{koszulresolutiontheorem}. Since our differentials are graded, for each $j> 0$ the complex
\begin{equation*}
    0\rightarrow{{}^*(A_j^!)}\rightarrow\dots\rightarrow{A_{j-2}\otimes_{A_0} {}^*(A_2^!)}\rightarrow{A_{j-1}\otimes_{A_0} {}^*(A_1^!)}\rightarrow{A_j}\rightarrow{0}
\end{equation*}is strictly exact. Taking the left dual of this complex we obtain the complex
\begin{equation*}
    0\leftarrow{{}^*(A_j^!)^*}\leftarrow\dots\leftarrow{(A_{j-2}\otimes_{A_0}{}^*(A_2^!))^*}\leftarrow{(A_{j-1}\otimes_{A_0} {}^*(A_1^!))^*}\leftarrow{A_j^*}\leftarrow{0}
\end{equation*}Equivalently, since $A^!$ is right dualisable, we have the complex
\begin{equation*}
   0\leftarrow{A_j^!}\leftarrow\dots\leftarrow{A_2^!\otimes_{A_0}  A_{j-2}^*}\leftarrow{A_1^!\otimes_{A_0}  A_{j-1}^*}\leftarrow{A_j^*}\leftarrow{0}
\end{equation*}which is strictly exact since the dual functor is exact by Proposition \ref{homstrictlyexactoriginal}. These sequences assemble to give the strictly exact sequence
\begin{equation*}
    \dots\rightarrow{A^!\otimes_{A_0} A_2^*}\rightarrow{A^!\otimes_{A_0}A_1^*}\rightarrow{A^!}\rightarrow{A_0}\rightarrow{0}
\end{equation*}The objects of the complex are all projective $A^!$-modules and the differentials are $A^!$-module morphisms. This complex gives the Koszul complex, $A^!\otimes_{A_0} A_\bullet^*$, of $A^!$ and provides a resolution of $A_0$. Hence, $A^!$ is a Koszul monoid.
\end{proof}

Suppose that $A$ is a monoid in $\mathcal{E}$. The symmetric group $\Sigma_n$ acts on $T_n(A)=A^{\otimes n}$ by permutation as follows. 
\begin{equation*}\begin{aligned}
    \sigma:T_n(A)&\rightarrow{T_n(A)}\\
    A^{(1)}\otimes_{A_0}\dots\otimes_{A_0} A^{(n)}&\rightarrow{A^{\sigma(1)}\otimes_{A_0}\dots\otimes_{A_0} A^{\sigma(n)}}
\end{aligned}\end{equation*} 
Let $S_n(A)$ be the coequalizer of all the maps $\sigma$, for $\sigma\in\Sigma_n$. Then, the \textit{symmetric monoid on $A$}, $S(A)$, can be defined to be \begin{equation*}
    S(A)=\bigoplus_{n\geq 0}S_n(A)
\end{equation*}Let $\bigwedge^nA$ be the coequalizer of all the maps $sgn(\sigma)\sigma$ for $\sigma\in\Sigma_n$. Then, the \textit{exterior monoid} on $A$, $\bigwedge A$, is defined to be \begin{equation*}
    \bigwedge A=\bigoplus_{n\geq 0} \bigwedge^n A
\end{equation*} 
If $A$ is a monoid and $M$ is an $(A,A)$-bimodule, we can easily define the symmetric module $S_A(M)$ by defining $S_A(M)_n$ to be the coequalizer of all the maps $\sigma:T_A(M)_n\rightarrow{T_A(M)_n}$. We can similarly define the exterior module $\bigwedge_A M$. 
{\begin{exmp}Suppose that $\mathcal{E}$ is enriched over the category of $\mathbb{Q}$-vector spaces and that $T_{A_0}(A_1)$ is a pre-Koszul monoid. Then, we can easily show that there is a map $q_n:S_{A_0}(A_1)_n\rightarrow{T_{A_0}(A_1)_n}$ which is a section of the coequalizer map $\pi_n:T_{A_0}(A_1)_n\rightarrow{S_{A_0}(A_1)_n}$. Therefore, each $S_{A_0}(A_1)_n$ is a projective $A_0$-module. The hom-condition in the definition of pre-Koszul follows from the isomorphism
\begin{equation*}
\textnormal{Hom}_{\textcat{gr}S_{A_0}(A_1)\textcat{-Mod}}(M,A_0\langle n\rangle)\simeq \textnormal{Hom}_{\textcat{gr}T_{A_0}(A_1)\textcat{-Mod}}(T_{A_0}(A_1)\otimes_{S_{A_0}(A_1)}M,A_0\langle n\rangle)
\end{equation*}for every $S_{A_0}(A_1)$-module $M$. Similarly, we can also show that the exterior module $\bigwedge_{A_0}A_1$ is pre-Koszul. 
If $A$ is left dualisable, the quadratic dual of $S_{A_0}(A_1)$ is $\bigwedge_{A_0}A_1^*$. We consider the Koszul complex
\begin{equation*}
    \mathcal{K}^{\bullet}=\dots\rightarrow{S_{A_0}(A_1)\otimes_{A_0} \bigwedge_{A_0}^2 A_1}\xrightarrow{d^2}{S_{A_0}(A_1)\otimes_{A_0} A_1}\xrightarrow{d^1}{S_{A_0}(A_1)}\twoheadrightarrow{A_0}\rightarrow{0}
\end{equation*}where the differentials are given by $d^i:S_{A_0}(A_1)\otimes_{A_0}\bigwedge_{A_0}^iA_1\rightarrow{S_{A_0}(A_1)\otimes_{A_0}\bigwedge_{A_0}^{i-1}A_1}$ with graded part
\begin{equation*}
    d_j^i:S_{A_0}^{j-i}(A_1)\otimes_{A_0} \bigwedge_{A_0}^i A_1\rightarrow{S_{A_0}^{j-i+1}(A_1)\otimes_{A_0} \bigwedge_{A_0}^{i-1} A_1}
\end{equation*}given by the composition \begin{equation*}\begin{aligned}
    S_{A_0}^{j-i}(A_1)\otimes_{A_0} \bigwedge_{A_0}^i A_1&\xrightarrow{id_{S_{A_0}^{j-i}(A_1)}\otimes_{A_0} {}^*\mu_{i-1,1}^!}{S_{A_0}^{j-i}(A_1)\otimes_{A_0} A_1\otimes_{A_0} \bigwedge_{A_0}^{i-1}A_1}\\
    &\xrightarrow{\mu_{j-i,1}\otimes_{A_0} id_{\bigwedge_{A_0}^{i-1} A_1}}S_{A_0}^{j-i+1} (A_1)\otimes_{A_0} \bigwedge_{A_0}^{i-1} A_1
\end{aligned}\end{equation*}We then define maps 
\begin{equation*}
    h_j^i:S_{A_0}^{j-i}(A_1)\otimes_{A_0} \bigwedge_{A_0}^i A_1\rightarrow{S_{A_0}^{j-i-1}(A_1)\otimes_{A_0} \bigwedge_{A_0}^{i+1} A_1}
\end{equation*}for each $i,j$ by \begin{equation*}\begin{aligned}
    S_{A_0}^{j-i}(A_1)\otimes_{A_0} \bigwedge_{A_0}^i A_1 &\xrightarrow{{}^*\mu_{1,j-i-1}\otimes_{A_0} id_{\bigwedge_{A_0}^i } A_1}{S_{A_0}^{j-i-1}(A_1)\otimes_{A_0}A_1\otimes_{A_0} \bigwedge_{A_0}^{i} A_1}\\
    &\xrightarrow{id_{S_{A_0}^{j-i-1}(A_1)}\otimes_{A_0} \mu_{1,i}^!}{S_{A_0}^{j-i-1}(A_1)\otimes_{A_0} \bigwedge_{A_0}^{i+1} A_1}
\end{aligned}\end{equation*}and we can then show that 
\begin{equation*}
    h_{j}^{i-1}\circ d_j^i+d_{j}^{i+1}\circ h^i_j=i\circ id+(j-i)\circ id=j\circ id
\end{equation*}which, since we are working in characteristic $0$, defines a homotopy between the cochain maps $id,0:\mathcal{K}^\bullet_j\rightarrow{\mathcal{K}^\bullet_j}$. It then follows, by \cite[Proposition 2.2.39]{kelly19} that $\mathcal{K}^{\bullet}_j$ is exact for each $j$, and hence $\mathcal{K}_\bullet$ is exact. Since $S_{A_0}(A_1)$ has a Koszul complex which is a resolution of $A_0$, it is a Koszul monoid by Proposition \ref{koszulresolutiontheorem}. Similarly, we can show that $\bigwedge_{A_0}A_1$ is a Koszul monoid.
\end{exmp}}

\section{Our Main Koszul Duality Result}\label{koszuldualityresult}

Suppose that $A$ is a left dualisable quadratic monoid, with quadratic dual monoid $A^!$. We denote by $\textcat{C}(\textcat{grA-Mod})$ the category of cochain complexes $M^\bullet$ of graded $A$-modules $M^i=\bigoplus_jM^i_j$. We can consider $M^\bullet\in\textcat{C}(\textcat{grA-Mod})$ and $N^\bullet\in\textcat{C}(\textcat{grA}^!\textcat{-Mod})$ as modules $\bigoplus_{i}M^i$ and $\bigoplus_iN^i$ over $A$ and $A^!$ respectively. Define the full subcategories $\textcat{C}^\uparrow(\textcat{grA-Mod})$ and $\textcat{C}^\downarrow(\textcat{grA-Mod})$ as follows. 

If $M^\bullet\in\textcat{C}(\textcat{grA-Mod})$, with each $M^i=\bigoplus_{j} M^i_j$, then $M^\bullet\in\textcat{C}^\uparrow(\textcat{grA-Mod})$ if 
\begin{equation*}
    M_j^i=0\text{ for } i\gg 0 \text{ or }i+j\ll 0
\end{equation*}and $M^\bullet\in\textcat{C}^\downarrow(\textcat{grA-Mod})$ if
\begin{equation*}
    M_j^i=0\text{ for } i\ll 0 \text{ or }i+j\gg 0
\end{equation*}
The following diagram illustrates where the non zero components $M_j^i$ lie.
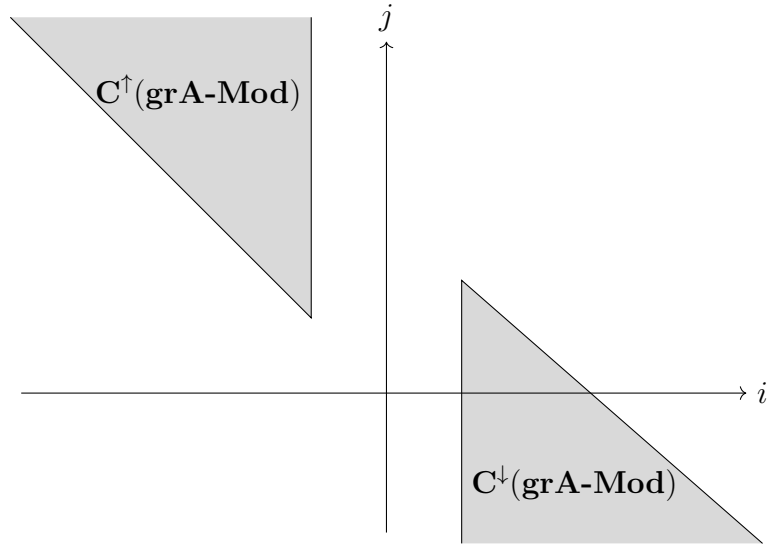
\begin{equation*}
    \begin{tikzpicture}
\node(a) at (-5,0) {} ;
\node(b) at (5,0) {$i$};
\node(c) at (0,5) {$j$};
\node(d) at (0,-2) {}; 
\draw[fill=gray!30]    (-1,5) -- (-1,1) -- (-5,5) ;
\draw[fill=gray!30]    (1,-2) -- (1,1.5) -- (5,-2) ;\draw[->] (a) -- (b); 
\draw[->] (d) -- (c);
\node(e) at (2.5,-1.2) {$\\\textcat{C}^\downarrow(\textcat{grA-Mod})$};
\node(f) at (-2.5,4) {$\\\textcat{C}^\uparrow(\textcat{grA-Mod})$};
\end{tikzpicture}
\end{equation*}

We let $\textcat{D}^\uparrow(\textcat{grA-Mod})$ and $\textcat{D}^\downarrow(\textcat{grA-Mod})$ denote the localisations of the categories $\textcat{C}^\uparrow(\textcat{grA-Mod})$ and $\textcat{C}^\downarrow(\textcat{grA-Mod})$ at strict quasi-isomorphisms. Note that these are triangulated subcategories of $\textcat{C}(\textcat{grA-Mod})$. 

Motivated by the Tensor-Hom adjunction for modules, we define functors
\begin{equation*}
\begin{split}
    F:\textcat{grA-Mod}\,&\leftrightarrow \textcat{grA}^!\textcat{-Mod}:G\\
    M^\bullet&\rightarrow{A^!\otimes_{A_0} M^\bullet}\\
    \textnormal{\underline{Hom}}_{\textcat{grA}_0\textcat{-Mod}}(A,N^\bullet)&\leftarrow{N^\bullet}
\end{split}
\end{equation*}We will show that these functors descend to triangulated functors
\begin{equation*}
    \textcat{D}F: \textcat{D}^\downarrow(\textcat{grA-Mod})\rightarrow{ \textcat{D}^\uparrow({\textcat{grA}}^!\textcat{-Mod})   }\,\text{ and } \, \textcat{D}G: \textcat{D}^\uparrow({\textcat{grA}}^!\textcat{-Mod})\rightarrow{ \textcat{D}^\downarrow(\textcat{grA}\textcat{-Mod})}
\end{equation*}inducing the following equivalence. 
\begin{thm}\label{bigmaintheorem}Suppose that $A$ is a left dualisable Koszul monoid with quadratic dual monoid $A^!$. Then, there is an equivalence of categories
\begin{equation*}
    \textcat{D}^\downarrow(\textcat{grA-Mod})\simeq{ \textcat{D}^\uparrow({\textcat{grA}}^!\textcat{-Mod})  }
\end{equation*}
\end{thm}

We prove this Theorem in several parts following closely the proof of \cite[Theorem 2.12.1]{beilinsonginzburg96}.
\begin{lem}\label{CFfunctor}There exists a functor $\textcat{C}F:\textcat{C}(\textcat{grA-Mod})\rightarrow{ \textcat{C}({\textcat{grA}}^!\textcat{-Mod}) }$.
\end{lem}
\begin{proof}Suppose that we have $(M^\bullet,d_M^\bullet)\in\textcat{C}(\textcat{grA-Mod})$. We consider $A^!$ as a complex $((A^!)^\bullet,d_{A^!}^\bullet)$ in $\textcat{C}({\textcat{grA}}^!\textcat{-Mod})$ and consider 
\begin{align*}
    F(M^\bullet)&=A^!\otimes_{A_0} M^\bullet\in \textcat{A}^!\textcat{-Mod}\\
\intertext{We construct the double complex of $A^!$ modules}
    F(M^\bullet)^{\bullet\bullet}&=A^!_\bullet\otimes_{A_0} M^\bullet
\intertext{with objects}
    F(M^\bullet)^{i,l}&=A^!_{l}\otimes_{A_0} M^i\\
\intertext{and graded by}
F(M^\bullet)^{i,l}_j&=A^!_{l}\otimes_{A_0} M^i_j
\end{align*}

The differentials are defined as follows. We first note that, since $A^!$ is right dualisable,
\begin{equation}\label{internalhomisoA!}
    A_{l}^!\otimes_{A_0} M^i_j\simeq\texthom{Hom}_{\textcat{A}_0\textcat{-Mod}}({}^*(A_{l}^!),M^i_j)
\end{equation}
Hence the $j^{th}$ component of the vertical differential 
\begin{equation*}
    (d_v^{i,l})_j:A_{l}^!\otimes_{A_0} M^i_j\rightarrow{A_{l+1}^!\otimes_{A_0} M^i_{j+1}}
\end{equation*}is $(-1)^{i+j}(D_v^{i,l})_j$ where $(D_v^{i,l})_j$ can be defined using Equation \ref{internalhomisoA!} and the composition 
\begin{equation*}
    {}^*(A_{l+1}^!)\xrightarrow{{}^*\mu_{l,1}^!}{A_1\otimes_{A_0} {}^*(A_{l}^!)}
    \rightarrow{A_1\otimes_{A_0} M^i_j}\rightarrow{M^i_{j+1}}
\end{equation*}We also remark that ${}^*\mu_{l,1}^!$ is one of the maps involved in the definition of the differential for the Koszul complex of $A$ given in Section \ref{koszulcomplexsection}. The $j^{th}$ component of the horizontal differential
\begin{equation*}
    (d_h^{i,l})_j:A_{l}^!\otimes_{A_0} M^i_j\rightarrow{A_{l}^!\otimes_{A_0} M^{i+1}_j}
\end{equation*}is defined by 
\begin{equation*}
    (d_h^{i,l})_j=id_{A^!_{l}}\otimes_{A_0} (d_M^i)_j
\end{equation*}It is easy to check that this is indeed a double complex. In degree $k$, we can visualise the double complex as follows
\begin{equation*}
    \begin{tikzcd}
   & \vdots & \vdots & \vdots \\
    \dots \arrow{r} & A_2^!\otimes_{A_0}M_{2-k}^{i}\arrow{r}\arrow{u} &  A_2^!\otimes_{A_0} M_{2-k}^{i+1}\arrow{u}\arrow{r} & A_2^!\otimes_{A_0} M_{2-k}^{i+2}\arrow{u}\arrow{r} & \dots\\
    \dots \arrow{r} & A_1^!\otimes_{A_0}M_{1-k}^{i} \arrow{r}\arrow{u} & A_1^!\otimes_{A_0} M_{1-k}^{i+1} \arrow{u}\arrow{r} & A_1^!\otimes_{A_0} M_{1-k}^{i+2}\arrow{u}\arrow{r} & \dots\\
     \dots \arrow{r} & M_{-k}^i \arrow{r}\arrow{u} &  M_{-k}^{i+1} \arrow{u}\arrow{r} &  M_{-k}^{i+2}\arrow{u}\arrow{r} & \dots\\
 & 0 \arrow{u}& 0 \arrow{u} & 0 \arrow{u}\\
    \end{tikzcd}
\end{equation*}We see that our double complex is bounded from below since $A^!$ is positively graded. We now consider a kind of `twisted' total complex $FM^\bullet$ with 
\begin{equation*}
    (FM)^n_k=\bigoplus_{i+j=n,k=l-j}A_{l}^!\otimes_{A_0} M^i_j
\end{equation*}and with associated total differential $d_{F}^\bullet=d_h^{\bullet\bullet}+d_v^{\bullet\bullet}$. We see that this complex is clearly a complex of $A^!$-modules.
We therefore have a functor
\begin{equation*}\begin{aligned}
 \textcat{C}F:\textcat{C}(\textcat{grA-Mod})&\rightarrow{ \textcat{C}({\textcat{grA}}^!\textcat{-Mod}) }\\
 M^\bullet &\rightarrow{(FM)^\bullet}
\end{aligned}\end{equation*}
\end{proof}

\begin{prop}\label{prop1}$\textbf{C}F$ induces a functor $\textcat{D}F: \textcat{D}^\downarrow(\textcat{grA-Mod})\rightarrow{ \textcat{D}^\uparrow({\textcat{grA}}^!\textcat{-Mod})   }$.
\end{prop}

\begin{proof}
Suppose that $M^\bullet\in\textcat{C}^\downarrow(\textcat{grA-Mod})$ and consider the functor\begin{equation*}\begin{aligned}
 \textcat{C}F:\textcat{C}(\textcat{grA-Mod})&\rightarrow{ \textcat{C}({\textcat{grA}}^!\textcat{-Mod}) }\\
 M^\bullet &\rightarrow{(FM)^\bullet}
\end{aligned}\end{equation*}defined in the previous Lemma. We now see why we specified the boundedness conditions on $M^\bullet$. If $n\gg 0$, then $i+j\gg 0$, and hence, since $M^\bullet\in\textcat{C}^\downarrow(\textcat{grA-Mod})$, we see that $M^i_j=0$. Now, if $n+k\ll 0$, then $i+j+k\ll 0$. Since $A^!$ is positively graded, we know that $l=j+k\geq 0$, and hence $i\ll 0$. Therefore, $M^i_j=0$ once again. We see that in both cases, $(FM)^n_k=0$, and hence $(FM)^\bullet\in\textcat{C}^\uparrow(\textcat{grA}^!\textcat{-Mod})$. We can therefore restrict $\textcat{C}F$ to a functor
\begin{equation*}
    \textcat{C}F: \textcat{C}^\downarrow(\textcat{grA-Mod})\rightarrow{ \textcat{C}^\uparrow({\textcat{grA}}^!\textcat{-Mod})   }
\end{equation*}To show that this functor induces the desired functor, it suffices, by Propositions \ref{inducedderived} and \ref{strictquasiexact}, to show that this functor preserves strict exactness. Indeed, suppose that $M^\bullet$ is a strictly exact complex. We note that, in each grading $j$, our double complex $F(M^\bullet)^{\bullet\bullet}_j$ can be considered as first quadrant since $M^i_j=0$ for $i\ll 0$. This double complex has strictly exact rows since $M^\bullet$ is strictly exact and each $A_i^!$ is a flat $A_0$-module. Under the canonical embedding 
\begin{equation*}
    I:\textcat{grA}^!\textcat{-Mod}\rightarrow{\mathcal{LH}(\textcat{grA}^!\textcat{-Mod})}
\end{equation*}from Proposition \ref{leftheartembedding}, we note that $I(F(M^\bullet)^{\bullet\bullet}_j)$ is a first quadrant double complex with exact rows, since, by Corollary \ref{strictexactembedding}, exactness is preserved by the embedding . 

Hence, we see that, in each degree $j$, associated to our first quadrant double complex $I(F(M^\bullet))_j^{\bullet\bullet}$, there exists a spectral sequence $\{{}^{II}E_{r,j}^{i,l}\}$ for $r\geq 0$ with first term  \begin{equation*}
    {}^{II}E_{1,j}^{i,l}=H_h^i(I(F(M^\bullet))_j^{\bullet, l})
\end{equation*}with differentials ${}^{II}d_{1,j}^{i,l}=H_h^i(I(d_h)^{\bullet,l}_j)$. Furthermore, by comparing the total complex with respect to the grading $(i,l)$ with the twisted total complex with respect to $(i,j)$, we have the convergence
\begin{equation*}
\begin{aligned}
    {}^{II}E_{2,j}^{i,l}=H_v^l(H_h^i( I(F(M^\bullet))_j^{\bullet\bullet}))&\Rightarrow{H^{i+l}(\textnormal{Tot}(I(F(M^\bullet))_j^{\bullet\bullet})^\bullet)}\\
    &\simeq H^{n+k}(I(FM)^{\bullet-k}_k)\\
    &\simeq H^{n}(I(FM)^{\bullet}_k)
\end{aligned}\end{equation*}where $n=i+j,k=l-j$. Since $I(F(M^\bullet))_j^{\bullet\bullet}$ has exact rows, $H_h^{i}(I(F(M^\bullet))_j^{\bullet\bullet})$ is the zero complex. Hence, since $I$ preserves finite products, equivalently finite coproducts, as it is a right adjoint by Proposition \ref{leftheartembedding}, then
\begin{equation*}\bigoplus_{k=l-j}H^{n}(I(FM)^{\bullet}_k)=H^n\bigg(\bigoplus_{k=l-j}I(FM)^{\bullet}_k\bigg)=H^n(I(FM)^\bullet)=0\end{equation*}
Hence, $I(FM)^\bullet=I(\textcat{C}F(M^\bullet))$ is an exact complex in $\mathcal{LH}(\textcat{grA}^!\textcat{-Mod})$. It follows that $\textcat{C}F(M^\bullet)$ is strictly exact in $\textcat{grA}^!\textcat{-Mod}$ by Corollary \ref{strictexactembedding}. Therefore, the functor $\textcat{C}F$ induces a functor
\begin{equation*}
    \textcat{D}F: \textcat{D}^\downarrow(\textcat{grA-Mod})\rightarrow{ \textcat{D}^\uparrow({\textcat{grA}}^!\textcat{-Mod})   }
\end{equation*}
\end{proof}

\begin{lem}\label{CGfunctor}
There exists a functor $\textcat{C}G:\textcat{C}({\textcat{grA}}^!\textcat{-Mod})\rightarrow{ \textcat{C}(\textcat{grA-Mod}) }$. 
\end{lem}

\begin{proof}
Suppose that we have $(N^\bullet,d_N^\bullet)\in\textcat{C}(\textcat{grA}^!\textcat{-Mod})$. We consider $A$ as a complex $(A^\bullet,d_A^\bullet)$ in $\textcat{C}(\textcat{grA-Mod})$ and consider 
\begin{align*}
    G(N^\bullet)&=\texthom{Hom}_{\textcat{A}_0\textcat{-Mod}}(A,N^\bullet)\in\textcat{A-Mod}\\
\intertext{We construct the double complex}
    G(N^\bullet)^{\bullet\bullet}&=\texthom{Hom}_{\textcat{A}_0\textcat{-Mod}}(A_\bullet,N^\bullet)\\
\intertext{with objects}
    G(N^\bullet)^{i,l}&=\texthom{Hom}_{\textcat{A}_0\textcat{-Mod}}(A_{-l},N^i)\\
\intertext{and graded by}
G(N^\bullet)^{i,l}_j&=\texthom{Hom}_{\textcat{A}_0\textcat{-Mod}}(A_{-l},N^i_j)
\end{align*}The differentials are defined as follows. We note that, since $A$ is left dualisable 
\begin{equation}\label{thisequation}
    \texthom{Hom}_{\textcat{A}_0\textcat{-Mod}}(A_{-l},N^i_j)\simeq A_{-l}^*\otimes_{A_0} N^i_j
\end{equation}by Proposition \ref{dualtensorhomprop}. Hence, the $j^{th}$ component of the vertical differential
\begin{equation*}
    (d_v^{i,l})_j:\texthom{Hom}_{\textcat{A}_0\textcat{-Mod}}(A_{-l},N^i_j)\rightarrow{\texthom{Hom}_{\textcat{A}_0\textcat{-Mod}}(A_{-(l+1)},N^i_{j+1})}
\end{equation*}is $(-1)^{i+j}(D_v^{i,l})_j$ where $(D_v^{i,l})_j$ can be defined, using Equation \ref{thisequation} and the composition
\begin{equation*}
    A^*_{-l}\otimes_{A_0} N^i_{j}\xrightarrow{{}\mu_{1,-(l+1)}^*\otimes_{A_0} id_{N^i_j}}{A^*_{-(l+1)}\otimes_{A_0}A_1\otimes_{A_0} N^i_{j}}
    \rightarrow{A^*_{-(l+1)}\otimes_{A_0} N^i_{j+1}}
\end{equation*}The $j^{th}$ component of the horizontal differential 
\begin{equation*}
     (d_h^{i,l})_j:\texthom{Hom}_{\textcat{A}_0\textcat{-Mod}}(A_{-l},N^i_j)\rightarrow{\texthom{Hom}_{\textcat{A}_0\textcat{-Mod}}(A_{-l},N^{i+1}_{j})}
\end{equation*}is defined by the composition 
\begin{equation*}
    A_{-l}\rightarrow{N^i_j}\xrightarrow{(d_N^i)_j}{N^{i+1}_j}
\end{equation*}It is easy to check that this is indeed a double complex. In degree $k$, we can visualise the double complex as follows\begin{equation*}
\begin{tikzcd}[column sep = 6]
    & 0 & 0 & 0 \\
    \dots\arrow{r} & N^i_{-k}\arrow{u} \arrow{r} & N^{i+1}_{-k} \arrow{r} \arrow{u} & N^{i+2}_{-k} \arrow{r}\arrow{u} & \dots\\
    \dots\arrow{r} & \textnormal{\underline{Hom}}_{A_0}(A_{1},N^i_{-k-1}) \arrow{u} \arrow{r} & \textnormal{\underline{Hom}}_{A_0}(A_{1},N^{i+1}_{-k-1}) \arrow{r} \arrow{u} &  \textnormal{\underline{Hom}}_{A_0}(A_{1},N^{i+2}_{-k-1})\arrow{r}\arrow{u} & \dots\\
    \dots\arrow{r} & \textnormal{\underline{Hom}}_{A_0}(A_{2},N^i_{-k-2}) \arrow{u} \arrow{r} &  \textnormal{\underline{Hom}}_{A_0}(A_{2},N^{i+1}_{-k-2})\arrow{r} \arrow{u} & \textnormal{\underline{Hom}}_{A_0}(A_{2},N^{i+2}_{-k-2})  \arrow{r}\arrow{u} & \dots\\
    &\vdots\arrow{u} &\vdots \arrow{u} &\vdots \arrow{u}
    \end{tikzcd}
 \end{equation*}   
We see that our double complex is bounded from above since $A$ is positively graded. We now consider a kind of `twisted' total complex $GN^\bullet$ with
\begin{equation*}
     (GN)^n_k=\bigoplus_{i+j=n,k=l-j} \texthom{Hom}_{\textcat{A}_0\textcat{-Mod}}(A_{-l},N^i_j)
\end{equation*}and with associated total differential $d_G^\bullet=d_h^{\bullet\bullet}+d_v^{\bullet\bullet}$. We see that this complex is clearly a complex of $A$-modules. We therefore have a functor
\begin{equation*}\begin{aligned}
    \textcat{C}G:\textcat{C}({\textcat{grA}}^!\textcat{-Mod})&\rightarrow{ \textcat{C}(\textcat{grA-Mod}) }\\
    N^\bullet &\rightarrow{(GN)^\bullet}
\end{aligned}\end{equation*}

\end{proof}

\begin{prop}\label{prop2}
$\textcat{C}G$ induces a functor $\textcat{D}G: \textcat{D}^\uparrow(\textcat{grA}^!\textcat{-Mod})\rightarrow{ \textcat{D}^\downarrow({\textcat{grA-Mod})   }}$.
\end{prop}

\begin{proof}
Suppose that $N^\bullet\in\textcat{C}^\uparrow(\textcat{grA}^!\textcat{-Mod})$, and consider the functor 
\begin{equation*}\begin{aligned}
    \textcat{C}G:\textcat{C}({\textcat{grA}}^!\textcat{-Mod})&\rightarrow{ \textcat{C}(\textcat{grA-Mod}) }\\
    N^\bullet &\rightarrow{(GN)^\bullet}
\end{aligned}\end{equation*} defined in the previous Lemma. We now see why we specified the boundedness conditions on $N^\bullet$. If $n\ll 0$, then $i+j\ll 0$, and hence, since $N^\bullet\in \textcat{C}^\uparrow({\textcat{grA}}^!\textcat{-Mod})$, we see that $N^i_j=0$. Now, if $n+k\gg 0$, then $i+j+k\gg 0$. Since $A$ is positively graded, we know that $-(j+k)\geq 0$, and hence $i\gg 0$. Therefore, $N^i_j=0$ once again. We see that in both cases, $(GN)^n_k=0$ and therefore, $ (GN)^\bullet\in \textcat{C}^\downarrow(\textcat{grA-Mod})$. We can therefore restrict $\textcat{C}G$ to a functor
\begin{equation*}
    \textcat{C}G: \textcat{C}^\uparrow(\textcat{grA}^!\textcat{-Mod})\rightarrow{ \textcat{C}^\downarrow({\textcat{grA-Mod})   }}
\end{equation*}To show that this functor induces the desired functor it suffices to show that this functor preserves strict exactness. Indeed, suppose that $N^\bullet$ is a strictly exact complex. We note that our double complex $G(N^\bullet)^{\bullet\bullet}$ is third quadrant since $N^i_j=0$ for $i\gg 0$. This double complex has strictly exact rows since  each $A_{-l}$ is a projective $A_0$-module, and hence $\texthom{Hom}_{\textcat{A}_0\textcat{-Mod}}(A_{-l},-)$ is an exact functor. Under the canonical embedding
\begin{equation*}
    I:\textcat{grA-Mod}\rightarrow{\mathcal{LH}(\textcat{grA-Mod}})
\end{equation*}we note that $I(G(N^\bullet)^{\bullet\bullet}_j)$ is a third quadrant double complex with exact rows. 

Hence, we see that, in each degree $j$, associated to our third quadrant double complex $I(G(N^\bullet))^{\bullet\bullet}_j$, there exists a spectral sequence $\{{}^{II}E_{r,j}^{i,l}\}$ for $r\geq 0$ with first term  \begin{equation*}
    {}^{II}E_{r,j}^{i,l}=H_h^i(I(G(N^\bullet))^{\bullet, l}_j)
\end{equation*}with differentials ${}^{II}d_{1,j}^{i,l}=H_h^i(I(d_h)^{\bullet,l}_j)$. Furthermore, by comparing the total complex with respect to the grading $(i,l)$ with the total twisted complex with respect to $(i,j)$, we have the convergence
\begin{equation*}
   \begin{aligned}
    {}^{II}E_{2,j}^{i,l}=H_v^l(H_h^i( I(G(N^\bullet))_j^{\bullet\bullet}))&\Rightarrow{H^{i+l}(\textnormal{Tot}(I(G(N^\bullet))_j^{\bullet\bullet})^\bullet)}\\
    &\simeq H^{n+k}(I(GN)^{\bullet-k}_k)\\
    &\simeq H^{n}(I(GN)^{\bullet}_k)
\end{aligned}\end{equation*}where $n=i+j,k=l-j$. Since $I(G(N^\bullet))_j^{\bullet\bullet}$ has exact rows, $H_h^{i}(I(G(N^\bullet))_j^{\bullet\bullet})$ is the zero complex. Now, since $I$ preserves finite products, equivalently finite coproducts, as it is a right adjoint by Proposition \ref{leftheartembedding}, then
\begin{equation*}\bigoplus_{k=l-j}H^{n}(I(GN)^{\bullet}_k)=H^n\bigg(\bigoplus_{k=l-j}I(GN)^{\bullet}_k\bigg)=H^n(I(GN)^\bullet)=0\end{equation*}
Hence, $I(GN)^\bullet=I(\textcat{C}G(N^\bullet))$ is an exact complex in $\mathcal{LH}(\textcat{grA}^!\textcat{-Mod})$. It follows that $\textcat{C}G(N^\bullet)$ is strictly exact in $\textcat{grA}^!\textcat{-Mod}$ by Corollary \ref{strictexactembedding}. Therefore, the functor $\textcat{C}G$ induces a functor
\begin{equation*}
    \textcat{D}G: \textcat{D}^\uparrow(\textcat{grA}^!\textcat{-Mod})\rightarrow{ \textcat{D}^\downarrow({\textcat{grA}}\textcat{-Mod})   }
\end{equation*}
\end{proof}

We note that, by Corollary \ref{tensorhomcor2}, for any modules $M\in\textcat{A-Mod}$ and $N\in\textcat{A}^!\textcat{-Mod}$, there is a chain of isomorphisms
\begin{equation*}\label{niceisosyay}
    \textnormal{Hom}_{\textcat{A}^!\textcat{-Mod}}(A^!\otimes_{A_0}M,N)\simeq \textnormal{Hom}_{\textcat{A}_0\textcat{-Mod}}(M,N)\simeq\textnormal{Hom}_{\textcat{A-Mod}}(M,\texthom{Hom}_{\textcat{A}_0\textcat{-Mod}}(A,N))
\end{equation*}natural in $M$ and $N$.

\begin{lem}
    Suppose that $M^\bullet\in\textcat{C}(\textcat{grA-Mod})$ and $N^\bullet\in \textcat{C}(\textcat{grA}^!\textcat{-Mod})$. There is an isomorphism
    \begin{equation*}
        \textnormal{Hom}_{\textcat{grA}^!\textcat{-Mod}}((FM)^\bullet,N^\bullet)\simeq \textnormal{Hom}_{\textcat{grA-Mod}}(M^\bullet,(GN)^\bullet)
    \end{equation*}
\end{lem}

\begin{proof}
{We note that we have an isomorphism
\begin{equation*}
    \textnormal{Hom}_{\textcat{A}^!\textcat{-Mod}}((FM)^\bullet,N^\bullet)\simeq \textnormal{Hom}_{\textcat{A}^!\textcat{-Mod}}(M^\bullet,(GN)^\bullet)
\end{equation*}We will show that this isomorphism respects the grading. Since
\begin{equation*}\begin{aligned}
\textnormal{Hom}_{\textcat{A}^!\textcat{-Mod}}((FM)^n_k,N^n_k)&=\textnormal{Hom}_{\textcat{A}^!\textcat{-Mod}}(\bigoplus_{i+j=n,k=l-j}A^!_l\otimes_{A_0}M^i_j,N^n_k)\\
    &\simeq\bigoplus_{i+j=n}\textnormal{Hom}_{\textcat{A}_0\textcat{-Mod}}(M^i_j,\bigoplus_{k=l-j}\texthom{Hom}_{\textcat{A}^!\textcat{-Mod}}(A_l^!,N^n_k))\\
\end{aligned}
\end{equation*}We see that a morphism $f^n:(FM)^n\rightarrow{N^n}$ respects the grading if and only if each morphism $M^i\rightarrow{\texthom{Hom}_{\textcat{A}^!\textcat{-Mod}}(A^!,N^n)}\simeq N^n$ sends pieces graded in degree $j$ to pieces graded in degree $-j$ for all $i+j=n$.
Similarly, since 
\begin{equation*}\begin{aligned}
    \textnormal{Hom}_{\textcat{A}\textcat{-Mod}}(M^i_q,(GN)^i_q)&=\textnormal{Hom}_{\textcat{A}\textcat{-Mod}}(M^i_q,\bigoplus_{n-j=i,q=l+j}\texthom{Hom}_{\textcat{A}_0\textcat{-Mod}}(A_{-l},N^n_{-j}))\\
    &\simeq\bigoplus_{n-j=i}\textnormal{Hom}_{\textcat{A}_0\textcat{-Mod}}(\bigoplus_{q=l+j}A_{-l}\otimes_{A}M^i_q, N^n_{-j})
\end{aligned}
\end{equation*}We see that a morphism $g^i:M^i\rightarrow{(GN)^i}$ respects the grading if and only if each morphism $M^i\simeq A\otimes_{A}M^i\rightarrow{N^n}$ sends pieces graded in degree $j$ to pieces graded in degree $-j$ for all $n=i+j$. Therefore, we see that $f^n$ preserves the grading for all $n$ if and only if $g^{m}$ preserves the grading for all $m$.} 
\end{proof}

\begin{prop}
   The above isomorphisms induce an adjunction 
   \begin{equation*}
       \textcat{C}F: \textcat{C}(\textcat{grA-Mod})\leftrightarrows{ \textcat{C}({\textcat{grA}}^!\textcat{-Mod})   }:\textcat{C}G
   \end{equation*}
\end{prop}

\begin{proof}By the previous lemma, it remains to show that the isomorphism
    \begin{equation*}
        \textnormal{Hom}_{\textcat{grA}^!\textcat{-Mod}}((FM)^\bullet,N^\bullet)\simeq \textnormal{Hom}_{\textcat{grA-Mod}}(M^\bullet,(GN)^\bullet)
    \end{equation*}
respects cochain maps for all $M^\bullet\in\textcat{C}(\textcat{grA-Mod})$ and $N^\bullet\in \textcat{C}(\textcat{grA}^!\textcat{-Mod})$. Indeed, suppose that there exists an element $f^\bullet\in\textnormal{Hom}_{\textcat{C}(\textcat{grA}^!\textcat{-Mod})}((FM)^\bullet,N^\bullet)$. Let its image be $g^\bullet\in\textnormal{Hom}_{\textcat{C}(\textcat{grA-Mod})}(M^\bullet,(GN)^\bullet) $. We need to show that $f^\bullet$ commutes with the differentials if and only if $g^\bullet$ does. We note that, by the previous Lemma, $f^n_k$ induces morphisms $M^{n+k}_{-k}\rightarrow{N^n_k}$. We consider the following diagram 
\begin{equation*}
      \begin{tikzcd}
      M^{n+k}_{-k}\arrow[hook]{d} \arrow{r}{(d^{n+k}_M)_{-k}} & M^{n+k+1}_{-k}\arrow[hook]{d}\\
        \bigoplus_{i+j=n,k=l-j}A^!_{l}\otimes_{A_0} M^i_j\arrow{r}{(d_F^{n})_k} \arrow{d}{f^{n}_{k}} & \bigoplus_{i+j=n+1,k=l-j}A^!_{l}\otimes_{A_0} M^{i}_j\arrow{d}{f^{n+1}_{k}}\\
        N^{n}_{k} \arrow{r}{(d_N^{n})_{k}} & N^{n+1}_{k}
    \end{tikzcd}
\end{equation*}We easily see that the top square of the diagram commutes. Commutativity of the bottom square, when composed with the inclusion $M^{n+k}_{-k}\rightarrow{\bigoplus_{i+j=n,k=l-j}A^!_{l}\otimes_{A_0} M^i_j}$ is equivalent to the commutativity on $M^{n+k}_{-k}$ of $f^n_k$ with the differentials. Therefore, $f^n_k$ commutes with the differentials for all $n,k$ if and only if the outer square commutes. Similarly, we obtain the following diagram 
\begin{equation*}
    \begin{tikzcd}
        M^{n+k}_{-k}\arrow{d}{g^{n+k}_{-k}}\arrow{r}{(d^{n+k}_M)_{-k}} & M^{n+k+1}_{-k}\arrow{d}{g^{n+k+1}_{-k}}\\
        \bigoplus_{i+j=n+k,-k=l-j}\texthom{Hom}_{\textcat{A}_0\textcat{-Mod}}(A_{-l},N^i_j)\arrow{r}{(d_G^{n+k})_{-k}} \arrow{d} & \bigoplus_{i+j=n+k+1,-k=l-j}\texthom{Hom}_{\textcat{A}_0\textcat{-Mod}}(A_{-l},N^i_j)\arrow{d}\\
        N^{n}_{k} \arrow{r}{(d_N^{n})_{k}} & N^{n+1}_{k}
    \end{tikzcd}
\end{equation*}
We note that the bottom square of this diagram commutes. Commutativity of the top square when composed with the evaluation morphism $\bigoplus_{i+j=n+k+1,-k=l-j}\texthom{Hom}_{\textcat{A}_0\textcat{-Mod}}(A_{-l},N^i_j)\rightarrow{N^{n+1}_k}$ is equivalent to the commutativity of $g^{n+k}_{-k}$ with the differentials. Therefore, $g^{n+k}_{-k}$ commutes with the differentials for all $n,k$ if and only if the outer square commutes. 

Hence, we see that commutativity of $f^n_k$ and $g^{n+k}_{-k}$ with the differentials is equivalent to commutativity of the following diagram 
\begin{equation*}
    \begin{tikzcd}
        M^{n+k}_{-k}\arrow{r}{(d_M^{n+k})_{-k}}\arrow{d} & M^{n+k+1}_{-k} \arrow{d}\\
        N^n_k\arrow{r}{(d^n_N)_k} & N^{n+1}_k
    \end{tikzcd}
\end{equation*}Therefore, $f^n_k$ commutes with the differentials for all $n,k$ if and only if $g^m_l$ does for all $m,l$.

\end{proof}

Consider the counit 
\begin{equation*}
    \varepsilon:\textcat{C}F\circ\textcat{C}G\rightarrow{id_{\textcat{C}(\textcat{grA}^!\textcat{-Mod})}}
\end{equation*}of the adjunction. Suppose that $N^\bullet\in\textcat{C}(\textcat{grA}^!\textcat{-Mod})$. Then, we may consider the double complex $(F\circ\textcat{C}G)(N^\bullet)^{\bullet\bullet}$ with 
\begin{equation*}
    (F\circ \textcat{C}G)(N^\bullet)^{i,l}= A^!_{l}\otimes_{A_0} \textcat{C}G(N^\bullet)^i=A^!_l\otimes_{A_0} \bigoplus_{p+q=i}\texthom{Hom}_{\textcat{A}_0\textcat{-Mod}}(A,N^p_q)
\end{equation*} and graded by 
\begin{equation*}
    (F\circ \textcat{C}G)(N^\bullet)^{i,l}_j
    =A^!_{l}\otimes_{A_0} \bigoplus_{p+q=i,j=r-q}\texthom{Hom}_{\textcat{A}_0\textcat{-Mod}}(A_{-r},N^p_q)
\end{equation*}
Since $A_{l}^!$ is right dualisable, we see that, by Corollary \ref{rightdualisableiso},
\begin{equation*}\begin{aligned}
    A_{l}^!\otimes_{A_0} \texthom{Hom}_{\textcat{A}_0\textcat{-Mod}}(A_{-r},N^p_q)&\simeq\texthom{Hom}_{\textcat{A}_0\textcat{-Mod}}({}^*(A_{l}^!),\texthom{Hom}_{\textcat{A}_0\textcat{-Mod}}(A_{-r},N^p_q))\\
    &\simeq \texthom{Hom}_{\textcat{A}_0\textcat{-Mod}}(A_{-r}\otimes_{A_0} {}^*(A_{l}^!),N^p_q)
\end{aligned}\end{equation*}and hence
\begin{equation*}
    (F\circ\textcat{C}G)(N^\bullet)^{i,l}_j\simeq \bigoplus_{p+q=i,j=r-q}\texthom{Hom}_{\textcat{A}_0\textcat{-Mod}}(A_{-r}\otimes_{A_0} {}^*(A_{l}^!),N^p_q)
\end{equation*}The horizontal differential is given by 
\begin{equation*}
    (d_h^{i,l})_j=id_{A^!_{l}}\otimes_{A_0} (d_G^i)_j
\end{equation*}where $(d_G^i)_j$ is the total differential defined in Lemma \ref{CGfunctor}. The vertical differential 
\begin{equation*}
    \begin{aligned}
    (d_v^{i,l})_j:&\bigoplus_{p+q=i,j=r-q}\texthom{Hom}_{\textcat{A}_0\textcat{-Mod}}(A_{-r}\otimes_{A_0} {}^*(A_{l}^!),N^p_q)\\&\rightarrow{\bigoplus_{p+q=i,j=r-q}\texthom{Hom}_{\textcat{A}_0\textcat{-Mod}}(A_{-(r+1)}\otimes_{A_0} {}^*(A_{l+1}^!),N^p_q)}
\end{aligned}
\end{equation*}is given by $(-1)^{i}(D_v^{i,l})_j$ where $(D_v^{i,l})_j$ can be defined by taking the coproduct of the maps
\begin{equation*}
    \texthom{Hom}_{\textcat{A}_0\textcat{-Mod}}(A_{-r}\otimes_{A_0} {}^*(A_{l}^!),N^p_q)\rightarrow{\texthom{Hom}_{\textcat{A}_0\textcat{-Mod}}(A_{-(r+1)}\otimes_{A_0} {}^*(A_{l+1}^!),N^p_q)}
\end{equation*}given by the composition 
\begin{equation*}
 A_{-(r+1)}\otimes_{A_0} {}^*(A_{l+1}^!)\xrightarrow{d^{l+1}_{l-r}}{A_{-r}\otimes_{A_0}{}^*(A_{l}^!)}\rightarrow{N^p_q}
\end{equation*}where $d^\bullet$ is the differential from the Koszul complex of $A$. Taking the total complex with respect to $i$ and $j$ we obtain 
\begin{equation*}\begin{aligned}
    (\textcat{C}F\circ\textcat{C}G)(N^\bullet)^n_k&=\bigoplus_{i+j=n,k=l-j}A^!_{l}\otimes_{A_0}\bigoplus_{p+q=i,j=r-q}\textnormal{\underline{Hom}}_{A_0}(A_{-r},N^p_q)\\
     &=\bigoplus_{i+j=n,k=l-j}\bigoplus_{p+q=i,j=r-q}\texthom{Hom}_{\textcat{A}_0\textcat{-Mod}}(A_{-r}\otimes_{A_0}{}^*(A^!_{l}),N^p_q)
\end{aligned}\end{equation*}with differential given by the total differential. 
\begin{lem}
The counit $\varepsilon$ is a split strict epimorphism. 
\end{lem}

\begin{proof}
It is easy to see that the counit is the strict natural transformation with components $\varepsilon_{N^\bullet}$ with
\begin{equation*}
    (\varepsilon_{N^\bullet}^n)_k: (\textcat{C}F\circ\textcat{C}G)(N^\bullet)^n_k=\bigoplus_{i+j=n,k=l-j}A^!_{l}\otimes_{A_0} \bigoplus_{p+q=i,j=r-q}\texthom{Hom}_{\textcat{A}_0\textcat{-Mod}}(A_{-r},N^p_q)\rightarrow{N^n_k}
\end{equation*}given by the composition of the following strict maps
\begin{equation*}
     \begin{aligned}
     &\bigoplus_{i+j=n,k=l-j}A^!_{l}\otimes_{A_0} \bigoplus_{p+q=i,j=r-q}\texthom{Hom}_{\textcat{A}_0\textcat{-Mod}}(A_{-r},N^p_q)\otimes_{A_0} A_0\\
    &\xrightarrow{id_{A^!}\otimes_{A_0} ev}{\bigoplus_{i+j=n,k=l-j}A_{l}^!\otimes_{A_0} N_{-j}^{i+j}}\\
    &\rightarrow{N^n_{l-j}}=N^n_k
    \end{aligned}
\end{equation*}where the last map is the one from Proposition \ref{strictmodulemap}. We consider the map 
\begin{equation*}
    \sigma_{N^\bullet}:N^\bullet\rightarrow{(\textcat{C}F\circ\textcat{C}G)(N^\bullet)}
\end{equation*}given in degree $n$ by the inclusion map 
\begin{equation*}\begin{aligned}
    N^n_k&\simeq A_0\otimes_{A_0} \texthom{Hom}_{\textcat{A}_0\textcat{-Mod}}(A_0,N^n_k)\\
    &\rightarrow{\bigoplus_{i+j=n,k=l-j}A^!_{l}\otimes_{A_0} \bigoplus_{p+q=i,j=r-q}\texthom{Hom}_{\textcat{A}_0\textcat{-Mod}}(A_{-r},N^p_q)}
\end{aligned}\end{equation*}We note that the map $((\varepsilon\circ\sigma)^n_{N^\bullet})_k$ is equivalent to the map 
\begin{equation*}
    N^n_k\simeq A_0\otimes_{A_0}\texthom{Hom}_{\textcat{A}_0\textcat{-Mod}}(A_0,N^n_k)\xrightarrow[ev]{\simeq }N^n_k
\end{equation*}which is the identity map. Hence, $\varepsilon\circ\sigma=id$ and so $\varepsilon$ is a split epimorphism.
\end{proof}

\begin{lem}\label{compquasi}
Let $f^\bullet:Y^\bullet\rightarrow{X^\bullet}$ be a strict cochain map in a quasi-abelian category $\mathcal{E}$. Suppose that there exists a strict map $g^\bullet:X^\bullet\rightarrow{Y^\bullet}$ such that $f^\bullet\circ g^\bullet$ is a strict quasi-isomorphism. Then, $f^\bullet$ is a strict quasi-isomorphism if and only if $g^\bullet$ is. 
\end{lem}

\begin{proof}
Consider the triangulated homotopy category $\textcat{K}(\mathcal{E})$ and consider the following exact triangles in $\textcat{K}(\mathcal{E})$
\begin{equation*}
    \begin{aligned}
    X^\bullet &\xrightarrow{g^\bullet}{Y^\bullet}\rightarrow{cone(g^\bullet)^\bullet}\rightarrow{X^\bullet[1]}\\
    Y^\bullet &\xrightarrow{f^\bullet}{X^\bullet}\rightarrow{cone(f^\bullet)^\bullet}\rightarrow{Y^\bullet[1]}\\
    X^\bullet &\xrightarrow{f^\bullet\circ g^\bullet}{X^\bullet}\rightarrow{cone(f^\bullet\circ g^\bullet)^\bullet} \rightarrow{X^\bullet[1]}
\end{aligned}
\end{equation*} We note that, by the octahedral axiom for triangulated categories, there exists an exact triangle
\begin{equation}\label{octahedraltriangle}
     cone(g^\bullet)^\bullet\rightarrow{cone(f^\bullet\circ g^\bullet)^\bullet}\rightarrow{cone(f^\bullet)^\bullet}\rightarrow{cone(g^\bullet)^\bullet[1]}
\end{equation}in $\textcat{K}(\mathcal{E})$. Hence, 
\begin{equation*}
    cone(f^\bullet)^\bullet\simeq cone(cone(g^\bullet)\rightarrow{cone(f^\bullet\circ g^\bullet)^\bullet})
\end{equation*}Suppose that $g^\bullet$ is a strict quasi-isomorphism. Since $g^\bullet$ and $f^\bullet\circ g^\bullet$ are both strict quasi-isomorphisms, $cone(g^\bullet)^\bullet$ and $cone(f^\bullet\circ g^\bullet)^\bullet$ are both strictly exact complexes. Hence, since the cone of a morphism between strictly exact complexes is strictly exact (see \cite[Lemma 1.1]{neeman89}), $cone(f^\bullet)^\bullet$ is strictly exact. Therefore, $f^\bullet$ is a strict quasi-isomorphism. 

Now, we shift the exact triangle in Equation \ref{octahedraltriangle} to obtain the exact triangle
\begin{equation*}
  {cone(f^\bullet\circ g^\bullet)^\bullet}\rightarrow{cone(f^\bullet)^\bullet}\rightarrow{cone(g^\bullet)^\bullet[1]}\rightarrow{cone(f^\bullet\circ g^\bullet)^\bullet[1]}
\end{equation*} in $\textcat{K}(\mathcal{E})$. Using a similar reasoning to before, if $f^\bullet$ is a strict quasi-isomorphism, then $cone(g^\bullet)^\bullet[1]$ is strictly exact. Hence, $cone(g^\bullet)^\bullet$ is strictly exact, and so $g^\bullet$ is a strict quasi-isomorphism.
\end{proof}

\begin{prop}
There is an equivalence $\textcat{D}F\circ \textcat{D}G\rightarrow{id_{\textcat{D}^\uparrow(\textcat{grA}^!\textcat{-Mod})}}$.
\end{prop}

\begin{proof}
Since strict quasi-isomorphisms become isomorphisms in $\textcat{D}^\uparrow(\textcat{grA}^!\textcat{-Mod})$, it suffices to prove that, for $N^\bullet\in\textcat{C}^\uparrow(\textcat{grA}^!\textcat{-Mod})$, the counit
\begin{equation*}
    \epsilon_{N^\bullet}:(\textbf{C}F\circ \textbf{C}G)(N^\bullet)\rightarrow{N^\bullet}
\end{equation*} is a strict quasi-isomorphism. However, since this map is a split epimorphism, it suffices, by Lemma \ref{compquasi}, to show that the map \begin{equation*}
    \sigma_{N^\bullet}:N^\bullet\rightarrow{(\textbf{C}F\circ \textbf{C}G)(N^\bullet)}
\end{equation*} is a strict quasi-isomorphism. Once again, we consider the canonical embedding
\begin{equation*}
    I:\textcat{grA}^!\textcat{-Mod}\rightarrow{\mathcal{LH}(\textcat{grA}^!\textcat{-Mod})}
\end{equation*} We note that, by Corollary \ref{Iquasiiso}, $\sigma_{N^\bullet}$ is a strict quasi-isomorphism if and only if $I(\sigma_{N^\bullet})$ is a quasi-isomorphism. We want to show that, for each $n$, $I(\sigma_{N^\bullet}^n)$ induces an isomorphism 
\begin{equation*}
    H^n(I(N^\bullet))\simeq H^n(I((\textcat{C}F\circ \textcat{C}G)(N^\bullet)))
\end{equation*}in cohomology. The image of the first quadrant double complex $(F\circ \textcat{C}G)(N^\bullet)^{\bullet\bullet}$ under the embedding is a first quadrant double complex $I((F\circ \textcat{C}G)(N^\bullet))^{\bullet\bullet}$. We consider the spectral sequence $\{{}^IE_{r}^{i,l}\}$ for $r\geq 0$, with differentials ${}^Id_{r}^{i,l}$, such that 
\begin{equation*}
    {}^I E_{0}^{i,l}=I((F\circ \textcat{C}G)(N^\bullet)^{i,l})
\end{equation*}and the maps ${}^I d_{0}^{i,l}$ are just the vertical differentials $I(d_v^{i,l})$. We also consider
\begin{equation*}
    {}^IE_{1}^{i,l}= H_v^{l}(I((F\circ \textcat{C}G)(N^\bullet))^{i,\bullet})
\end{equation*}with differentials ${}^Id_{1}^{i,l}=H_v^{l}(I(d_v)^{i,\bullet})$. We have that 
\begin{align*}
  H_v^{l}(I((F\circ \textcat{C}G)(N^\bullet))^{i,\bullet})&=H_v^{l}\bigg(I\bigg(\bigoplus_{p+q=i}\texthom{Hom}_{\textcat{A}_0\textcat{-Mod}}(A\otimes_{A_0} {}^*(A^!_{\bullet}),N^p_q)\bigg)\bigg)\\
  &\simeq \bigoplus_{p+q=i}H_v^{l}(I(\texthom{Hom}_{\textcat{A}_0\textcat{-Mod}}(A\otimes_{A_0} {}^*(A^!_{\bullet}),N^p_q)))
  \intertext{Now, since $A$ is Koszul, the complex $A\otimes_{A_0} {}^*(A_{\bullet}^!)$ is a projective resolution of $A_0$. Therefore,}
  &\simeq \begin{cases}\bigoplus_{p+q=i}I(\texthom{Ext}^0_{\textcat{A}_0\textcat{-Mod}}(A_0, N^p_q)) \quad & \text{if $l=0$}\\
  0 \quad & \text{otherwise}\end{cases}\\
  &\simeq \begin{cases}\bigoplus_{p+q=i}I(\texthom{Hom}_{\textcat{A}_0\textcat{-Mod}}(A_0, N^p_q)) \quad & \text{if $l=0$}\\
  0 \quad & \text{otherwise}\end{cases}\\
  &\simeq \begin{cases}\bigoplus_{p+q=i}I(N^p_q) \quad & \text{if $l=0$}\\
  0 \quad & \text{otherwise}\end{cases}
  \end{align*}
  
The differentials are given by \begin{equation*}
    {}^Id_{1}^{i,l}=\begin{cases}\bigoplus_{p+q=i}I((d^p_N)_q)\quad & \text{if $l=0$}\\
  0 \quad & \text{otherwise}\end{cases}
\end{equation*}We have the convergence
\begin{equation*}
\begin{aligned}
    {}^IE_{2}^{i,l}=H_h^i(H_v^{l}(I((F\circ \textcat{C}G)(N^\bullet))^{i,\bullet}))&\Rightarrow{H^{i+l}(\textnormal{Tot}(I((F\circ \textbf{C}G)(N^\bullet))^{\bullet\bullet})^\bullet)}\\
    &\simeq H^{n}(I((\textcat{C}F\circ \textbf{C}G)(N^\bullet))^{\bullet})
\end{aligned}
\end{equation*}where $n=i+l$.

We may view $N^\bullet$ as a double complex $N^{\bullet\bullet}$ with \begin{equation*}
    N^{i,l}=\begin{cases}\bigoplus_{p+q=i}N^p_q\quad &\text{if $l=0$}\\
0\quad &\text{otherwise}\end{cases}
\end{equation*}The horizontal differentials are given by \begin{equation*}d^{i,l}=
    \begin{cases}\bigoplus_{p+q=i}(d^p_N)_q\quad & \text{if $l=0$}\\
  0 \quad & \text{otherwise}\end{cases}
\end{equation*}and the vertical differentials are zero. The total complex is 
\begin{equation*}
    \textnormal{Tot}(N^{\bullet\bullet})^n=\bigoplus_{i+l=n}N^{i,l}=\bigoplus_{p+q=n}N^p_q
\end{equation*}We note that, since $N^\bullet\in\textcat{C}^\uparrow(\textcat{grA}^!\textcat{-Mod})$, then $N^{\bullet\bullet}$, and hence $I(N^{\bullet\bullet})$, is a first quadrant double complex. We construct another spectral sequence $\{{}^I\overline{E}_{r}^{i,l}\}$ for $r\geq 0$, with differentials ${}^I\overline{d}_{r}^{i,l}$, such that 
\begin{equation*}
    {}^I\overline{E}_{0}^{i,l}=I(N)^{i,l}
\end{equation*}and the maps ${}^I\overline{d}_{0}^{i,l}$ are just the corresponding differentials of $N^{\bullet\bullet}$. We have
\begin{equation*}
    {}^I\overline{E}_{1}^{i,l}=H_v^l(I(N^{i,\bullet}))=\begin{cases}\bigoplus_{p+q=i}I(N^p_q) \quad & \text{if $l=0$}\\
  0 \quad & \text{otherwise}\end{cases}
\end{equation*}We have the convergence
\begin{equation*}
    {}^I\overline{E}_2^{i,l}\Rightarrow{H^{i+l}(\textnormal{Tot}(I(N)^{\bullet\bullet})^\bullet)}=H^{n}(I(N^\bullet))
\end{equation*}
The map $\sigma$ induces a morphism of double complexes. Hence, we see that, since $\{{}^I\overline{E}_{r,j}^{i,l}\}$ and $\{{}^I{E}_{r,j}^{i,l}\}$ have the same first page and the differentials on the first page are the same, these spectral sequences must be equal for $r\geq 1$. Hence, 
\begin{equation*}
    H^{n}(I(N^{\bullet}))=H^{n}(I((\textbf{C}F\circ \textbf{C}G)(N^\bullet)^\bullet))
\end{equation*}and therefore $I(\sigma_{N^\bullet})$ is a quasi-isomorphism. It follows that $\sigma$ is a strict quasi-isomorphism, and hence so is $\varepsilon$.

\end{proof}

We now consider the unit
\begin{equation*}
    \eta:id_{\textcat{C}(\textcat{grA}^!\textcat{-Mod})}\rightarrow{\textcat{C}G\circ\textcat{C}F}
\end{equation*}of the adjunction. Suppose that $M^\bullet$ is in $\textcat{C}(\textcat{grA-Mod})$. Then, we may consider the double complex $(G\circ\textcat{C}F)(M^\bullet)^{\bullet\bullet}$ with 
\begin{equation*}\begin{aligned}
    (G\circ\textcat{C}F)(M^\bullet)^{i,l}=\texthom{Hom}_{\textcat{A}_0\textcat{-Mod}}(A_{-l},\textcat{C}F(M^\bullet)^{i})=\texthom{Hom}_{\textcat{A}_0\textcat{-Mod}}(A_{-l},\bigoplus_{p+q=i}A^!\otimes_{A_0} M^p_q)
\end{aligned}
\end{equation*}and graded by
\begin{equation*}\begin{aligned}
    (G\circ\textcat{C}F)(M^\bullet)^{i,l}_j&=\texthom{Hom}_{\textcat{A}_0\textcat{-Mod}}(A_{-l},\textcat{C}F(M^\bullet)^{i}_j)\\&=\texthom{Hom}_{\textcat{A}_0\textcat{-Mod}}(A_{-l},\bigoplus_{p+q=i,j=r-q}A^!_r\otimes_{A_0} M^p_q)
\end{aligned}
\end{equation*}
Since $A_{-l}$ is left dualisable, we see that 
\begin{equation*}\begin{aligned}
    (G\circ\textcat{C}F)(M^\bullet)^{i,l}_j&\simeq A_{-l}^*\otimes_{A_0}\bigg(\bigoplus_{p+q=i,j=r-q}A_r^!\otimes_{A_0} M^p_q\bigg)\\&\simeq \bigoplus_{p+q=i,j=r-q} A_{-l}^*\otimes_{A_0} A^!_r\otimes_{A_0} M^p_q
\end{aligned}\end{equation*}The horizontal differential is given by 
\begin{equation*}
    (d_h^{i,l})_j=id_{A^*_{-l}}\otimes_{A_0} (d_F^i)_j
\end{equation*}where $(d_F^i)_j$ is the total differential defined in Lemma \ref{CFfunctor}. The vertical differential 
\begin{equation*}
    (d_v^{i,l})_j:\bigoplus_{p+q=i,j=r-q} A_{-l}^*\otimes_{A_0} A_{r}^!\otimes_{A_0} M^p_q\rightarrow{\bigoplus_{p+q=i,j=r-q} A_{-(l+1)}^*\otimes_{A_0} A_{r+1}^!\otimes_{A_0} M^p_q}
\end{equation*}is given by $(-1)^i(D_v^{i,l})_j$, where $(D_v^{i,l})_j$ can be defined by taking the coproduct of the maps 
\begin{equation*}
    A_{-l}^*\otimes_{A_0} A_{r}^!\otimes_{A_0} M^p_q\xrightarrow{d^{-l}_{r-l}\otimes_{A_0} id_{M^p_q}}{{ A_{-(l+1)}^*\otimes_{A_0} A_{r+1}^!\otimes_{A_0} M^p_q}}
\end{equation*}where $d^\bullet$ is the differential from the Koszul complex of $A^!$. Taking the total complex we obtain
\begin{equation*}
    (\textcat{C}G\circ\textcat{C}F)(M^\bullet)^{n}_k=  \bigoplus_{i+j=n,k=l-j}\bigoplus_{p+q=i,j=r-q} A_{-l}^*\otimes_{A_0} A_{r}^!\otimes_{A_0} M^p_q
\end{equation*}with differential given by the total differential. 

\begin{lem}
The unit $\eta$ is a split strict monomorphism.
\end{lem}

\begin{proof}
It is easy to see that the unit is the strict natural transformation with components $(\eta^n_{M^\bullet})_k$, with 
\begin{equation*}
    (\eta^n_{M^\bullet})_k:M^n_k\rightarrow{\bigoplus_{i+j=n,k=l-j}\bigoplus_{p+q=i,j=r-q} A_{-l}^*\otimes_{A_0} A_{r}^!\otimes_{A_0} M^p_q}=(\textcat{C}G\circ\textcat{C}F)(M^\bullet)^{n}_k
\end{equation*}given in degree $n$ by the inclusion map
\begin{equation*}
    M^n_k\simeq{A_0\otimes_{A_0}A_0\otimes_{A_0} M^n_k}\hookrightarrow{\bigoplus_{i+j=n,k=l-j}\bigoplus_{p+q=i,j=r-q} A_{-l}^*\otimes_{A_0} A_{r}^!\otimes_{A_0} M^p_q}
\end{equation*}We consider the map 
\begin{equation*}
    \rho_{M^\bullet}:(\textcat{C}G\circ\textcat{C}F)(M^\bullet)\rightarrow{M^\bullet}
\end{equation*} given in degree $n$ by the projection map
\begin{equation*}
    \bigoplus_{i+j=n,k=l-j}\bigoplus_{p+q=i,j=r-q} A_{-l}^*\otimes_{A_0} A_{r}^!\otimes_{A_0} M^p_q\rightarrow{A_0\otimes_{A_0} A_0\otimes_{A_0} M^n_k}\simeq M^n_k
\end{equation*}We note that the map $((\rho\circ \eta)^n_{M^\bullet})_k$ is equivalent to the map 
\begin{equation*}
    M^n_k\simeq A_0\otimes_{A_0} A_0\otimes_{A_0} M^n_k\xrightarrow{\simeq} M^n_k
\end{equation*}which is the identity map. Hence, $\rho\circ\eta=id$, and so $\eta$ is a split monomorphism. 

\end{proof}

\begin{prop}
There is an equivalence $id_{\textcat{D}^\downarrow(\textcat{grA-Mod})}\rightarrow{\textcat{D}G\circ \textcat{D}F}$.
\end{prop}
\begin{proof}
Since strict quasi-isomorphisms become isomorphisms in $\textcat{D}^\downarrow(\textcat{grA-Mod})$, it suffices to prove that for $M^\bullet\in\textcat{C}^\downarrow(\textcat{grA-Mod})$, the unit
\begin{equation*}
    \eta_{M^\bullet}:M^\bullet\rightarrow{(\textcat{C}G\circ\textcat{C}F)(M^\bullet)}
\end{equation*}is a strict quasi-isomorphism. However, since this map is a split monomorphism, it suffices, by Lemma \ref{compquasi}, to show that the map 
\begin{equation*}
    \rho_{M^\bullet}:(\textcat{C}G\circ\textcat{C}F)(M^\bullet)\rightarrow{M^\bullet}
\end{equation*}is a strict quasi-isomorphism. Once again, we consider the canonical embedding 
\begin{equation*}
    I:\textcat{grA-Mod}\rightarrow{\mathcal{LH}(\textcat{grA-Mod})}
\end{equation*}We note that, by Corollary \ref{Iquasiiso}, $\rho_{M^\bullet}$ is a strict quasi-isomorphism if and only if $I(\rho_{M^\bullet})$ is a quasi-isomorphism. We want to show that, for each $n$, $I(\rho_{M^\bullet}^n)$ induces an isomorphism
\begin{equation*}
     H^n(I((\textcat{C}G\circ\textcat{C}F)(M^\bullet)))\simeq H^n(I(M^\bullet))
\end{equation*}in cohomology. The image of the third quadrant double complex $(G\circ\textcat{C}F)(M^\bullet)^{\bullet\bullet}$ under the embedding is a third quadrant double complex $I((G\circ\textcat{C}F)(M^\bullet))^{\bullet\bullet}$. We consider the spectral sequence $\{{}^I E_r^{i,l}\}$, for $r\geq 0$, with differentials ${}^Id_r^{i,l}$ such that 
\begin{equation*}
    {}^IE_0^{i,l}=I((G\circ\textcat{C}F)(M^\bullet)^{i,l})
\end{equation*}and the maps ${}^I d_0^{i,l}$ are just the vertical differentials $I(d_v^{i,l})$. We also consider
\begin{equation*}
    {}^IE_1^{i,l}=H_v^l(I((G\circ\textcat{C}F)(M^\bullet))^{i,\bullet})
\end{equation*}with differentials ${}^Id_1^{i,l}=H_v^l(I(d_v^{i,\bullet}))$. We have that 
\begin{align*}
H_v^{l}(I((G\circ \textcat{C}F)(M^\bullet))^{i,\bullet})&=H_v^{l}\bigg(I\bigg(\bigoplus_{p+q=i}A^*_{\bullet}\otimes_{A_0} A^!\otimes_{A_0} M^p_q\bigg)\bigg)\\
  \intertext{Since $I$ preserves finite coproducts,}
  &\simeq \bigoplus_{p+q=i} H_v^l(I(A^*_{\bullet}\otimes_{A_0} A^!\otimes_{A_0} M^p_q))
\intertext{Now, as $A^!$ is Koszul, the complex $A_\bullet^*\otimes_{A_0} A^!$ is a projective resolution of $A_0$. Therefore,}
&\simeq \begin{cases}
  \bigoplus_{p+q=i}I(\texthom{Tor}^0_{\textcat{A}_0\textcat{-Mod}}(A_0,M^p_q)) \quad & \text{if $l=0$}\\
  0 \quad & \text{otherwise}
\end{cases}\\
&\simeq \begin{cases}
  \bigoplus_{p+q=i}I(A_0\otimes_{A_0} M^p_q) \quad & \text{if $l=0$}\\
  0 \quad & \text{otherwise}
\end{cases}\\
&\simeq \begin{cases}
  \bigoplus_{p+q=i}I( M^p_q) \quad & \text{if $l=0$}\\
  0 \quad & \text{otherwise}
\end{cases}
\end{align*}
The differentials are given by \begin{equation*}
    {}^Id_1^{i,l}=\begin{cases}
  \bigoplus_{p+q=i}I((d_M)^p_q) \quad &\text{if $l=0$}\\
  0 \quad &\text{otherwise}
\end{cases}
\end{equation*}We have the convergence
\begin{equation*}
\begin{aligned}
    {}^IE_{2}^{i,l}=H_h^i(H_v^l(I((G\circ \textcat{C}F)(M^\bullet))^{i,\bullet}))&\Rightarrow{H^{i+l}(\textnormal{Tot}(I((G\circ \textbf{C}F)(M^\bullet))^{\bullet\bullet})^\bullet)}\\
    &=H^{n}(I((\textbf{C}G\circ \textbf{C}F)(M^\bullet))^\bullet)
\end{aligned}
\end{equation*}where $n=i+l$. 

We may view $M^\bullet$ as a double complex $M^{\bullet\bullet}$ with 
\begin{equation*}
    M^{i,l}=\begin{cases}
      \bigoplus_{p+q=i} M^p_q \quad &\text{if $l=0$}\\
      0\quad &\text{otherwise}
    \end{cases}
\end{equation*}The horizontal differentials are given by 
\begin{equation*}
    d^{i,l}=\begin{cases}
    \bigoplus_{p+q=i} (d_N^p)_q\quad & \text{if $l=0$}\\
    0\quad & \text{otherwise}
    \end{cases}
\end{equation*}and the vertical differentials are zero. The total complex is 
\begin{equation*}
 \textnormal{Tot}(M^{\bullet\bullet})^n=\bigoplus_{i+l=n}M^{i,j}=\bigoplus_{p+q=n}M^p_q
\end{equation*}We note that, since $M^\bullet\in\textcat{C}^{\downarrow}(\textcat{grA-Mod})$, then $M^{\bullet\bullet}$, and hence $I(M^{\bullet\bullet})$, is a third quadrant double complex. We construct another spectral sequence $\{{}^I\overline{E}_r^{i,l}\}$ for $r\geq 0$, with differentials ${}^I\overline{d}_r^{i,l}$, such that 
\begin{equation*}
    {}^I\overline{E}_0^{i,l}=I(M)^{i,l}
\end{equation*}and the maps ${}^I\overline{d}_0^{i,l}$ are just the corresponding differentials of $M^{\bullet\bullet}$. We have
\begin{equation*}
    {}^I\overline{E}_1^{i,l}=H_v^l(I(M^{i,\bullet}))=\begin{cases}\bigoplus_{p+q=i} I(M^p_q) \quad &\text{if $l=0$}\\
    0\quad &\text{otherwise}
    \end{cases}
\end{equation*}We have the convergence
\begin{equation*}
    {}^I\overline{E}_2^{i,l}\Rightarrow{H^{i+l}(\textnormal{Tot}(I(M)^{\bullet\bullet})^\bullet)}=H^{n}(I(M^\bullet))
\end{equation*}
The map $\rho$ induces a morphism of double complexes. Hence, we see that, since $\{{}^I\overline{E}_r^{i,l}\}$ and $\{{}^I{E}_r^{i,l}\}$ have the same first page and the differentials on the first page are the same, these spectral sequences must be equal for $r\geq 1$. Hence, 
\begin{equation*}
    H^{n}(I(M^{\bullet}))=H^{n}(I((\textbf{C}G\circ \textbf{C}F)(M^\bullet)^\bullet))
\end{equation*}and therefore $I(\rho_{N^\bullet})$ is a quasi-isomorphism. It follows that $\rho$ is a strict quasi-isomorphism, and hence so is $\eta$.

\end{proof}

\begin{proof}[Proof of Theorem \ref{bigmaintheorem} ] We have seen, by Propositions \ref{prop1} and \ref{prop2}, that there exist functors $\textcat{C}F$ and $\textcat{C}G$ inducing functors
\begin{equation*}
    \textcat{D}F: \textcat{D}^\downarrow(\textcat{grA-Mod})\rightarrow{ \textcat{D}^\uparrow({\textcat{grA}}^!\textcat{-Mod})   }\,\text{ and }\,\textcat{D}G: \textcat{D}^\uparrow(\textcat{grA}^!\textcat{-Mod})\rightarrow{ \textcat{D}^\downarrow({\textcat{grA-Mod})   }}
\end{equation*}The functors $\textcat{C}F$ and $\textcat{C}G$ are adjoint, and the unit and counit coming from this adjunction induce natural isomorphisms
\begin{equation*}
    \textcat{D}F\circ \textcat{D}G\rightarrow{id_{\textcat{D}^\uparrow(\textcat{grA}^!\textcat{-Mod})}}\quad\text{and}\quad id_{\textcat{D}^\downarrow(\textcat{grA-Mod})}\rightarrow{\textcat{D}G\circ \textcat{D}F}
\end{equation*}Hence, there is an equivalence of categories
\begin{equation*}
    \textcat{D}^\downarrow(\textcat{grA-Mod})\simeq{ \textcat{D}^\uparrow({\textcat{grA}}^!\textcat{-Mod})  }
\end{equation*}

\end{proof}

\begin{appendices}

\label{appendix}

\section{Banach Spaces}\label{banachspaces}

\begin{defn}
A \textbf{valued field} $k$ is a field, $k$, equipped with a multiplicative norm 
\begin{equation*}
    |-|:k\rightarrow{\mathbb{R}_{\geq 0}}
\end{equation*}such that, for all $\lambda,\mu\in k$,
\begin{itemize}
    \item $|\lambda+\mu|\leq |\lambda|+|\mu|$
    \item $|\lambda|=0$ if and only if $\lambda=0$
    \item $k$ is complete with respect to the metric defined by the norm.
\end{itemize}
\end{defn}
 
We say that $k$ is \textbf{non-Archimedean} if $|\lambda+\mu|\leq max\{|\lambda|,|\mu|\}$ and is \textbf{Archimedean} otherwise.

\begin{defn}
A \textbf{Banach space} $V$ is a $k$-vector space $V$ equipped with a norm 
\begin{equation*}
    ||-||:V\rightarrow{\mathbb{R}_{\geq 0}}
\end{equation*}such that, for all $\lambda\in k$ and $v,w\in V$, we have 
\begin{itemize}
    \item $||\lambda v||=|\lambda|\cdot||v||$,
    \item $||v+w||\leq ||v||+||w||$, 
    \item $V$ is a complete metric space with respect to the metric defined by the norm.
\end{itemize}
\end{defn}
We say that $V$ is \textbf{non-Archimedean} if $||v+w||\leq max\{||v||,||w||\}$ and is \textbf{Archimedean} otherwise. 
\begin{defn}
A linear map $f:V\rightarrow{W}$ between Banach spaces is \textbf{bounded} if there exists a constant $C>0$ such that, for all $v\in V$,
\begin{equation*}
    ||f(v)||_W\leq C||v||_V
\end{equation*}
\end{defn}

We denote the category of Banach spaces, equipped with bounded linear maps as its morphisms, by $\textcat{Ban}_k$. When a distinction is necessary, we will denote the category of Archimedean (resp. non-Archimedean) Banach spaces over an Archimedean (resp. non-Archimedean) field $k$ by $\textcat{Ban}^A_k$ (resp. $\textcat{Ban}^{nA}_k$).

\begin{defn}
A \textbf{Banach ring} is a unital commutative ring $R$ equipped with a norm
\begin{equation*}
    |-|:R\rightarrow{\mathbb{R}_{\geq 0}}
\end{equation*}such that, for all $r,s\in R$,
\begin{itemize}
    \item $|r|=0$ precisely if $r=0$, 
    \item $|r+s|\leq |r|+|s|$,
    \item $|rs|\leq |r|\cdot|s|$,
    \item $R$ is a complete metric space with respect to the metric defined by the norm.
    \end{itemize}
\end{defn}

We say $R$ is \textbf{non-Archimedean} if $|r+s|\leq max\{|r|,|s|\}$ and is \textbf{Archimedean} otherwise.

\begin{defn}
Suppose that $R$ is a Banach ring. A \textbf{Banach $R$-module} is an $R$-module $M$ together with a map \begin{equation*}||-||:M\rightarrow{\mathbb{R}_{\geq 0}}\end{equation*} such that, for all $r\in R$, and for all $m,n\in M$,
\begin{itemize}
    \item $||m||=0$ if and only if $m=0$, 
    \item $||m+n||\leq ||m||+||n||$, 
    \item $||rm||\leq |r|\cdot||m||$, 
    \item $M$ is complete with respect to the metric defined by $||-||$.
\end{itemize}

\end{defn}

\begin{defn}\label{boundedbanachmap}
A homomorphism $f:M\rightarrow{N}$ between Banach $R$-modules is \textbf{bounded} if there exists a constant $C>0$ such that, for all $m\in M$,
\begin{equation*}
    ||f(m)||_N\leq C||m||_M
\end{equation*}
\end{defn}

The category of Banach modules over an Archimedean (resp. non-Archimedean) Banach ring $R$ will be denoted $Ban_R^A$ (resp. $Ban_R^{nA}$). The morphisms are bounded $R$-linear maps.

\section{Ind-Objects}\label{inductiveappendix}

Suppose that we are working within a category $\mathcal{C}$. In $\mathcal{C}$ we will denote limits by $\varprojlim$ and colimits by $\varinjlim$.
\begin{defn}\cite[Definition 2.1]{bambozzibenbassat15} If $I$ is a small filtered category, then an \textbf{Ind-object} of $\mathcal{C}$ is a functor $X:I\rightarrow{\mathcal{C}}$. This is denoted by $X=(X_i)_{i\in I}$. 
\end{defn}

 We denote by $\textcat{Ind}\mathcal{C}$ the collection of Ind-objects $X=(X_i)_{i\in I}$ with morphisms defined as follows.

\begin{defn}
Suppose that we have two Ind-objects $X,Y$. Then, the collection of \textbf{morphisms of inductive systems} $\textnormal{Hom}_{\textcat{Ind}\mathcal{C}}(X,Y)$ is given by  
\begin{equation*}
    \textnormal{Hom}_{\textcat{Ind}\mathcal{C}}(X,Y)\simeq \varprojlim_{i\in I}\varinjlim_{j\in J}\textnormal{Hom}(X_i,Y_j)
\end{equation*}
\end{defn}

We consider the Yoneda embedding $Y:\mathcal{C}\rightarrow{\textcat{PSh}(\mathcal{C})}$. We denote by $``\underset{i\in I}{\varinjlim}" X_i$ the presheaf
\begin{equation*}
    ``\varinjlim_{i\in I}" X_i=\varinjlim_{i\in I} Y(X_i)
\end{equation*}We note that the map
\begin{equation*}
    X\rightarrow{``\varinjlim_{i\in I}" X_i}
\end{equation*} induces an equivalence of categories between $\textcat{Ind}\mathcal{C}$ and the full subcategory of $\textcat{PSh}(\mathcal{C})$ consisting of presheaves isomorphic to filtrant inductive limits of representable functors. We will interchange between the notation $X$ and $``\underset{i\in I}{\varinjlim}" X_i$ frequently.

\begin{defn}
An object $X\in\textcat{Ind}\mathcal{C}$ is \textbf{essentially monomorphic} if all the morphisms in the corresponding inductive system $``\underset{i\in I}{\varinjlim}" X_i$ are monomorphic.
\end{defn}

\section{Bornological Spaces}\label{bornologyappendix}

Bornological spaces possess enough structure to consider questions of boundedness, and thus are an ideal setting for bringing together homological algebra and functional analysis. We recall the following details, for more see \cite{bambozzibenbassat15} or \cite{meyer07}.

\begin{defn}
Let $X$ be a set. A \textbf{bornology} on $X$ is a collection $\mathcal{B}$ of subsets of $X$ such that 
\begin{itemize}
    \item $\mathcal{B}$ covers $X$, i.e.  for every $x\in X$ there exists some $B\in \mathcal{B}$ such that $x\in \mathcal{B}$, 
    \item $\mathcal{B}$ is stable under inclusions, i.e. for every inclusion $A\subset B\in \mathcal{B}$, we have $A\in\mathcal{B}$,
    \item $\mathcal{B}$ is stable under finite unions, i.e. for each $n\in\mathbb{N}$ and $B_1,\dots,B_n\in\mathcal{B}$, we have $\cup_{i=1}^nB_i\in\mathcal{B}$.
\end{itemize}
\end{defn}
The pair $(X,\mathcal{B})$ is called a \textit{bornological set}, and the elements of $\mathcal{B}$ are called bounded subsets of $X$. A family of subsets $A\subset\mathcal{B}$ is called a \textit{basis} for $\mathcal{B}$ if, for any $B\in\mathcal{B}$, there exist $A_1,\dots,A_n\in A$ such that $B\subset A_1\cup \dots\cup A_n$. A \textit{morphism of bornological sets} is any map which sends bounded subsets to bounded subsets. We remark that the category of bornological sets is complete and cocomplete.

Suppose that we have a complete non-trivially valued field $k$.
\begin{defn}
A \textbf{bornological vector space} over $k$ is a $k$-vector space $V$ with a bornology on the underlying set of $V$ such that the maps $(\lambda,v)\rightarrow{\lambda v}$ and $(v,w)\rightarrow{v+w}$ are bounded. 
\end{defn}

We will now detail certain categories of bornological vector spaces. We let $k^\circ=\{\lambda\in k\mid |\lambda|\leq 1\} $. A subset $W$ of a vector space $V$ is \textit{convex} if, for every  $v,w\in W$ and $t\in [0,1]$, we have that $(1-t)v+tw\in W$, and $W$ is \textit{balanced} if for every $\lambda\in k^\circ$, $\lambda W\subset W$.

\begin{defn}
Let $V$ be a $k$-vector space. A subset $W\subset V$ is called \textbf{absolutely convex (or a disk)} if
\begin{itemize}
    \item for $k$ Archimedean, $W$ is convex and balanced. 
    \item for $k$ non-Archimedean, $W$ is a $k^\circ$-submodule of $V$.
\end{itemize}
\end{defn}

\begin{defn}
A bornological vector space is said to be of \textbf{convex type} if it has a basis made up of absolutely convex subsets. We denote the category of bornological $k$-vector spaces of convex type, equipped with bounded linear maps, by $\textcat{Born}_k$.
\end{defn}

\begin{defn}
A bornological vector space over $k$ is said to be \textbf{separated} if its only bounded vector subspace is the trivial subspace $\{0\}$. We denote the full subcategory of $\textcat{Born}_k$ consisting of separated bornological $k$-vector spaces by $\textcat{SBorn}_k$.
\end{defn}

\begin{defn}
A separated bornological $k$-vector space $V$ is said to be \textbf{complete} if there exists a small filtered category $I$, a functor $I\rightarrow{\textcat{Ban}}_k$ and an isomorphism 
\begin{equation*}
    V\simeq \varinjlim_{i\in I}V_i
\end{equation*}for a filtered colimit of Banach spaces over $k$ for which the system morphisms are all injective and the colimit is calculated in $\textcat{Born}_k$. We denote the full subcategory of $\textcat{Born}_k$ consisting of complete bornological $k$-vector spaces by $\textcat{CBorn}_k$. 
\end{defn}

\begin{prop}\label{dissectionfunc}\cite[Proposition 3.60]{bambozzibenbassat15}There is a functor, the dissection functor, \begin{equation*}\begin{aligned}
    \textnormal{diss}:\textcat{CBorn}_k&\rightarrow{\textcat{IndBan}_k}\\
    V&\rightarrow{``\varinjlim\limits_{i\in I}"V_i}
\end{aligned}\end{equation*} which defines an equivalence of $\textcat{CBorn}_k$ with the full subcategory of essentially monomorphic objects of $\textcat{IndBan}_k$.
\end{prop}

\end{appendices}

\bibliographystyle{plain}
\bibliography{references}

\end{document}